\documentclass[11pt,reqno]{amsart}
\usepackage{fullpage}
\usepackage{amsmath,amsthm,verbatim,amssymb,amsfonts,amscd, graphicx,bbm,bm}
\usepackage{graphics}
\usepackage{mathtools}
\usepackage{mathrsfs}
\usepackage{esint}
\usepackage{color}
\usepackage{tikz,tikz-cd}
\usepackage[all,cmtip]{xy}
\usetikzlibrary{decorations.pathreplacing,arrows.meta,calc,patterns,angles,quotes,decorations.markings}

\usepackage[pdfstartview=FitH, bookmarksnumbered=true,bookmarksopen=true, colorlinks=true, pdfborder=001, citecolor=blue, linkcolor=blue,urlcolor=blue]{hyperref}

\newtheorem{theorem}{Theorem}
\newtheorem{proposition}[theorem]{Proposition}
\newtheorem{lemma}[theorem]{Lemma}
\newtheorem{corollary}[theorem]{Corollary}

\theoremstyle{definition}
\newtheorem{remark}{Remark}

\newcommand{\cref}[1]{Corollary~\ref{c.#1}}

\numberwithin{equation}{section}
\numberwithin{theorem}{section}

\newcommand{\R}{\mathbb{R}}
\newcommand{\N}{\mathbb{N}}
\newcommand{\ol}{\overline}
\newcommand{\eps}{\varepsilon}

\newcommand{\bT}{\mathbb{T}}
\newcommand{\bC}{\mathbb{C}}

\newcommand{\Z}{\mathbb{Z}}

\newcommand{\wtcL}{\widetilde{\mathcal{L}_\delta}}
\newcommand{\wtcA}{\widetilde{\mathcal{A}}}

\newcommand{\bTtrace}{\mathcal{T}}
\newcommand{\bEext}{\mathcal{E}}

\newcommand{\mi}{\mathrm{i}}

\title{Wave packets propagation in the subwavelength regime near the Dirac point}

\author{Habib Ammari}
\address[H. Ammari]{Department of Mathematics, ETH Zürich, Zürich, Switzerland}
\email{habib.ammari@math.ethz.ch}

\author{Xin Fu}
\address[X. Fu]{Yau Mathematical Sciences Center, Tsinghua University, Beijing 100084, P.R. China}
\email{fux20@mails.tsinghua.edu.cn}

\author{Wenjia Jing}
\address[W. Jing]{Yau Mathematical Sciences Center, Tsinghua University, Beijing 100084 and Beijing Institute of Mathematical Sciences and Applications, Beijing 101408, P.R. China}
\email{wjjing@tsinghua.edu.cn}

\date{\today}

\begin{document}
    
	\begin{abstract}
		In [Ammari et al., SIAM J Math Anal., 52 (2020),
  pp. 5441--5466], the first author with collaborators proved the existence of Dirac dispersion cones at subwavelength scales in bubbly honeycomb phononic crystals. In this paper, we study the time-evolution of wave packets, which are spectrally concentrated near such conical points. We prove that the wave packets dynamics is governed by a time-dependent effective Dirac system, which still depends, but in a simple way, on the subwavelength scale. 

\bigskip
		
		\noindent{\bf Key words}: Dirac degeneracy, bubbly honeycomb phononic crystals, time-dependent Dirac equation, evolution of wave packets.

		\bigskip
		
		\noindent{\bf Mathematics subject classification (MSC 2020)}: 
35B27, 35J05, 35C20.

\end{abstract}

  \maketitle

\section{Introduction}

Since the successful fabrication of graphene, a monolayer of carbon atoms arranged in a honeycomb lattice \cite{RevModPhys.81.109}, conical dispersions, originally proposed for relativistic particles based on the Dirac equation, have attracted much interest. Conical dispersions lead to unusual transport phenomena. They are not limited to electronic structures. They have been observed in artificial structures (also called metamaterials) for classical waves. 

Metamaterials with Dirac singularities have been extensively studied both experimentally and numerically; see, \emph{e.g.},  \cite{PhysRevLett.108.174301,scientificreport}. Subwavelength resonators are the building blocks of metamaterials. They are typically high-contrast material inclusions.
In acoustics, due to its low density and low bulk modulus, 
a gas bubble in a liquid is known to have a resonance frequency corresponding to wavelengths which are several orders of magnitude larger than the bubble \cite{MR3906861,minnaert}. This opens up the possibility of creating small-scaled acoustic metamaterials, whereby the operating frequency corresponds to wavelengths much larger than the device size. 

The simplicity of the gas bubble makes bubbly media an ideal model for subwavelength metamaterials. After suitable simplification we consider the operator $\mathcal{L}_\delta u = -\sigma_\delta(x)\nabla \cdot\left(\sigma_\delta^{-1}(x)\nabla u\right)$ where $\sigma_\delta(x)$ has value $\delta$ in the bubbles and has value $1$ in the liquid. The property of honeycomb lattice of bubbly acoustic structures where $\sigma_\delta$ satisfies honeycomb lattice symmetry is largely determined by the Floquet-Bloch dispersion relation of $\mathcal{L}_\delta$. Proofs of the existence of Dirac cones, that is conical singularities occurring at the intersections of neighboring dispersion surfaces, and mathematical analyses of their properties are provided in \cite{ammari_honeycomb-lattice_2020,ammari_high-frequency_2020}. More precisely, in this high-contrast setting, the Dirac cone structure can be seen around high-symmetry quasi-momenta $\alpha^*$ given by \eqref{eq:cone} where conical intersection occurs for the first two dispersion surfaces. In terms of the eigenvalues of $\mathcal{L}_\delta$ with $\alpha$-quasi-periodic conditions, those sub-wavelength bands behave like
\begin{equation}
\label{eq:Dirac_omega}
\omega^2_\pm (\alpha) = \omega^2(\alpha^*) \pm \delta \mathrm{c} |\alpha - \alpha^*| (1+\mathcal{O}(|\alpha-\alpha^*|))
\end{equation}
around $\alpha^*$, where $\mathrm{c}$ is a constant independent of the contrast material parameter $\delta$; see \cite{ammari_honeycomb-lattice_2020}. Moreover, those surfaces are low-lying (with amplitude of order $\mathcal{O}(\sqrt{\delta})$) and correspond to propagation modes of large wave-length (or, in other words, the bubbles are subwavelength resonators for those modes). The eigenvalue $\omega^2(\alpha^*)$ at the Dirac point $\alpha^*$ has a degeneracy of order two, and one can find two independent eigen-modes (Dirac modes) satisfying proper symmetry relations. The expansion of certain eigen-modes of the first two bands for quasi-momenta $\alpha$ near $\alpha^*$ in terms of the Dirac modes is also studied in detail in \cite{ammari_high-frequency_2020}. In Section \ref{sec:DiracStructure} below we present more  details of those results. 

In this paper, we study the propagation of wave packets formed by a superposition of the degenerate eigenmodes at the Dirac point $\alpha^*$,  and prove in Theorem \ref{thm:main} that the wave packets evolve as, approximately and valid for a rather long time, a superposition of the Dirac modes with time-dependent profiles governed by Dirac equations. The approximation hence serves as an effective and much simplified model of the wave propagation. Due to the high-contrast nature of the bubbly acoustic media, however, the effective Dirac equation still depends on the subwavelength scale (or, the high-contrast scale) $\delta$. 

The paper is organized as follows. In Section \ref{sec:formulation} we describe the  bubbly honeycomb phononic crystal and state our main results in this paper. 
Section \ref{sec:prelim} is devoted to some preliminary results on the Floquet-Bloch theory and representation formulas for the Bloch eigenmodes. 
In Section \ref{sec:DiracStructure} we study the subwavelength band structure of the honeycomb crystal and recall the results proved in \cite{ammari_honeycomb-lattice_2020,ammari_high-frequency_2020} on the existence of Dirac degeneracies between the two first band functions and the behavior of the corresponding Bloch modes near the Dirac points. In Section \ref{sec:proofs} we outline the proofs of our main results on the dynamics of wave packets formed by those degenerate eigenmodes. In Section \ref{sec:i3} we finish the proof of the main results in this paper. The appendix  is devoted to derive some new estimates for wave equations in the high-contrast regime used in Section \ref{sec:i3}. 

\section{Problem formulation and main results} \label{sec:formulation}

\subsection{Triangular lattice,  symmetric inclusions and honeycomb structure}

We first describe the geometric setup of the bubbly honeycomb structure. Consider the equilateral triangular lattice in $\mathbb{R}^2$ defined by 
\begin{equation*}
	\Lambda = \mathbb{Z}l_1 \oplus \mathbb{Z}l_2 = \{ m_1 l_1 + m_2 l_2: (m_1, m_2) \in \mathbb{Z}^2 \},
\end{equation*} 
which is generated by the lattice vectors
\begin{equation*}
l_1 = L \left(
	\frac{\sqrt{3}}{2},
	\frac{1}{2}
	\right), \quad \quad l_2 = L\left(
	\frac{\sqrt{3}}{2},
	-\frac{1}{2}
	\right).
\end{equation*}
The lattice constant $L$ is chosen later. This is a two dimensional hexagonal Bravais lattice. The dual lattice of $\Lambda$ is denoted by $\Lambda^*$, and is generated by the lattice vectors
\begin{equation*}
	\alpha_1 = \frac{2\pi}{L} \left( \frac{\sqrt{3}}{3}, 1\right) ,\quad \quad 	\alpha_2 = \frac{2\pi}{L} \left( \frac{\sqrt{3}}{3}, -1\right).
\end{equation*}
They satisfy $\alpha_i \cdot l_j = 2\pi \delta_{ij}$, for $i, j = 1, 2$, where $\delta_{ij}$ denotes the Kronecker symbol. Let $Y$ and $Y^*$ be the fundamental domains of the lattices $\Lambda$ and $\Lambda^*$ which are defined, respectively, as
\begin{equation*}
\begin{aligned}
	&Y= \{ s l_1 + tl_2: 0 \leq s, t\leq 1\},\\
	&Y^* = \{ s\alpha_1 + t\alpha_2 :0\leq s, t\leq 1 \}.
	\end{aligned}
\end{equation*}
The lattice constant $L$ is chosen so that $|Y^*|=1$. We also denote the central Wigner-Seitz cell of the lattices $\Lambda$ by $Y_{\rm B}$, which consists of points in $\R^2$ that are closer to $(0,0) \in \Lambda$ than any other points of the lattice $\Lambda$. Similarly, the central Wigner-Seitz cell of the dual lattice $\Lambda^*$ is denoted by $Y^*_{\rm B}$ or simply by $\mathcal{B}$, and it is also known as the first Brillouin zone. See Figure \ref{fig:honeycomb1} for illustrations. In this article we also let $\bT$ denote the torus $\R^2/\Lambda$ and denote by $\bT^*$ the dual torus $\mathbb{R}^2 / \Lambda^*$.

The unit cell $Y$ of $\Lambda$ is divided into two equilateral triangular domains $Y_1$ and $Y_2$:
\begin{equation*}
    Y_1 = \{s l_1 + tl_2 \,:\, 0\le s,t, \, \text{and} \, t+s\le 1\}, \quad Y_2 = \{s l_1 + t l_2 \,:\, 0\le s,t \le 1, \,\text{and}\, 1\le t+s\}.
\end{equation*}
Let the centers of $Y$, $Y_1$ and $Y_2$ be denoted by $x_0, x_1$ and $x_2$, respectively. They are given by
\begin{equation*}
	x_0 = \frac{l_1  + l_2 }{2} , \quad x_1 = \frac{l_1  + l_2 }{3}, \quad x_2 = \frac{2( l_1  + l_2 )}{3}.
\end{equation*}

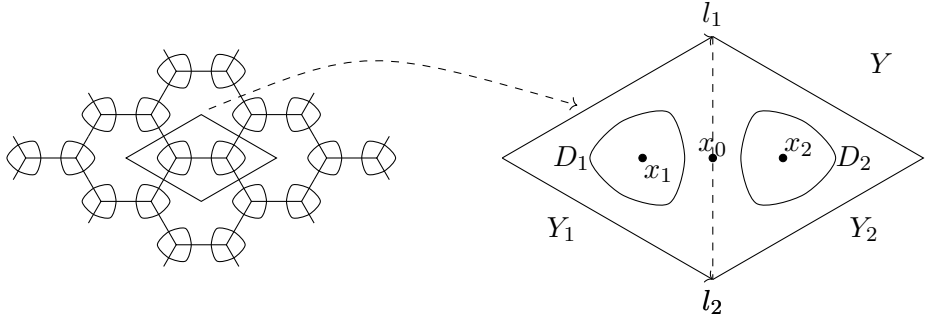
\begin{figure}[!h]
	\centering
	\begin{tikzpicture}
	\begin{scope}[xshift=-5cm,scale=1]
	\coordinate (a) at (1,{1/sqrt(3)});		
	\coordinate (b) at (1,{-1/sqrt(3)});	
	\pgfmathsetmacro{\rb}{0.25pt}
	\pgfmathsetmacro{\rs}{0.2pt}
	
	\draw (0,0) -- (1,{1/sqrt(3)}) -- (2,0) -- (1,{-1/sqrt(3)}) -- cycle; 
	\begin{scope}[xshift = 1.33333cm]
	\draw plot [smooth cycle] coordinates {(0:\rb) (60:\rs) (120:\rb) (180:\rs) (240:\rb) (300:\rs) };
	\end{scope}
	\begin{scope}[xshift = 0.666667cm, rotate=60]
	\draw plot [smooth cycle] coordinates {(0:\rb) (60:\rs) (120:\rb) (180:\rs) (240:\rb) (300:\rs) };
	\end{scope}
	
	\draw[opacity=0.2] ({2/3},0) -- ({4/3},0)
		($0.5*(1,{1/sqrt(3)})$) -- ({2/3},0)
		($0.5*(1,{-1/sqrt(3)})$) -- ({2/3},0)
		($(1,{1/sqrt(3)})+0.5*(1,{-1/sqrt(3)})$) -- ({4/3},0)
		($0.5*(1,{1/sqrt(3)})+(1,{-1/sqrt(3)})$) -- ({4/3},0);

\begin{scope}[shift = (a)]
	\begin{scope}[xshift = 1.33333cm]
	\draw plot [smooth cycle] coordinates {(0:\rb) (60:\rs) (120:\rb) (180:\rs) (240:\rb) (300:\rs) };
	\end{scope}
	\begin{scope}[xshift = 0.666667cm, rotate=60]
	\draw plot [smooth cycle] coordinates {(0:\rb) (60:\rs) (120:\rb) (180:\rs) (240:\rb) (300:\rs) };
	\end{scope}	
	\draw[opacity=0.2] ({2/3},0) -- ({4/3},0)
		($0.5*(1,{1/sqrt(3)})$) -- ({2/3},0)
		($0.5*(1,{-1/sqrt(3)})$) -- ({2/3},0)
		($(1,{1/sqrt(3)})+0.5*(1,{-1/sqrt(3)})$) -- ({4/3},0)
		($0.5*(1,{1/sqrt(3)})+(1,{-1/sqrt(3)})$) -- ({4/3},0);
\end{scope}
\begin{scope}[shift = (b)]
	\begin{scope}[xshift = 1.33333cm]
	\draw plot [smooth cycle] coordinates {(0:\rb) (60:\rs) (120:\rb) (180:\rs) (240:\rb) (300:\rs) };
	\end{scope}
	\begin{scope}[xshift = 0.666667cm, rotate=60]
	\draw plot [smooth cycle] coordinates {(0:\rb) (60:\rs) (120:\rb) (180:\rs) (240:\rb) (300:\rs) };
	\end{scope}
	\draw[opacity=0.2] ({2/3},0) -- ({4/3},0)
		($0.5*(1,{1/sqrt(3)})$) -- ({2/3},0)
		($0.5*(1,{-1/sqrt(3)})$) -- ({2/3},0)
		($(1,{1/sqrt(3)})+0.5*(1,{-1/sqrt(3)})$) -- ({4/3},0)
		($0.5*(1,{1/sqrt(3)})+(1,{-1/sqrt(3)})$) -- ({4/3},0);
\end{scope}
\begin{scope}[shift = ($-1*(a)$)]
	\begin{scope}[xshift = 1.33333cm]
	\draw plot [smooth cycle] coordinates {(0:\rb) (60:\rs) (120:\rb) (180:\rs) (240:\rb) (300:\rs) };
	\end{scope}
	\begin{scope}[xshift = 0.666667cm, rotate=60]
	\draw plot [smooth cycle] coordinates {(0:\rb) (60:\rs) (120:\rb) (180:\rs) (240:\rb) (300:\rs) };
	\end{scope}
	\draw[opacity=0.2] ({2/3},0) -- ({4/3},0)
	($0.5*(1,{1/sqrt(3)})$) -- ({2/3},0)
	($0.5*(1,{-1/sqrt(3)})$) -- ({2/3},0)
	($(1,{1/sqrt(3)})+0.5*(1,{-1/sqrt(3)})$) -- ({4/3},0)
	($0.5*(1,{1/sqrt(3)})+(1,{-1/sqrt(3)})$) -- ({4/3},0);
	\end{scope}
	\begin{scope}[shift = ($-1*(b)$)]
	\begin{scope}[xshift = 1.33333cm]
	\draw plot [smooth cycle] coordinates {(0:\rb) (60:\rs) (120:\rb) (180:\rs) (240:\rb) (300:\rs) };
	\end{scope}
	\begin{scope}[xshift = 0.666667cm, rotate=60]
	\draw plot [smooth cycle] coordinates {(0:\rb) (60:\rs) (120:\rb) (180:\rs) (240:\rb) (300:\rs) };
	\end{scope}
	\draw[opacity=0.2] ({2/3},0) -- ({4/3},0)
	($0.5*(1,{1/sqrt(3)})$) -- ({2/3},0)
	($0.5*(1,{-1/sqrt(3)})$) -- ({2/3},0)
	($(1,{1/sqrt(3)})+0.5*(1,{-1/sqrt(3)})$) -- ({4/3},0)
	($0.5*(1,{1/sqrt(3)})+(1,{-1/sqrt(3)})$) -- ({4/3},0);
	\end{scope}
\begin{scope}[shift = ($(a)+(b)$)]
	\begin{scope}[xshift = 1.33333cm]
	\draw plot [smooth cycle] coordinates {(0:\rb) (60:\rs) (120:\rb) (180:\rs) (240:\rb) (300:\rs) };
	\end{scope}
	\begin{scope}[xshift = 0.666667cm, rotate=60]
	\draw plot [smooth cycle] coordinates {(0:\rb) (60:\rs) (120:\rb) (180:\rs) (240:\rb) (300:\rs) };
	\end{scope}
	\draw[opacity=0.2] ({2/3},0) -- ({4/3},0)
	($0.5*(1,{1/sqrt(3)})$) -- ({2/3},0)
	($0.5*(1,{-1/sqrt(3)})$) -- ({2/3},0)
	($(1,{1/sqrt(3)})+0.5*(1,{-1/sqrt(3)})$) -- ({4/3},0)
	($0.5*(1,{1/sqrt(3)})+(1,{-1/sqrt(3)})$) -- ({4/3},0);
	\end{scope}
	\begin{scope}[shift = ($-1*(a)-(b)$)]
	\begin{scope}[xshift = 1.33333cm]
	\draw plot [smooth cycle] coordinates {(0:\rb) (60:\rs) (120:\rb) (180:\rs) (240:\rb) (300:\rs) };
	\end{scope}
	\begin{scope}[xshift = 0.666667cm, rotate=60]
	\draw plot [smooth cycle] coordinates {(0:\rb) (60:\rs) (120:\rb) (180:\rs) (240:\rb) (300:\rs) };
	\end{scope}
	\draw[opacity=0.2] ({2/3},0) -- ({4/3},0)
	($0.5*(1,{1/sqrt(3)})$) -- ({2/3},0)
	($0.5*(1,{-1/sqrt(3)})$) -- ({2/3},0)
	($(1,{1/sqrt(3)})+0.5*(1,{-1/sqrt(3)})$) -- ({4/3},0)
	($0.5*(1,{1/sqrt(3)})+(1,{-1/sqrt(3)})$) -- ({4/3},0);
	\end{scope}
\begin{scope}[shift = ($(a)-(b)$)]
	\begin{scope}[xshift = 1.33333cm]
	\draw plot [smooth cycle] coordinates {(0:\rb) (60:\rs) (120:\rb) (180:\rs) (240:\rb) (300:\rs) };
	\end{scope}
	\begin{scope}[xshift = 0.666667cm, rotate=60]
	\draw plot [smooth cycle] coordinates {(0:\rb) (60:\rs) (120:\rb) (180:\rs) (240:\rb) (300:\rs) };
	\end{scope}
	\draw[opacity=0.2] ({2/3},0) -- ({4/3},0)
	($0.5*(1,{1/sqrt(3)})$) -- ({2/3},0)
	($0.5*(1,{-1/sqrt(3)})$) -- ({2/3},0)
	($(1,{1/sqrt(3)})+0.5*(1,{-1/sqrt(3)})$) -- ({4/3},0)
	($0.5*(1,{1/sqrt(3)})+(1,{-1/sqrt(3)})$) -- ({4/3},0);
	\end{scope}
	\begin{scope}[shift = ($-1*(a)+(b)$)]
	\begin{scope}[xshift = 1.33333cm]
	\draw plot [smooth cycle] coordinates {(0:\rb) (60:\rs) (120:\rb) (180:\rs) (240:\rb) (300:\rs) };
	\end{scope}
	\begin{scope}[xshift = 0.666667cm, rotate=60]
	\draw plot [smooth cycle] coordinates {(0:\rb) (60:\rs) (120:\rb) (180:\rs) (240:\rb) (300:\rs) };
	\end{scope}
	\draw[opacity=0.2] ({2/3},0) -- ({4/3},0)
	($0.5*(1,{1/sqrt(3)})$) -- ({2/3},0)
	($0.5*(1,{-1/sqrt(3)})$) -- ({2/3},0)
	($(1,{1/sqrt(3)})+0.5*(1,{-1/sqrt(3)})$) -- ({4/3},0)
	($0.5*(1,{1/sqrt(3)})+(1,{-1/sqrt(3)})$) -- ({4/3},0);
\end{scope}
\end{scope}

\draw[dashed,opacity=0.5,->] (-3.9,0.65) .. controls(-1.8,1.5) .. (1,0.7);
	\begin{scope}[scale=2.8]	
	\coordinate (a) at (1,{1/sqrt(3)});		
	\coordinate (b) at (1,{-1/sqrt(3)});	
	\coordinate (Y) at (1.8,0.45);
	\coordinate (c) at (2,0);
	\coordinate (x1) at ({2/3},0);
	\coordinate (x0) at (1,0);
	\coordinate (x2) at ({4/3},0);

	\pgfmathsetmacro{\rb}{0.25pt}
	\pgfmathsetmacro{\rs}{0.2pt}
	
	\begin{scope}[xshift = 1.33333cm]
	\draw plot [smooth cycle] coordinates {(0:\rb) (60:\rs) (120:\rb) (180:\rs) (240:\rb) (300:\rs) };
	\draw (0:\rb) node[xshift=7pt] {$D_2$};
	\end{scope}
	\begin{scope}[xshift = 0.666667cm, rotate=60]
	\draw plot [smooth cycle] coordinates {(0:\rb) (60:\rs) (120:\rb) (180:\rs) (240:\rb) (300:\rs) };
	\end{scope}
	\draw ({0.6666667-\rb},0) node[xshift=-7pt] {$D_1$};
	
	\draw (Y) node{$Y$};
	\draw[->] (0,0) -- (a) node[above]{$l_1$};
	\draw[->] (0,0) -- (b) node[pos=0.4,below left]{$Y_1$} node[below]{$l_2$};
	\draw[dashed] (a) -- (b) node[below]{$l_2$};
	\draw (a) -- (c) -- (b) node[pos=0.4,below right]{$Y_2$};
	\draw[fill] (x1) circle(0.5pt) node[xshift=6pt,yshift=-6pt]{$x_1$}; 
	\draw[fill] (x0) circle(0.5pt) node[yshift=4pt]{$x_0$}; 
	\draw[fill] (x2) circle(0.5pt) node[xshift=6pt,yshift=4pt]{$x_2$}; 
	\end{scope}
\end{tikzpicture}
\vspace{10pt}
\caption{Honeycomb crystal.}\label{fig:honeycomb1}
	\end{figure}

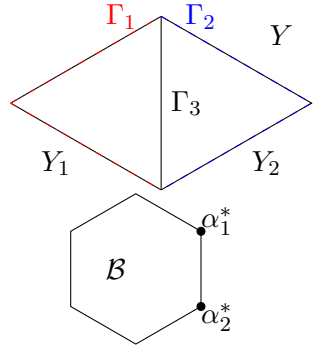
\begin{figure}[!h]
	\centering
	\begin{tikzpicture}[scale=2]
	\coordinate (a) at (1,{1/sqrt(3)});		
	\coordinate (b) at (1,{-1/sqrt(3)});	
	\coordinate (Y) at (1.8,0.45);
	\coordinate (c) at (2,0);
	\coordinate (x1) at ({2/3},0);
	\coordinate (x0) at (1,0);
	\coordinate (x2) at ({4/3},0);
	
	\draw (Y) node{$Y$};
	\draw[opacity=0.2] (0,0) -- (a);
	\draw[opacity=0.2] (0,0) -- (b);
	\draw[opacity=0.2] (a) -- (c) -- (b);
	
	\draw[red,dash dot] (0,0) -- (a);
	\draw[red,dash dot] (0,0) -- (b);
	\draw (a) -- (b) node[right,pos=0.5]{$\Gamma_3$};
	\draw[blue,dashed] (a) -- (c) -- (b);
	\draw (0.3,-0.4) node{$Y_1$};
	\draw (1.7,-0.4) node{$Y_2$};
	\draw[blue] (a) node[xshift=15pt]{$\Gamma_2$};
	\draw[red] (a) node[xshift=-15pt]{$\Gamma_1$};
	\end{tikzpicture}
 
\begin{tikzpicture}[scale=1.5]
	\coordinate (a) at ({1/sqrt(3)},1);
	\coordinate (b) at ({1/sqrt(3)},-1);
	\coordinate (c) at ({2/sqrt(3)},0);
	\coordinate (K1) at ({1/sqrt(3)},{1/3});
	\coordinate (K2) at ({1/sqrt(3)},{-1/3});
	\coordinate (K3) at (0,{-2/3});
	\coordinate (K4) at ({-1/sqrt(3)},{-1/3});
	\coordinate (K5) at ({-1/sqrt(3)},{1/3});
	\coordinate (K6) at (0,{2/3});

	\draw[fill] (K1) circle(1pt) node[xshift=6pt,yshift=4pt]{$\alpha_1^*$};
	\draw[fill] (K2) circle(1pt) node[xshift=6pt,yshift=-4pt]{$\alpha_2^*$};
	\draw[fill] (0,0)   node[left]{$\mathcal{B}$};

	\draw (K1) -- (K2) -- (K3) -- (K4) -- (K5)  -- (K6) -- cycle;
\end{tikzpicture}
\vspace{10pt}
\caption{Fundamental domain $Y$ and corresponding Brillouin zone.}\label{fig:honeycomb2}
\end{figure}

The bubbly honeycomb crystal consists of high-contrast acoustic inclusions (bubbles) embedded in an otherwise homogeneous background matrix (water). Moreover, the set occupied by the inclusions is $\Lambda$-periodic and satisfies certain symmetries that give the honeycomb structure. Let $R: \R^2\to \R^2$ be the clockwise rotation by $\frac{2\pi}{3}$ around the origin, and define
\begin{equation}
\label{eq:R12def}
	R_1 x = R x +l_1 , \quad R_2 x = R x +2 l_1, \quad R_0 x = 2x_0 - x.
\end{equation}
Let $e_j$, $j=1,2$, be the unit vectors of the Euclidean basis of $\R^2$. One easily checks that $l_1-Le_2=l_2$, $Rl_1=-Le_2$ and $Rl_2=-l_1$. Therefore, $R_0$ is the clockwise rotation around $x_0$ by $\pi$, and $R_j$ is the
clockwise rotation  around $x_j$ by $\frac{2\pi}{3}$, for $j=1,2$. Let $R_3$ be the reflection across the line $x_0 + \mathbb{R} e_2$, \emph{i.e.}, reflection across $\Gamma_3 = \partial Y_1\cap \partial Y_2$; see Figure \ref{fig:honeycomb2}.

The acoustic inclusions set $D$ inside the unit cell $Y$ is then modeled as the union  $D_j$, $j = 1,2 $, each an open and simply connected set with Lipschitz boundary centered at $x_j$. Moreover, $D_1$ and $D_2$ satisfy the following symmetric relations: 
\begin{equation}
\label{eq:hcsym}
	R_1 D_1 =D_1, \quad R_2 D_2 =D_2, \quad R_0 D_1 =
 R_3 D_1 = D_2.
\end{equation}
That is, the set $D=D_1\cup D_2$ is centrally symmetric with respect to the center $x_0$ and mirror symmetric with respect to $\Gamma_3$, and each inclusion $D_j$ is $2\pi/3$-rotation symmetric with respect to its own center. The background set in the unit cell is then $\Omega = Y \setminus \overline{D}$. By $\Lambda$-periodic extension of $D$, the inclusions array is defined by
\begin{equation*}
	D^{\Lambda} = \bigcup_{\mathbf{n} \in \Lambda} (D+\mathbf{n} ).
\end{equation*}
Similarly, the whole background material occupies the region
\begin{equation*}
	\Omega^{\Lambda} = \mathbb{R}^2 \setminus \overline{ D^{\Lambda} }.
\end{equation*}
Due to the symmetric relations imposed above, the centers of all inclusions $\{D_j + \mathbf{n} \,:\, \mathbf{n}\in \Lambda\}$ form the set 
\begin{equation*}
	\mathbb{H} = (x_1 + \Lambda) \cup ( x_2 + \Lambda),
\end{equation*}
which has a honeycomb structure; see Figure \ref{fig:honeycomb1}.

In view of the honeycomb structure, the two points
\begin{equation*}
	\alpha_1^* = \frac{2\alpha_1 + \alpha_2}{3} , \quad \quad 	\alpha_2^* = \frac{\alpha_1 +2 \alpha_2}{3}
\end{equation*}
in the Brillouin zone $Y^*$ possess high symmetry, see Figure \ref{fig:honeycomb2}. Moreover, it is proved in \cite{ammari_honeycomb-lattice_2020} that those points are the Dirac points where the first two low-lying dispersion relation curves caused by the resonant bubbles meet conically. 

\subsection*{Notations} 

In the rest of the paper, we let $\bm{\alpha}$ denote the 2-tuple $(\alpha_1,\alpha_2)$ and let $\bm{l}$ denote the 2-tuple $(l_1,l_2)$. Given $\mathbf{m} = (m_1,m_2) \in \Z^2$, $\mathbf{m}\bm{\alpha} \in \Lambda^*$ and $\mathbf{m}\bm{l} \in \Lambda$ denote the following products:
\begin{equation*}
\mathbf{m}\bm{\alpha} := m_1\alpha_1 + m_2\alpha_2, \quad \mathbf{m}\bm{l} := m_1 l_1 + m_2 l_2.
\end{equation*}
In this paper, all functions are assumed to take values in $\bC$ or $\bC^2$ and hence we no longer write the target spaces in the notations of functional classes.

\subsection{Bubbly honeycomb crystal and the associated differential operator}

We assume that the whole plane $\mathbb{R}^2$ is realized as an acoustic composite, equipped with the symmetric inclusions $D^{\Lambda}$ and honeycomb structure $\mathbb{H}$ described in the last section. Let the inclusions $D^{\Lambda}$ are occupied by air, and the background $\Omega^{\Lambda}$ be occupied by liquids. Basically speaking, the plane $\mathbb{R}^2$ becomes a composite by air bubbles and liquids.

For acoustic materials,  \textit{compressibility} and  \textit{density} are two essential parameters to describe the acoustic properties. 
Let the compressibility and the density of air bubbles are denoted by $\kappa_b$ and $\rho_b$, respectively. The corresponding parameters of the surrounding liquid are denoted by $\kappa_0$ and $\rho_0$, respectively. Let
\begin{equation*}
	v_0 := \sqrt{\frac{\kappa_0}{\rho_0}}, \quad v_b := \sqrt{\frac{\kappa_b}{\rho_b}}, \quad \delta := \frac{\rho_b}{\rho}.
\end{equation*}
Then $v_0$ and $v_b$ are the sound speeds in the background material and in the inclusions, respectively; $\delta$ is the density ratio of the inclusion compared with the background. For air bubbles in fluid, for example, it is natural to assume $v_0,v_b = \mathcal{O}(1)$ and $\delta \ll 1$. 
For simplicity and to fix ideas, after proper rescaling we assume $\kappa_0 =\rho_0 =1$, and $\kappa_b =\rho_b = \delta$. Then the differential operator that governs the acoustic behaviors of the material is simplified to
\begin{equation}
\label{eq:Ldeldef}
	\mathcal{L}_{\delta} = - \sigma_{\delta}\nabla \cdot ( \sigma_{\delta}^{-1} \nabla),
\end{equation}
where the material parameter $\sigma_{\delta}$ is defined by
\begin{equation*}
	\sigma_{\delta} (x)= \left\{
	\begin{aligned}
	& 1 , && x \in \Omega^{\Lambda}, \\
	& \delta, && x \in D^{\Lambda}. 
\end{aligned}\right.
\end{equation*}

We remark that making the bulk modulus contrast and the density contrast equal (the sound speeds $v_b$ and $v_0$ in and outside the bubbles then coincide), as done above, simplifies the expression of the operator $\mathcal{L}_\delta$. Nevertheless, the method used in Section \ref{sec:DiracStructure} for the subwavelength band structures works for the general setting where $v_0 \ne v_b$. In addition, our method in Sections \ref{sec:proofs} and \ref{sec:i3} to derive the approximation of wave dynamics work for more general setting, \emph{i.e.}, for the operator
\begin{equation*}
	\mathcal{L}_\delta = -\kappa(x)\nabla\cdot\left(\frac{1}{\rho(x)}\nabla \cdot\right),
\end{equation*}
where $\kappa(x) = \kappa_0 \mathbf{1}_{\Omega^\Lambda} + \kappa_b \mathbf{1}_{D^\Lambda}$, and $\rho(x)=\rho_0\mathbf{1}_{\Omega^\Lambda} + \rho_b \mathbf{1}_{D^\Lambda}$.

\subsection{Rescaled bubbly honeycomb crystal and the main problem} 

To model wave propagation in bubbly acoustic media formed by tiny air bubbles in water, we consider the rescaled bubbly honeycomb crystal by replacing the lattice constant $L$ with $\varepsilon L$, where $\varepsilon>0$ is the microscopic parameter which is assumed to be very small.

Thus the rescaled differential operator, denoted by $\mathcal{L}_{\varepsilon,\delta}$, is given by
\begin{equation*}
	\mathcal{L}_{\varepsilon,\delta} =  - \sigma_{\delta}(\cdot/\varepsilon) \nabla \cdot \left( \sigma_{\delta}(\cdot/\varepsilon)^{-1}   \nabla \right) .
\end{equation*}
We study the following wave equation that models the propagation of wave packets spectrally concentrated near the Dirac point:  
\begin{equation}\label{problem 1}
	\left\{
	\begin{aligned}
		& \left( \partial_t^2 + \mathcal{L}_{\varepsilon,\delta} \right) w_{\varepsilon,\delta} = 0  \quad \mathrm{in} \ \mathbb{R}^2\times (0,\infty), \\
		& w_{\varepsilon,\delta} (x,0) =  F_1(x) S_1\left(\frac{x}{\varepsilon}\right) +F_2(x) S_2\left(\frac{x}{\varepsilon}\right) , \\
		&\partial_t w_{\varepsilon,\delta} (x,0) =\mathrm{i} \frac{\omega_{\delta}^*}{\varepsilon} \left( F_1(x) S_1\left(\frac{x}{\varepsilon}\right) + F_2(x) S_2\left(\frac{x}{\varepsilon}\right) \right).
	\end{aligned}\right.
\end{equation}
Here, $\omega_{\delta}^*$ is the Bloch frequency at a Dirac point $\alpha^*$, see \eqref{eq:omegastar}; note that $\omega_\delta^* = \mathcal{O}(\sqrt{\delta})$. The functions $S_1$ and $ S_2$ are given by single layer potentials of some densities (see \eqref{eq:Sjdef}) and they are the quasi-stationary approximations of the Bloch 
eigenmodes at the Dirac point $\alpha^*$ associated to the frequency $\omega_\delta^*$. The true Bloch eigenmodes at $\alpha^*$ are denoted by $S_{1,\delta}$ and $S_{2,\delta}$; see Section \ref{sec:DiracStructure}. The initial data are hence wave packets spectrally concentrated near the Dirac point, with $F_1, F_2\in \mathscr{S}(\mathbb{R}^2)$ modeling the envelopes of the wave packets. Here and below, $\mathscr{S}(\R^2)$ denotes the Schwartz class of smooth functions decaying rapidly at infinity.  

\begin{remark}
	In general, the initial wave packets spectrally centered at the Dirac point $\alpha^*$ should have the form:
	\begin{equation}\label{full initial condition}
		\left\{
		\begin{aligned}
			& w_{\varepsilon,\delta} (x,0) =  F_1(x) S_1\left(\frac{x}{\varepsilon}\right) +F_2(x) S_2\left(\frac{x}{\varepsilon}\right) , \\
			&\partial_t w_{\varepsilon,\delta} (x,0) =G_1(x) S_1\left(\frac{x}{\varepsilon}\right) +G_2(x) S_2\left(\frac{x}{\varepsilon}\right) ,
		\end{aligned}\right.
	\end{equation}
where $F_j, G_j$, $j=1,2$, are all in the Schwartz space $\mathscr{S}(\mathbb{R}^2)$. Thanks to the linearity of the wave operator, the general data can always be decomposed into two parts evolving independently, one as in \eqref{problem 1} and the other of a similar form with $\mi$ replaced by $-\mi$ in the expression of $\partial_t w_{\eps,\delta}(\cdot,0)$. We focus on \eqref{problem 1} as the other part can be treated in exactly the same way.
\end{remark}


The main objective of this paper is to derive a simplified model for $w_{\eps,\delta}$ that is effective in the large scale (for $\eps$ small) and valid for a rather long time (for $t$ up to scale of $\eps^{-s_1}$ with some $s_1 > 0$). Since the initial data have oscillations at the microscopic scale $\eps$, the problem falls in the regime of high frequency wave propagations. Such problems in periodic structure was formulated and studied in the seminal work \cite{bensoussan_asymptotic_1978}. We mention several closely related previous works that study the effective dynamics of wave packets propagations, in the context of Schr\"odinger equations and wave equations, where the initial wave packets are spectrally concentrated at a quasimomentum $\widetilde \alpha$. More precisely, the initial data is the product of the (rescaled) eigenmode $S(\cdot/\eps)$ associated to the eigenvalue $\lambda(\widetilde\alpha)$ with a slowly varying smooth profile function. It turns out that the effective dynamics depends crucially on the behavior of the dispersion surface $\alpha \mapsto \lambda(\alpha)$ at the point $\widetilde \alpha$. For example, one can distinguish the following settings:

\begin{itemize}
	\item[(a)] When $\widetilde\alpha$ is a regular non-stationary point of $\lambda$, i.e., $\lambda(\widetilde\alpha)$ is a simple Bloch eigenvalue and $\mathbf{v} := \nabla_\alpha \lambda\rvert_{\alpha=\widetilde\alpha} \neq 0$, the problem is referred as the \emph{Ballistic propagation problem}. The wave packets evolves as a modulation of the oscillatory eigenmode with profile function governed by a first order transport equation with conduction velocity proportional to $\mathbf{v}$. See, e.g., \cite[section 4.1]{allaire_homogenization_2005} for the Schr\"odinger equation.
	\item[(b)] Suppose that there is a simple Bloch eigenvalue at $\widetilde\alpha$ and has zero gradient but sign-definite Hessian $\nabla^2_\alpha \lambda$ at $\widetilde \alpha$.  $\lambda(\widetilde\alpha)$ is then located at the edge of a band. One can further distinguish two cases:
	\begin{itemize}
		\item[(b.1)] The edge is at the bottom of all bands. Either one has $\lambda(\widetilde\alpha)=0$ and the Bloch eigenmode at this point is constant so we are back in the low frequency wave propagation regime, or $\lambda(\widetilde\alpha)>0$ and the ground state is positive everywhere and we can factorize the (rescaled) ground state out. 
  In both cases the effective dynamics is given by the  \emph{classical homogenization} theory; see, \emph{e.g.}, \cite{bensoussan_asymptotic_1978, conca_homogenization_1997} and \cite[Theorem 1.1]{allaire_homogenization_2005}. 
	\item[(b.2)] The edge is at a higher energy band. In the set-up of Schr\"odinger equations, this is the main  setting studied in \cite{allaire_homogenization_2005} and the authors there obtained an \emph{Effective mass theory} for the wave propagation: the packets evolve with time-dependent profiles governed by a ``homogenized" equation of the same type where the effective diffusion matrix is proportional to the Hessian matrix of $\lambda$ at $\widetilde\alpha$.  
	\end{itemize}
	\item[(c)] Suppose $\lambda(\widetilde\alpha)$ is a simple Bloch eigenvalue and $\widetilde\alpha$ is a third order stationary point, \emph{i.e.}, $\nabla^4_\alpha \lambda$ does not vanish at $\widetilde\alpha$ but $\nabla^k_\alpha \lambda$ vanishes at $\widetilde\alpha$ for $k=1,2,3$.
The problem in the set-up of Schr\"odinger equation with proper rescaling in time was studied in  \cite{allaire_homogenization_2004, allaire_homogenization_2005}, and the authors there named it \emph{Fourth-order homogenization}. The effective profile of the wave packets satisfies an evolution equation involving a fourth order operator $\nabla\cdot \nabla \cdot (A_{\mathrm{hom}} \nabla \nabla)$, where the homogenized coefficient $A_{\mathrm{hom}}$ is proportional to $\nabla^4_{\alpha} \lambda\rvert_{\alpha=\widetilde\alpha}$.
\end{itemize}  
In this paper, and in the previous works \cite{fefferman_wave_2014} (for the Schr\"odinger equation with honeycomb lattice potential) and \cite{MR4264639} (for wave operator $\partial_t^2 u -\nabla\cdot(W(x)\nabla u)$ with honeycomb lattice material weight $W$), the initial data are spectrally localized at a Dirac point $\widetilde\alpha = \alpha_*$ where two dispersion surfaces intersect conically. Let $\lambda_\pm$ denote the two branches, then $\lambda_+(\alpha_*)=\lambda_-(\alpha_*)$ has multiplicity $2$ and $\lambda_\pm$ have conical singularity at $\alpha_*$. It was proved in \cite{fefferman_wave_2014,MR4264639} that wave packets formed by modulated superposition of the Dirac modes  still evolve as wave packets of such type and the time-dependent profiles vary slowly in time and are governed by a Dirac system. Finally, we also mention that the case of smooth touching of two or more dispersion relations at a degenerate point $\widetilde \alpha$ was also studied in \cite{allaire_homogenization_2005}. The superposition of wave packets formed by the multiple eigenmodes then evolve with profiles governed by system of effective mass equations possibly coupled only through zero-order terms. Needless to say, the phenomena are very different from the case of Dirac cone intersections.   

The problem set-up and the main result (Theorem \ref{thm:main} below) of this paper are most similar to those of \cite{fefferman_wave_2014,MR4264639}. However, there are several key differences. First, there is another small parameter, the high contrast scale $\delta$, in the problem. It results two low-lying dispersion surfaces for  $\mathcal{L}_\delta$ and for sufficiently small $\delta$ there are Dirac cones at the high-symmetry points $\alpha_*$. The microscopic units of the honeycomb structure hence have \emph{subwavelength} scale compared to propagation modes of those branches. Such a structure allows waves emitted in the near field to be transmitted while preserving their phase \cite{ammari_honeycomb-lattice_2020}. 
Secondly, the operator $\mathcal{L}_\delta$ is not of divergence form, and it becomes self-adjoint only in weighted spaces (by considering the weighted measure $\sigma_\delta^{-1}dx$). Thirdly, $\mathcal{L}_\delta$ has piecewise continuous (constant) coefficients and we are out of the smooth settings of \cite{fefferman_wave_2014,MR4264639}. It is fair to say that our proof strategy in Sections \ref{sec:proofs} and \ref{sec:i3} follows that of \cite{fefferman_wave_2014}, which is based on the Floquet-Bloch state representation of the solution to wave equation; see also \cite{MR4264639}. However, the $\delta$-dependence of the band structure of $\mathcal{L}_\delta$ and the discontinuity of the coefficients in $\mathcal{L}_\delta$ cause significant difficulties that we overcome by a careful formulation of Floquet-Bloch theory in Section \ref{sec:FBstates} for this setting and by some useful study of $\mathcal{L}_\delta$-wave equations in the Appendix.  


\subsection{Main results}

We seek the solution $w_{\varepsilon,\delta} $ of \eqref{problem 1} in the form
\begin{equation}
\label{ansatz}
	w_{\varepsilon,\delta}(x,t) =  e^{\mathrm{i} \frac{\omega_{\delta}^*}{\varepsilon} t} \left( V_1(x, t)S_1\left(\frac{x}{\varepsilon} \right) + V_2(x, t)S_2\left(\frac{x}{\varepsilon} \right)  \right) + r_{\eps,\delta}(x,t).
\end{equation}
The first part on the right-hand side is an approximation of the solution $w_{\eps,\delta}$ and it is still a combination of the wave packets with involving envelopes $(V_1,V_2)$ to be determined. The term $r_{\eps,\delta}$ accounts for the approximation error. Our main goal is to find the dynamics of $V_1,V_2$, so that $r_{\eps,\delta}$ is under control. Our main theorem is the following.
\begin{theorem}\label{thm:main}
	Assume that $w_{\varepsilon,\delta}$ is the solution of problem \eqref{problem 1}, and $(V_1,V_2)$ is the solution of the following homogenized problem:
	\begin{equation}\label{eq:Dirac}
		2\mi \omega^*_\delta\partial_t \begin{pmatrix}
			V_1\\ V_2
		\end{pmatrix} = \begin{pmatrix}
			 &a_{\delta}   (\partial_1 + \mathrm{i} \partial_2)  \\
			- \overline{a_{\delta}}(\partial_1 -\mathrm{i}\partial_2)& 
		\end{pmatrix}  \begin{pmatrix}
			V_1\\ V_2
		\end{pmatrix}, 
	\end{equation}
with initial data $(V_1,V_2)(x,0) = (F_1,F_2)(x)$, where $a_{\delta}$ satisfies \eqref{eq:adelta}. Assume that $\bm{F} = (F_1,F_2)$ are Schwartz functions, and let $\delta = \varepsilon^{\kappa}$, where $\kappa \in (0,2)$. Then for any $\nu \in (0,1)$ and positive integer $N\ge 2$ we can find a constant $C_{\bm{F}}>0$ depending only on $\bm{F}$, $N$ and other data of the honeycomb structure, such that the error term in the ansatz \eqref{ansatz} satisfies
\begin{equation*}
\begin{aligned}
	\|r_{\eps,\delta}(\cdot,t)\|_{L^2(\R^2)} &\le C_{\bm{F}}\Big(\eps^{\frac{\kappa}{2}} + \eps^{1-\frac{\kappa}{2}} + t\eps^{\min\{N+\nu-1-\kappa, \nu-1+\frac{\kappa}{2}, 2\nu-1-\frac{\kappa}{2}\}} \\
&\quad + t^2\varepsilon^{\min\{ \nu-2+\frac{3\kappa}{2}, 2\nu-2+\frac{\kappa}{2}, (1-\nu)N-2\}}\Big).
\end{aligned}
\end{equation*}
In particular, for $\kappa\in (\frac{2}{3},2)$, by choosing $\nu \in (\frac{1}{2},1)$ sufficiently close to $1$ and then a positive integer $N$ sufficiently large, we can find positive numbers $s_1$ and $s_2$ such that 
\begin{equation*}
\max_{0\le t \le \eps^{-s_1}} \| r_{\varepsilon,\delta}(x,t) \|_{L^2(\mathbb{R}^2_x)} \le C_{\bm{F}} \eps^{s_2}.
\end{equation*}
\end{theorem}

The theorem above basically says that, for a time interval that is quite long, the dynamics of initial wave packets spectrally centered at the Dirac point can be well approximated by the main term of \eqref{ansatz}, that is wave packets of the same type with time-dependent profiles evolving by the Dirac system \eqref{eq:Dirac}.

\begin{remark}
Some remarks about the Dirac system \eqref{eq:Dirac} and about the approximation \eqref{ansatz} are in order. We can rewrite the Dirac equation \eqref{eq:Dirac} as 
\begin{equation*}
	\mi \partial_t \begin{pmatrix} V_1\\V_2 \end{pmatrix}
	= \left[\Theta_\delta \sigma_1 D_1 - \Theta_\delta \sigma_2 D_2\right] \begin{pmatrix} V_1 \\ V_2 \end{pmatrix},
\end{equation*}
where the matrices in the front of $D_1 := \frac{1}{\mi} \partial_{x_1}$ and $D_2 := \frac{1}{\mi}  \partial_{x_2}$ are Hermitian matrices obtained by multiplication with the Pauli matrices $\sigma_1,\sigma_2$. More precisely, with $\eta_{\#,\delta} = \frac{\mi a_\delta}{2\omega^*_\delta}$, the matrices involved are
\begin{equation*}
	\Theta_\delta = \begin{pmatrix}
		\eta_{\#,\delta} & \\ & \overline{\eta_{\#,\delta}}
	\end{pmatrix}, \quad
\sigma_1 = \begin{pmatrix}
		 & 1\\ 1 
	\end{pmatrix}, \quad
\sigma_2 = \begin{pmatrix}
		 & -\mi\\ \mi 
	\end{pmatrix}.
\end{equation*}
Note that the modified matrices satisfy
\begin{equation*}
	(\Theta_\delta \sigma_j)^2 = |\eta_{\#,\delta}|^2 I_2, \, \text{for $j=1,2$,}\; \text{and} \; (\Theta_\delta \sigma_1) (\Theta_\delta \sigma_2) + (\Theta_\delta \sigma_2)(\Theta_\delta \sigma_1) = 0.  
\end{equation*}
As a result, both $V_1$ and $V_2$ satisfy the wave equation
\begin{equation}
\label{eq:waveV}
\partial^2_t V_j = \left|\eta_{\#,\delta}\right|^2 \Delta V_j, \quad j=1,2.
\end{equation}
Since $\omega_{\delta}^* = \sqrt{\frac{  c_1^{\alpha^*} }{|D_1|}} \delta^{\frac{1}{2}} + \mathcal{O}(\delta)$ and $a_{\delta} =  \frac{\mathrm{i}c}{|D_1|} \delta+ \mathcal{O}(\delta^{\frac{3}{2}})$, the wave speed $|\eta_{\#,\delta}|$ above is of order $\sqrt{\delta}$.

We observe that for any multi-index $\mathbf{m}=(m_0,m_1,m_2)\in \N^3$, the pair $(\partial^{\mathbf{m}}_{t,x_1,x_2} V_1, \partial^{\mathbf{m}}_{t,x_1,x_2} V_2)$ satisfy the same Dirac equation as in \eqref{eq:Dirac}.

\medskip

Note we can further write $V_j(x,t)$ as $\widetilde V_j(x,|\eta_{\#,\delta}|t)$, $j=1,2$, so that  $(\widetilde{V}_1,\widetilde{V}_2)$ solve a Dirac system with coefficients of unit size. Since $|\eta_{\#,\delta}| = \mathcal{O}(\sqrt{\delta})$ and $\omega_\delta^*/\eps = \mathcal{O}(\sqrt{\delta}/{\eps})$, we see that the approximation \eqref{ansatz} remains a superposition of the (quasi-stationary) Dirac modes $S_1$ and $S_2$ with profiles slowly varying in space and time and with phase rapidly varying in time.  
\end{remark}

\section{Preliminaries} \label{sec:prelim}

\subsection{Floquet-Bloch theory for the operator $\mathcal{L}_{\delta}$} 
\label{sec:FB_L}

We briefly review the Floquet-Bloch theory for the operator $\mathcal{L}_{\delta}$ acting on $L^2(\R^2)$. Due to the function $\sigma_{\delta}$ outside the divergence in \eqref{eq:Ldeldef}, it is natural to consider the following weighted $L^2$ space and the associated Sobolev spaces.

Let $L_{\delta}^2(\mathbb{R}^2) $ be the Hilbert space $L^2(\mathbb{R}^2) $ with weighted inner product
\begin{equation*}
	\left\langle u, v \right\rangle_{L^2_{\delta}} = \int_{\mathbb{R}^2}   \sigma^{-1}_{\delta} u(x)\overline{v(x)} \,dx , \quad \quad \mathrm{for} \ \mathrm{any} \ u,v \in L^2(\mathbb{R}^2) ,
\end{equation*}
and denote $\|\cdot \|_{L^2_{\delta}(\mathbb{R}^2)}$ the associated norm. 
For any $\alpha\in Y^*$, let $L^2_{\delta,\Lambda^\alpha}$ be the Hilbert space of $\alpha$-quasiperiodic functions:
\begin{equation}
\label{eq:L2delalp}
L^2_{\delta,\Lambda^\alpha} = \{f \in L^2_{\rm loc}(\R^2) \;:\; f(\cdot +v)=e^{\mi \alpha\cdot v}f(\cdot) \, \text{for all $v\in \Lambda$}\},
\end{equation}
and let $L^2_{\delta,\Lambda^\alpha}$ be equipped with the inner product 
\begin{equation}
\label{eq:ipYdel}
	\left\langle f, g \right\rangle_{L^2_{\delta,\Lambda^\alpha}} = \int_Y \sigma^{-1}_{\delta} f(x)\overline{g(x)} \,dx.
\end{equation} 
We can also view any $f\in L^2_{\delta,\Lambda^\alpha}$ as defined on the torus $\bT = \R^2/\Lambda$ since $f$ is determined by its value on the unit cell $Y$. Hence, we also denote the inner product above simply by $\langle\cdot,\cdot\rangle_{L^2_\delta(Y)}$ and the corresponding norm by $\|\cdot\|_{L^2_\delta(Y)}$. For $\alpha = 0$, we simply write $L^2_{\delta,\Lambda}$ for $L^2_{\delta,\Lambda^0}$ and it is the space of locally square-integrable $\Lambda$-periodic functions. 

In the same manner, we also introduce the weighted Sobolev space $H^1_\delta(\R^2)$ which is the usual space $H^1(\R^2)$ equipped with the weighted inner product
\begin{equation}
\label{def: H^1 R2}	
\langle u, v \rangle_{H^1_{\delta}} = \int_{\mathbb{R}^2}  \sigma_{\delta}^{-1} \big( u(x)\overline{v(x)} +  \nabla u(x) \cdot \overline{\nabla v(x)} \big) \,dx , \quad \quad \mathrm{for} \ \mathrm{all} \ u,v \in H^1(\mathbb{R}^2).
\end{equation}
Let $\|\cdot \|_{H^1_{\delta}(\mathbb{R}^2)}$ denote the associated norm. For each $\alpha \in Y^*$, let $H^1_{\delta,\Lambda^\alpha} = H^1_{\delta,\Lambda^\alpha}(\R^2/\Lambda)$ denote the Hilbert space
\begin{equation}
\label{eq:H1delalp}
H^1_{\delta,\Lambda^\alpha} = \{f\in H^1_{\rm loc} \,:\, f(\cdot +v) = e^{\mi \alpha\cdot v}f(\cdot) \, \, \text{ for all $v\in \Lambda$}\}
\end{equation}
equipped with the weighted inner product
\begin{equation}
\langle u,v\rangle_{H^1_{\delta,\Lambda^\alpha}} = \int_Y \frac{1}{\sigma_\delta(x)} \left( f(x)\ol{g(x)} + \nabla f(x)\cdot \overline{\nabla g(x)}\right)\,dx.
\end{equation}
We also denote this inner product simply by $\langle\cdot,\cdot\rangle_{H^1_\delta(Y)}$, and by $\|\cdot\|_{H^1_{\delta}(Y)}$ be the associated norm. $H_{\delta}^{-1} (\mathbb{R}^2)$ denotes the dual space of $H_{\delta}^1(\mathbb{R}^2)$ with the dual-norm.

We now introduce the Floquet-Bloch transform. First for Schwartz space $\mathscr{S}(\mathbb{R}^2)$, the transform is given by
\begin{equation}
\label{eq:FBT-S}
	\mathcal{U} f (x,\alpha) =\sum_{l \in \Lambda} e^{\mi\alpha \cdot l} f(x-l), \quad \mathrm{for} \ x \in \R^2, \ \alpha \in \R^2.
\end{equation}
Clearly, for each $\alpha \in \R^2$, $\mathcal{U}f(\cdot,\alpha)$ satisfies
\begin{equation*}
\mathcal{U}f(\cdot+l,\alpha) = e^{\mi\alpha\cdot l} \mathcal{U}f(\cdot,\alpha), \quad \mathcal{U}f(x,\cdot+ \gamma) = \mathcal{U}f(x,\cdot), \quad \forall\, l\in \Lambda, \forall\, \gamma \in \Lambda^*. 
\end{equation*}
The first property above is referred to as \emph{$\alpha$-quasiperiodic}. The second property shows $\mathcal{U}f(x;\cdot)$ is $\Lambda^*$-periodic and hence we only need to consider $\alpha \in Y^*$. Hence, the Floquet transform maps $\mathscr{S}(\R^2)$ to the space
\begin{equation*}
\begin{aligned}
\mathcal{W}(\R^2 \times \R^2) &:= \bigg\{ \phi \in C^\infty(\R^2_x\times \R^2_\alpha;\bC) \,:\, \forall\,x\in \R^2, \,\alpha \in \R^2, \,j=1,2, \, \\
&\qquad\qquad \phi(x+l_j,\alpha) = e^{\mathrm{i}\alpha\cdot l_j}\phi(x,\alpha) \text{ and }\, \phi(x,\alpha+\alpha_j) = \phi(x,\alpha) \bigg\}.
\end{aligned}
\end{equation*}
In fact, the mapping $\mathcal{U}: \mathscr{S}(\R^2)\to \mathcal{W}(\R^2\times \R^2)$ is an isomorphism, and the inverse operator is
\begin{equation}\label{inversion}
	\mathcal{U}^{-1} \phi (x) = \int_{Y^*} \phi (x,\alpha)\,d\alpha, \qquad \quad \mathrm{for}\ \phi \in 	\mathcal{W}(\R^2\times \R^2).
\end{equation}
Indeed, it is easy to check that $\mathcal{U}^{-1}\mathcal{U} = \mathrm{id}$ holds on $\mathscr{S}(\R^2)$, so injectivity holds. To check surjectivity, it suffices to show that $\mathcal{U}^{-1}$ maps $\mathcal{W}$ to $\mathscr{S}(\R^2)$ because given this the identity $\mathcal{U}\mathcal{U}^{-1} = \mathrm{id}$ holds on $\mathcal{W}$. Fix an arbitrary $\phi \in \mathcal{W}(\R^2_x\times \R^2_\alpha)$, one show that for any fixed multi-index $\gamma, \beta \in \N^2$, $x^\gamma \partial^\beta \mathcal{U}^{-1}\phi$ is uniformly bounded over $x \in \R^2$. We refer to \cite{MR1232660,AmmFep} for the details.

On the other hand, one easily checks that the following Plancherel-type identity holds:
\begin{equation}\label{plancherel}
 \| f \|^2_{L^2_{\delta}(\mathbb{R}^2)}
= \int_{Y^*}  \| \mathcal{U} f (\cdot,\alpha) \|^2_{L^2_{\delta}(Y)}  \,d\alpha .
\end{equation}
The Plancherel identity extends $\mathcal{U}$ as an isometry from $L_{\delta}^2(\mathbb{R}^2) $ to the space $L^2(\bT^*;L^2_\delta(Y))$, the space of $\Lambda^*$-periodic and $L^2_\delta(Y)$-valued functions. In fact, it is more revealing to view the range of $L_{\delta}^2(\R^2)$ under $\mathcal{U}$ as $\int_{Y^*}^\oplus L^2_\delta(\alpha)\,d\alpha$ --- the direct integral of the Hilbert spaces $L^2_\delta(\alpha)$, where each space can be thought as $L^2_{\delta,\Lambda^\alpha}$ (restricted to the unit cell $Y$). See the next subsection.
\medskip

Now consider the operator $\mathcal{L}_\delta$ defined by \eqref{eq:Ldeldef}. More precisely, It is defined by
\begin{equation}
\label{eq:cLL2}
\begin{aligned}
&\mathcal{L}_\delta u = -\sigma_\delta \nabla \cdot \left(\sigma_\delta^{-1}\nabla u\right),\\
	&\mathrm{Dom}\,( 	\mathcal{L}_{\delta} )  = \left\{  u \in L^2_\delta(\mathbb{R}^2) :  - \sigma_{\delta} \nabla \cdot \left( \sigma_{\delta}^{-1} \nabla u \right) \in  L_{\delta}^2(\mathbb{R}^2)   \right\} \subset L^2_\delta(\R^2).
 \end{aligned}
\end{equation}
In view of the weight in the inner product of $L^2_\delta$, we see that $\mathcal{L}_\delta$ is a non-negative self-adjoint operator. Note that $\mathcal{L}_\delta$ acts on functions defined in non-compact spatial space. 
Via the Floquet-Bloch theory, it is known that the operator $\mathcal{L}_{\delta}$ is unitarily equivalent to the direct integral operator $\int^{\oplus}_{Y^*} \mathcal{L}_{\delta}(\alpha)\,d\alpha$ via conjugation by $\mathcal{U}$; that is,
\begin{equation}\label{fl-b}
		\mathcal{L}_{\delta} = \mathcal{U}^{-1} \left( \int^{\oplus}_{Y^*} \mathcal{L}_{\delta}(\alpha)\,d\alpha \right)\mathcal{U}, \quad \mathrm{meaning}\ \  \left(\mathcal{U} \mathcal{L}_{\delta} u\right) (\cdot,\alpha)= \mathcal{L}_{\delta}(\alpha ) \left(\mathcal{U}u(\cdot,\alpha)\right)\ \forall\, \alpha \in Y^*.
\end{equation}
Here for each $\alpha \in Y^*$, $\mathcal{L}_\delta(\alpha)$ denotes the operator $-\sigma_\delta \nabla \cdot(\sigma^{-1}_\delta \nabla)$ with domain
\begin{equation}
\label{eq:cLL2alpha}
	\mathrm{Dom}\,( 	\mathcal{L}_{\delta}(\alpha) )  = \left\{  u \in H^1_{\delta,\Lambda^\alpha} \;:\;  - \sigma_{\delta} \nabla \cdot \left( \sigma_{\delta}^{-1} \nabla \right) \in  L_{\delta}^2(\alpha)   \right\}.
\end{equation}
We emphasize that although $\mathcal{L}_\delta(\alpha)$ have the same expression \eqref{eq:Ldeldef} for all $\alpha \in Y^*$, the domain of those operators vary with respect to $\alpha$. 

We summarize the above discussions into the following commutative diagram:
	\begin{displaymath}
	\xymatrix{
		{L_{\delta}^2(\mathbb{R}^2)} \ar[rr]^{\mathcal{L}_{\delta}} \ar[d]^{\mathcal{U}}
		&& {L_{\delta}^2(\mathbb{R}^2)} \ar[d]^{\mathcal{U}} \\ 
		{\int_{Y^*}^{\oplus} L_{\delta}^2(\alpha)\,d\alpha} \ar[rr]^{\int_{Y^*}^{\oplus} \mathcal{L}_{\delta}(\alpha)\,d\alpha}
		&& {\int_{Y^*}^{\oplus} L_{\delta}^2(\alpha)\,d\alpha} 
	}
\end{displaymath}

\subsection{Bloch eigenmodes and representation formulas.} 
\label{sec:FBstates}

For each $\alpha \in Y^*$, the operator $\mathcal{L}_{\delta}(\alpha)$ is self-adjoint and non-negative, and it has compact resolvents. In fact, by standard spectral theory, its spectrum consists of a sequence of non-negative eigenvalues:
\begin{equation*}
	0 \leq \omega_{1,\delta}^2(\alpha) \leq \omega_{2,\delta}^2(\alpha) \leq \cdots \leq \omega_{j,\delta}^2(\alpha) \leq \cdots,
\end{equation*}
listed with multiplicity, and tending to positive infinity. The frequency $\omega_{j,\delta}$ and the corresponding \textit{Bloch eigenmode} $\Phi_{j,\delta}$ are determined by the following eigen-value problem: find a $L^2_\delta(Y)$-normalized $\alpha$-quasiperiodic function  $\Phi_{j,\delta}(\cdot,\alpha)  \in H^1_{\delta,\Lambda^\alpha}$ such that the pair $(\omega^2_{j,\delta}(\alpha),\Phi_{j,\delta})$ solves 
\begin{equation}
\label{def eigenmode}
	\begin{aligned}
	&&\mathcal{L}_{\delta}(\alpha ) \Phi_{j,\delta}(\cdot,\alpha) = \omega_{j,\delta}^2(\alpha) \Phi_{j,\delta}(\cdot,\alpha), \quad &\mathrm{for\ any\ }j \geq 1,\\
		&&\left\langle \Phi_{j,\delta}(\cdot,\alpha) , \Phi_{k,\delta}(\cdot,\alpha) \right\rangle_{L^2_{\delta}(Y)} = \delta_{jk}, \quad &\mathrm{for\ any\ }j,k  \geq 1. 
  \end{aligned}
\end{equation}
The second line requires that $\{\Phi_{j,\delta}(\cdot,\alpha)\}_j$ is orthonormal in $L^2_\delta(Y)$ (or equivalently, in $L^2_{\delta,\Lambda^\alpha}(\R^2)$). Each $\omega_j(\alpha)$ is called a \textit{Bloch eigenfrequency}, and each pair $(\omega_{j,\delta}^2(\alpha),\Phi_{j,\delta}(\cdot,\alpha))$ is called a \emph{Floquet-Bloch state}. For each $j\in \N^*$, we call the mapping
\begin{equation*}
\begin{aligned}
\omega_{j,\delta} \quad : \quad Y^* &\mapsto [0,\infty)\\
	\alpha &\mapsto \omega_{j,\delta}(\alpha)
 \end{aligned}
\end{equation*}
the $j$-th \textit{Bloch (frequency) band} associated to the operator $\mathcal{L}_{\delta}$. It is known to be a Lipschitz continuous function. The \textit{band structure} of $\mathcal{L}_{\delta}$ refers to the collection of Bloch bands and their associated Bloch eigenmodes.

By \eqref{fl-b}, the spectrum of $\mathcal{L}_{\delta}$ can be obtained from the spectra of the fiber operators:
\begin{equation*}
	\mathrm{Spec}\,(\mathcal{L}_{\delta}) = \bigcup_{j=1}^{\infty}  \big[ \min_{\alpha \in Y^*} \omega_{j,\delta}^2(\alpha) ,  \max_{\alpha \in Y^*} \omega_{j,\delta}^2(\alpha) \big] .
\end{equation*}

The Floquet-Bloch states $\{\omega_{j,\delta}^2(\alpha),\Phi_{j,\delta}(\cdot,\alpha)\}$ give convenient representation for the direct integral of Hilbert spaces $\int_{Y^*}^\oplus L^2_\delta(\alpha)\,d\alpha$, and hence for  $L^2_\delta$ functions via Floquet-Bloch transformation. Indeed, we can choose a family $\{(\omega_{j,\delta}^2(\alpha),\Phi_{j,\delta}(\cdot,\alpha)\}_{\alpha,j}$ so that for each fixed $j$, $\alpha\mapsto \omega_{j,\delta}(\alpha)$ and $\alpha \mapsto \Phi_{j,\delta}(\cdot,\alpha)$ are measurable and, for each fixed $\alpha \in Y^*$, the states $\{\Phi_{j,\delta}(\cdot,\alpha)\}_{j=1}^\infty$ form an orthonormal basis for $L^2_\delta(\alpha)$; see e.g.\,\cite{MR516045}. We say this choice specifies a measurable field of Hilbert spaces $\{L^2_\delta(\alpha)\}_{\alpha\in Y^*}$ over $Y^*$. A section $f : \alpha \mapsto \prod_{\alpha\in Y^*} L^2_\delta(\alpha)$ (that is, a mapping satisfying $f(\cdot,\alpha)\in L^2_\delta(\alpha)$ for all $\alpha$) then called a measurable vector field if $\alpha \mapsto \langle f(\cdot,\alpha),\Phi_j(\cdot,\alpha)\rangle_{L^2_\delta(\alpha)}$ is measurable for all $j$. The direct integral $\int_{Y^*}^\oplus L^2_\delta(\alpha)\,d\alpha$ is then the space of measurable vector fields such that $\|f(\cdot,\alpha)\|_{L^2_\delta(\alpha)}$ is square integrable over $\alpha \in Y^*$ and can be made a separable Hilbert space itself. We refer to \cite[section 7.4]{MR3444405} for the precise definitions and for more details.

With the above set-up, we see that the Floquet-Bloch transform $\mathcal{U}$ maps $L^2_\delta(\R^2)$ isometrically to the direct integral $\int_{Y^*}^\oplus L^2_\delta(\alpha)\,d\alpha$, and the following representations hold.
\begin{theorem}
	For any $v \in \mathscr{S}(\R^2)$, 
the identity
	\begin{equation}
 \label{eq:L2FBrep}
		v(x) = \sum_{j=1}^{\infty} \int_{Y^*}\langle v(\cdot) , \Phi_{j,\delta}(\cdot,\alpha) \rangle_{L^2_{\delta}(\mathbb{R}^2)} \Phi_{j,\delta}(x,\alpha) \,d\alpha
	\end{equation}
 holds in $L^2(\R^2)$ and we have the following Plancherel-type inequality
\begin{equation}
\label{eq:plancherel}
	\|v \|_{L^2_{\delta}(\mathbb{R}^2)}^2 = \sum_{j=1}^{\infty} \int_{Y^*} \left| \langle v(\cdot) , \Phi_{j,\delta}(\cdot,\alpha) \rangle_{L^2_{\delta}(\mathbb{R}^2)}  \right|^2 \,d\alpha .
\end{equation}
\end{theorem}
Here, for $v\in \mathscr{S}(\R^2)$, the pairing $\langle v, \Phi_{j,\delta}(\cdot,\alpha)\rangle_{L^2_\delta(\R^2)}$  in $L^2_\delta(\R^2)$, but it makes sense as the integral over $\R^2$ of the product $v\overline{\Phi_{j,\delta}}$ with respect to the weighted measure $\sigma_\delta^{-1}dx$ and it is consistent with the $L^2_\delta$ inner product. In fact, we can define it more precisely as
\begin{equation}
\label{eq:vPhiL2}
\langle v, \Phi_{j,\delta}(\cdot,\alpha)\rangle_{L^2_\delta(\R^2)} := \langle \mathcal{U} v,\Phi_{j,\delta}\rangle_{L^2_\delta(\alpha),L^2_\delta(\alpha)}.
\end{equation}
Here, the right hand side above is the inner product in $L^2_\delta(\alpha)$. Indeed, for $v\in \mathscr{S}(\R^2)$, using the periodicity of $\sigma_\delta$ and the $\alpha$-quasiperiodicity of $\Phi_{j,\delta}(\cdot,\alpha)$, we have
\begin{equation*}
\begin{aligned}
\langle \mathcal{U} v, \Phi_{j,\delta}\rangle_{L^2_\delta(\alpha),L^2_\delta(\alpha)} &= \int_{Y} \sum_{\ell \in \Lambda} v(x-\ell)e^{\mi \alpha\cdot \ell}\overline{\Phi_{j,\delta}(x,\alpha)}\sigma_\delta^{-1}(x)\,dx\\
 &= \sum_{\ell \in \Lambda} \int_{Y}  v(x-\ell)\overline{\Phi_{j,\delta}(x-\ell,\alpha)}\sigma_\delta^{-1}(x-\ell)\,dx\\
 &= \int_{\R^2} v(x)\overline{\Phi_{j,\delta}(x,\alpha)}\sigma^{-1}_\delta(x)\,dx.
\end{aligned}
\end{equation*}
By the density of $\mathscr{S}$ in $L^2_\delta(\R^2)$, the Plancherel identity above (via Floquet-Bloch states) holds for general $u\in L^2_\delta(\R^2)$ and the representation \eqref{eq:L2FBrep} holds in proper senses.

Similarly, we also have:
\begin{corollary}
	For any $u \in H^1_{\delta}(\mathbb{R}^2)$, one has
	\begin{equation}\label{eq:H1FB}
		\|u \|^2_{H^1_{\delta}(\mathbb{R}^2)} = \sum_{j=1}^{\infty} \int_{Y^*} \Big(1 + \omega_{j,\delta}^2(\alpha) \Big)\left| \langle u(\cdot) , \Phi_{j,\delta}(\cdot,\alpha) \rangle_{L^2_{\delta}(\mathbb{R}^2)}  \right|^2 \,d\alpha .
	\end{equation}
\end{corollary}
The proof of the corollary above is straightforward and is essentially due to the fact that for each fixed $\alpha\in Y^*$, the sequence $\{\Phi_{j,\delta}(\cdot,\alpha)\}_j$ form a Hilbert basis for $H^1_{\delta,\Lambda^\alpha}$ and the identity $\|\nabla \Phi_{j,\delta}\|^2_{L^2_\delta(Y)} = \omega^2_{j,\delta}(\alpha)$ as seen from the eigen-value problem. One then checks the identity for $u\in \mathscr{S}$ and concludes by density of $\mathscr{S}$ in $H^1_\delta(\R^2)$. In fact, just as our discussion for $L^2_\delta(\R^2)$ before, the Floquet-Bloch transform $\mathcal{U}$ is an isomorphism from $H^1_\delta(\R^2)$ to the direct integral $\int_{Y^*}^\oplus H^1_\delta(\alpha)\,d\alpha$ where $H^1_\delta(\alpha)$ is the space $H^1_{\delta,\Lambda^\alpha}(\R^2)$ restricted to the unit cell $Y$. Indeed, the Floquet-Bloch states chosen before specify $\{H^1_{\delta}(\alpha)\}_{\alpha \in Y^*}$ as a measurable field of Hilbert spaces and then we define the direct integral space as before. Moreover, if the latter Hilbert space is equipped with the norm
\begin{equation*}
\|f\| = \Big(\int_{Y^*} \|f(\cdot,\alpha)\|_{H^1_\delta(\alpha)}^2 \,d\alpha\Big)^{\frac12},
\end{equation*}
then $\mathcal{U}$ is also isometric. Compute the $H^1_\delta(\alpha)$ norm of $\mathcal{U}f(\cdot,\alpha)$ according to its coefficients with respect to the basis $\{\Phi_{j,\delta}(\cdot,\alpha)\}_j$, we arrive at the Plancherel type identity \eqref{eq:H1FB}.

It is also noteworthy that the direct integral $\int_{Y^*}^\oplus L^2_\delta(\alpha)\,d\alpha$ is somewhat trivial in the sense that $L^2_\delta(\alpha)$ are the same for all $\alpha \in Y^*$; indeed, $L^2_\delta(\alpha)$ is just $L^2_\delta(Y)$ (all $L^2(Y)$ functions can be extended $\alpha$-quasi-periodically from $Y$ to $\R^2$). This is not the case for $\int_{Y^*}^\oplus H^1_\delta(\alpha)\,d\alpha$ and $\{H^1_\delta(\alpha)\}_{\alpha \in Y^*}$ because not all $H^1(Y)$ functions can be extended $\alpha$-quasi-periodically. This point was stressed in \cite{MR3501794}.

In this paper, we sometime need to generalize the definition of Floquet-Bloch transform and the Plancherel type identity above to the class of $H^{-1}_\delta(\R^2)$. Let $\mathscr{H}$ denote the space $\int_{Y^*}^\oplus H^1_\delta(\alpha)\,d\alpha$, and let $\mathscr{H}'$ denote its dual space $\int_{Y^*}^\oplus H^{-1}_\delta(\alpha)\,d\alpha$ where, for each $\alpha \in Y^*$, $H^{-1}_\delta(\alpha)$ is the dual of $H^1_\delta(\alpha)$. By the eigenvalue problem, for each $\alpha \in Y^*$, the sequence $\{\Phi_{j,\delta}(\cdot,\alpha)\}_j$ is still an orthogonal basis for $H^{-1}_\delta(\alpha)$. Moreover, for $f\in \mathscr{H}'$, $f(\cdot,\alpha)$ is characterized by the representation
\begin{equation*}
f(\cdot,\alpha) = \sum_{j=1}^\infty \frac{\langle f(\cdot,\alpha),\Phi_{j,\delta}(\cdot,\alpha)\rangle_{H^{-1}_\delta(\alpha),H^1_\delta(\alpha)}}{\sqrt{1+\omega^2_{j,\delta}(\alpha)}} \Phi_{j,\delta}(\cdot,\alpha).
\end{equation*}
Following \cite{AmmFep}, for $u\in H^{-1}_\delta(\R^2)$,  we define
\begin{equation*}
\langle \mathcal{U}u, f\rangle_{\mathscr{H}',\mathscr{H}} := \langle u, \mathcal{U}^{-1} f\rangle_{H^{-1}_\delta(\R^2),H^1_\delta(\R^2)}, \qquad \forall \, f\in \mathscr{H}.
\end{equation*}
It is straightforward to check that for $u\in \mathscr{S}(\R^2)$, the above agrees with the original definition \eqref{eq:FBT-S}. Clearly, $\mathcal{U}$ is continuous on $H^{-1}_\delta(\R^2)$. Hence, we say the above defines the Floquet-Bloch transform for $H^{-1}_\delta(\R^2)$. Then we have the following result:
\begin{corollary}
For $f \in H^{-1}_{\delta}(\mathbb{R}^2)$, we have
\begin{equation*}
	f = \sum_{j=1}^{\infty} \int_{Y^*} \langle f,\Phi_{j,\delta}(\cdot,\alpha)\rangle_{L^2_\delta(\R^2)} \Phi_{j,\delta}(\cdot,\alpha) \,d\alpha
\end{equation*}
and
\begin{equation}\label{eq:H-1exp}
	\|  f \|_{H_{\delta}^{-1} (\mathbb{R}^2) }^2 = \sum_{j=1}^{\infty} \int_{Y^*} \frac{|\langle f,\Phi_{j,\delta}(\cdot,\alpha)\rangle_{L^2_\delta(\R^)}|^2}{1+ \omega_{j,\delta}^2(\alpha)} \,d\alpha<\infty.
\end{equation}
Here, as an analog to \eqref{eq:vPhiL2}, $\langle f,\Phi_{j,\delta}(\cdot,\alpha)\rangle_{L^2_\delta(\R^2)}$ is understood as the $H^{-1}_\delta(\alpha)$-$H^1_\delta(\alpha)$ pairing between $\mathcal{U}f(\cdot,\alpha)$ and $\Phi_{j,\delta}(\cdot,\alpha)$.
\end{corollary}

\subsection{Some estimates for the Dirac equation}
\label{sec: energy homogenized}

In this section, we collect some estimates on solutions of the Dirac equation \eqref{eq:Dirac}, which are the profiles of the wave packets in the approximation \eqref{ansatz}. Taking Fourier transform on \eqref{eq:Dirac} for the spatial variable, we obtain the equation
\begin{equation}\label{eq:DiracFourierV}
	\mi\partial_t \begin{pmatrix}
		\widehat{V_1}\\ \widehat{V_2}
	\end{pmatrix} = \Omega(\xi) \begin{pmatrix}
		\widehat{V_1}\\ \widehat{V_2}
	\end{pmatrix}  , \quad \begin{pmatrix}
		\widehat{V_1}\\ \widehat{V_2}
	\end{pmatrix}   (\xi,0) = \begin{pmatrix}
		\widehat{F_1}\\ \widehat{F_2}
	\end{pmatrix}(\xi),
\end{equation}
where $\Omega = \Omega_\delta(\xi)$ is the Hermitian matrix
\begin{equation}\label{hermitian}
	\Omega(\xi) =  \begin{pmatrix}
		& \mi \frac{a_{\delta}}{2\omega_{\delta}^*} (\xi_1 +\mathrm{ i} \xi_2)  \\
		\overline{\mi\frac{a_{\delta}}{2\omega_{\delta}^*}} (\xi_1 - \mathrm{i}\xi_2)& 
	\end{pmatrix}.
\end{equation}
We also know that $|a_\delta/2\omega^*_\delta| = \mathcal{O}(\delta^{1/2})$. As a consequence of the fact that $\Omega^*(\xi) = \Omega(\xi)$, the following proposition about $\bm{V} = (V_1,V_2)$ holds.

\begin{proposition}
\label{prop:Dirac}
	Let $\bm{V}$ solve the Dirac equation \eqref{eq:Dirac} with initial data $\bm{V}(0) = \bm{F}$. Assume that $m \in \N_{>0}$ and $\bm{V}(0) \in H^{m}(\R^2)$. Then,
	\begin{itemize}
		\item[(i)] The Fourier transform (in $x$) of $\bm{V}$ is given by the formula
		\begin{equation*}
			\widehat{\bm{V}}(\xi,T) = e^{-\mi \Omega_\delta(\xi)T} \widehat{\bm{F}}(\xi), \quad \forall T \in \R;
		\end{equation*}
		\item[(ii)] For all $\xi\in \R^2$ and $T\in \R$, $|\widehat{\bm{V}}(\xi,T)| = |\widehat{\bm{F}}(\xi)|$;
		\item[(iii)] For any index $\mathbf{n} = (n_1,n_2) \in \N^2$ with $|\mathbf{n}| := n_1+n_2 \le m$, and for any $T\ge 0$,
		\begin{equation}
			\label{eq:DVL2}
		\|\mathscr{F}_{x\mapsto \xi}\left(\partial^{\mathbf{n}}_x \bm{V}\right)(\xi,T)\|_{L^2(\R^2_\xi)} \sim \|\partial^{\mathbf{n}}_x \bm{V}(x,T)\|_{L^2(\R^2)} = \|\partial^{\mathbf{n}} \bm{F}(x)\|_{L^2(\R^2)}.
		\end{equation}
\end{itemize}
\end{proposition}

The above results are easily proved by solving the equation \eqref{eq:DiracFourierV} and using the fact that the Fourier transform is an isometry from $L^2(\R^2_x)$ to $L^2(\R^2_\xi)$ (after factoring a constant), and hence the proof is omitted. It is also true that the solutions to \eqref{eq:Dirac} are smooth in space and time, in fact, $V_j(x,t) \in C^\infty(\R_t,\mathscr{S}(\R^2_x))$ given $F_j \in \mathscr{S}(\R^2)$, for $j=1,2$. 
As a consequence, we have the following useful estimates.
\begin{proposition}
\label{prop:DiracFourierEstimate}
	Assume that $F_1,F_2 \in \mathscr{S}(\mathbb{R}^2)$. Let $(V_1,V_2)$ be the solution to the Dirac equation \eqref{eq:Dirac}. Then we have the following results:
	\begin{itemize}
\item[(i)] There is a constant $C>0$ such that for any $k\in \N$, $s\in \R$ and 2-index $\mathbf{n}\in \N^2$, we have
\begin{equation}
\label{eq:DiracF-1}
\|\partial^{\mathbf{n}}_x\partial^k_t \bm{V}(x,t)\|_{H^s(\R^2_x)} \le C\delta^{\frac{k}{2}} \|\bm{F}\|_{H^{k+|\mathbf{n}|+s}(\R^2_x)}, \qquad \forall \,t\in \R;
\end{equation}
\item[(ii)] For any even integer $N \in \N_{>0}$, there is a constant $C(N)>0$ such that for any $k\in \N$ and 2-index $\mathbf{n}\in \N^2$, we have
\begin{equation}
\label{eq:DiracF-2}
\left| (\mathscr{F}_{x\mapsto \xi} \partial^{\mathbf{n}}_x \partial^k_t \bm{V})(\xi,t)\right| \le \frac{C\delta^{\frac k2}}{(1+|\xi|)^N}\|\bm{F}\|_{W^{N + k + |\mathbf{n}|,1}(\R^2_x)}, \quad \forall (\xi,t) \in \R^2_\xi\times \R_t;
\end{equation}
	\item[(iii)] For any even integer $N \in \N_{>0}$, there is a constant $C(N)>0$ such that for any $k\in \N$ and 2-index $\mathbf{n}\in \N^2$, we have
	\begin{equation}\label{eq:decayDirac}
		\left|\partial_x^{\mathbf{n}}\partial_t^k \bm{V}(x,t) \right| \leq \frac{C\delta^{\frac k2}}{(1+|x|)^N}\|(I-\Delta_\xi)^{N/2}\langle\xi\rangle^{k+|\mathbf{n}|} \widehat{\bm{F}}(\xi)\|_{L^1(\R^2_\xi)}, \quad \forall (x,t) \in \R^2_x\times \R_t.
	\end{equation}
\end{itemize}
\end{proposition}
\begin{proof}
	We first prove the results assuming $\mathbf{n}=(0,0)$. Concerning the time derivatives, for $k\in \N$, we have
\begin{equation*}
	\partial_t^{k} V_1(x,t) = \begin{cases}
		|\eta_{\#,\delta}|^m \Delta^m V_1, \quad&k=2m,\\
|\eta_{\#,\delta}|^m \eta_{\#,\delta} (-\partial_1-\mi \partial_2)\Delta^m V_2, \quad&k=2m+1.
	\end{cases}
\end{equation*}
Similarly, 
\begin{equation*}
	\partial_t^{k} V_2(x,t) = \begin{cases}
		|\eta_{\#,\delta}|^m \Delta^m V_2, \quad&k=2m,\\
|\eta_{\#,\delta}|^m \overline{\eta_{\#,\delta}} (-\partial_1+\mi \partial_2)\Delta^m V_1, \quad&k=2m+1.
	\end{cases}
\end{equation*}
We also note that $(\partial^k_tV_1, \partial^k_t V_2)$ solves the Dirac equation \eqref{eq:Dirac} with proper initial data (simply replacing $V_1,V_2$ on the
right-hand side above by $F_1,F_2$). By item (iii) of Proposition \ref{prop:Dirac} and the fact that $|\eta_{\#,\delta}| \sim \delta^{\frac12}$, we have, for any $s \in \R$ and $j=1,2$,
\begin{equation*}
\|\partial_t^k V_j(x,t)\|_{H^s(\R^2_x)} \le 
\begin{cases}
		C\delta^m \|F_j\|_{H^{k+s}(\R^2_x)}, \quad &k=2m,\\
C\delta^{m+\frac12} \|F_{j'}\|_{H^{k+s}(\R^2_x)}, \quad &k=2m+1,
	\end{cases}	
\end{equation*}	
where $j'=1$ if $j=2$ and $j'=2$ if $j=1$. This proves \eqref{eq:DiracF-1}.

\smallskip

For \eqref{eq:DiracF-2}, which estimates the decay in $|\xi|$ of Fourier transforms of partial derivatives of $\mathcal{V}$, we use integration by parts in space. Taking, for instance, $k=1$ yields 
\begin{equation*}
\begin{aligned}
\left|(1+|\xi|^2)^{\frac N2}\widehat{\partial_t V_1}(\xi,t)\right| &= \left|(1+|\xi|^2)^{\frac N2} \widehat{\partial_t V_1}(\xi,0)\right| \\
&= |\eta_{\#,\delta}| \left|\mathscr{F}_{x\to \xi}(I-\Delta_x)^{\frac N2}(\partial_1+\mi \partial_2)F_2(\xi)\right| \le C_N \delta^{\frac12} \|F_2\|_{W^{N+1,1}(\R^2_x)}.
\end{aligned}
\end{equation*}
In fact, this result for $k=1$ is already in \cite{fefferman_wave_2014}; the case for general $k$ is exactly the same.

\smallskip

For \eqref{eq:decayDirac}, which estimates the decay in $|x|$ of the partial derivatives of $\bm{V}$, we note that
\begin{equation}
\label{eq:DiracF-3-1}
(1+|x|^2)^{\frac N2} \partial^k_t V_j(t,x) = \mathscr{F}_{\xi\to x}(I-\Delta_\xi)^{\frac N2} \widehat{\partial^k_t V_j(\cdot,t)}(\xi), \quad j=1,2.
\end{equation}
From the wave equation \eqref{eq:waveV}, we get
\begin{equation*}
\begin{aligned}
\widehat{V}_j(\xi,t) &= \cos(|\eta_{\#,\delta}||\xi|t) \widehat{F}_j(\xi)+ \frac{\sin(|\eta_{\#,\delta}||\xi|t)}{|\eta_{\#,\delta}||\xi|} \partial_t \widehat{V}_j(\xi,0), \quad j=1,2.
\end{aligned}
\end{equation*}
Using the form of the Dirac equation \eqref{eq:Dirac}, we simplify the above to
\begin{equation}
\label{eq:VhatjSC}
\begin{aligned}
	\widehat{V_1}(\xi,t) &= \cos(|\eta_{\#,\delta}||\xi|t)\widehat{F_1}(\xi) + \sin(|\eta_{\#,\delta}||\xi|t) \mathrm{sgn}\left(\eta_{\#,\delta}(-\mi \xi_1 + \xi_2)\right) \widehat{F_2}(\xi),\\
	\widehat{V_2}(\xi,t) &= \cos(|\eta_{\#,\delta}||\xi|t)\widehat{F_2}(\xi) + \sin(|\eta_{\#,\delta}||\xi|t) \mathrm{sgn}\left(\overline{\eta_{\#,\delta}}(-\mi \xi_1 - \xi_2)\right) \widehat{F_1}(\xi).
\end{aligned}
\end{equation}
Here, $\mathrm{sgn}(z) := z/|z|$ is the unit vector associated to a complex number $z \ne 0$, and $\mathrm{sgn}(0) := 0$.
Take time derivatives in the above formulas and use them in \eqref{eq:DiracF-3-1}. Then, by the basic $L^\infty$ estimate for $\mathscr{F}_{\xi\to x}$, we get
\begin{equation*}
\begin{aligned}
\left|(1+|x|^2)^{\frac N2} \partial^k_t V_j(t,x)\right| &\le C\|(I-\Delta_\xi)^{\frac N2} \partial^k_t \widehat{V_1}(\xi,t)\|_{L^1(\R^2_\xi)}\\
&\le C\delta^{\frac k2} \sum_{\ell=1}^2 \|(I-\Delta_\xi)^{N/2} \langle\xi\rangle^k \widehat{F_\ell}(\xi)\|_{L^1(\R^2_\xi)}.
\end{aligned}
\end{equation*}
This verifies \eqref{eq:decayDirac}.

\smallskip

Finally, for general $2$-index $\mathbf{n}\in \N^2$, since $(\partial^{\mathbf{n}}_x V_1, \partial^{\mathbf{n}}_x V_2)$ also satisfies the Dirac system \eqref{eq:Dirac}, with initial data $(\partial^{\mathbf{n}}_x F_1, \partial^{\mathbf{n}}_x F_2)$, the desired result follows directly from the special cases above.
\end{proof}

\section{Subwavelength band structure near the Dirac points} 
\label{sec:DiracStructure}

In this section, we study the band structure of the periodic differential operator $\mathcal{L}_{\delta}$. The most important feature is, due to the high-contrast of the acoustic parameters of the bubbles compared with the background, the low lying bands correspond to subwavelength regime.

It was shown in \cite{ammari_subwavelength_2017, ammari_honeycomb-lattice_2020} that if there are $N$ bubbles in the unit cell $Y$, the first $N$ bands are low in amplitude and are isolated as $\delta \rightarrow 0$ in the following sense: 
there are constants $C_1,C_2>0$ such that for sufficiently small  $\delta>0$, 
\begin{equation}
\label{eq:bandall}
	\begin{aligned}
		\bigcup_{j=1}^N  \big[ \min_{\alpha \in Y^*} \omega_{j,\delta}(\alpha) ,  \max_{\alpha \in Y^*} \omega_{j,\delta}(\alpha) \big] &\subset [0,C_1\sqrt{\delta}],  \\
		\bigcup_{j=N+1}^{\infty}  \big[ \min_{\alpha \in Y^*} \omega_{j,\delta}(\alpha) ,  \max_{\alpha \in Y^*} \omega_{j,\delta}(\alpha) \big] &\subset [C_2,\infty).
	\end{aligned}
\end{equation}
Based on this fact, we refer the first two bands to the \textit{subwavelength bands} (in our case, $N=2$). Formula	\eqref{eq:bandall} is established in \cite{ammari_subwavelength_2017} for $N=1$, and in \cite{ammari_honeycomb-lattice_2020} for $N=2$. Their method based on capacitance matrix relies on layer potential theory and Gohberg-Sigal theory, and can be generalized to the $N$ bubbles case. 

\subsection{Quasiperiodic  layer potentials and capacity matrix}
\label{sec:layerpotential}

We first review some layer potential theory for the Helmholtz equation, which is of fundamental importance in the approach of \cite{ammari_honeycomb-lattice_2020,ammari_high-frequency_2020}. For each $\alpha \in \R^2$ and wave number $k\in \R$, let $G^{\alpha, k}$ be the $\alpha$-quasiperiodic Green's function to the Helmholtz equation defined by
\begin{equation*}
	\Delta G^{\alpha, k}(x)+k^2 G^{\alpha, k}(x)=\sum_{l \in \Lambda} \delta(x-l) e^{\mathrm{i} \alpha \cdot l} ,
\end{equation*}
where $\delta(\cdot)$ is the standard Dirac distribution.
Then it can be shown that $G^{\alpha, k}$ is given by \cite{ammari2018mathematical, ammari2009layer}
\begin{equation*}
	G^{\alpha, k}(x)=\frac{1}{|Y|} \sum_{q \in \Lambda^*} \frac{e^{\mathrm{i}(\alpha+q) \cdot x}}{k^2-|\alpha+q|^2},  
\end{equation*}
for $k^2\neq |\alpha+q|^2, \, \forall \, q \in \Lambda^*$.

For a given bounded domain $D$ in $Y$, with Lipschitz boundary $\partial D$, the single layer potential $\mathcal{S}_D^{\alpha, k}: L^2(\partial D) \rightarrow H_{\mathrm{loc}}^1(\mathbb{R}^2)$ is defined by
\begin{equation}
\label{eq:Salphak}
	\mathcal{S}_D^{\alpha, k}\varphi(x):=\int_{\partial D} G^{\alpha, k}(x-y) \varphi(y) \,d \sigma(y), \quad x \in \mathbb{R}^2 .
\end{equation}
Here, we denote by $H_{\mathrm{loc}}^1(\mathbb{R}^2)$ the space of functions that are square integrable on every compact subset of $\mathbb{R}^2$ and have a weak first derivative that is also square integrable. The following jump relations are well-known \cite{ammari2018mathematical, ammari2009layer}:
\begin{equation}
\label{eq:Salphatrace}
	\left.\frac{\partial}{\partial \nu} \mathcal{S}_D^{\alpha, k}\varphi\right|_{ \partial D\pm}=\left( \pm \frac{1}{2} I+(\mathcal{K}_D^{-\alpha, k})^*\right) \varphi ,
\end{equation}
where the Neumann-Poincar\'{e} operator $(\mathcal{K}_D^{-\alpha, k})^*: L^2(\partial D) \rightarrow L^2(\partial D)$ is defined as
\begin{equation*}
	(\mathcal{K}_D^{-\alpha, k})^*\varphi(x)=\mathrm{p.v.} \int_{\partial D} \frac{\partial}{\partial {\nu}_x} G^{\alpha, k}(x-y) \varphi(y)\,d \sigma(y), \quad x \in \partial D.
\end{equation*}

The Green's function can be asymptotically expanded for small $k$ as follows \cite{ammari2018mathematical}:
\begin{equation*}
	G^{\alpha, k}=G^{\alpha, 0}+k^2 G_1^\alpha+\mathcal{O}(k^4), \quad \quad G_1^\alpha(x):=-\sum_{q \in \Lambda^*} \frac{e^{\mathrm{i}(\alpha+q) \cdot x}}{|\alpha+q|^2},
\end{equation*}
where the error term is uniform in $\alpha$ in a neighbourhood of $\alpha^*$ and $x \in Y$. This leads to the following expansion of the single layer potential $\mathcal{S}_D^{\alpha, k}: L^2(\partial D) \rightarrow H^1(\partial D)$
\begin{equation}\label{eq:Sk0error}
		\mathcal{S}_D^{\alpha, k}=  \mathcal{S}_D^{\alpha, 0}+k^2 \mathcal{S}_{D, 1}^\alpha+\mathcal{O}_{L^2(\partial D) \rightarrow H^1(\partial D)}(k^4), \quad \mathcal{S}_{D, 1}^\alpha\phi(x):=\int_{\partial D} G_1^\alpha(x-y) \phi(y)\,d \sigma(y),
\end{equation}
where $\mathcal{O}(k^4)$ denotes an operator $L^2(\partial D) \rightarrow H^1(\partial D)$ with operator norm of order $k^4$, uniformly for $\alpha$ in a neighbourhood of $\alpha^*$. Similarly, for the Neumann-Poincar\'{e} operator, we have
\begin{equation}\label{eq:Kk0error}
	(\mathcal{K}_D^{-\alpha, k})^*=  (\mathcal{K}_D^{-\alpha, 0})^*+k^2 \mathcal{K}_{D, 1}^\alpha+\mathcal{O}(k^4), \quad \mathcal{K}_{D, 1}^\alpha\phi(x):= \int_{\partial D} \frac{\partial}{\partial \nu_x} G_1^\alpha (x-y) \phi(y) \,d \sigma(y),
\end{equation}
where $\mathcal{O}(k^4)$ denotes an operator $L^2(\partial D) \rightarrow L^2(\partial D)$ with operator norm of order $k^4$, uniformly for $\alpha$ in a neighbourhood of $\alpha^*$.

It is known that $\mathcal{S}_D^{\alpha, 0}: L^2(\partial D) \rightarrow H^1(\partial D)$ is invertible when $\alpha \neq 0$ \cite{ammari2018mathematical, ammari2009layer}. Let $\psi_j^\alpha \in L^2(\partial D)$ be given by
\begin{equation}
\label{eq:Salpha0}
\mathcal{S}_D^{\alpha, 0}\left[\psi_j^\alpha\right]=\mathbf{1}_{\partial D_j} \quad \text { on } \partial D, \quad j=1,2,
\end{equation}
where $\mathbf{1}$ denotes the indicator function. It is known that, for every $\alpha \in Y^*$, 
\begin{equation}
\label{eq:kerK0}
	\mathrm{ker}\, \left( -\frac{1}{2}I+  ( \mathcal{K}_D^{-\alpha, 0} )^* \right) = \mathrm{Span}\,\left\{ \psi_1^{\alpha},  \psi_2^{\alpha}    \right\}.
\end{equation}
Define
	\begin{equation}\label{eq:Sjdef}
		S^\alpha_j(x) = \mathcal{S}^{\alpha,0}_D \psi_j^{\alpha}(x),  \quad x\in \R^2, \quad \mathrm{for}\  j = 1,2.
	\end{equation}
Then $S^\alpha_j$ (extended by constants in $D$) are $\alpha$-quasiperiodic solutions to the Laplace equation in the perforated domain $\R^2\setminus (D+\Lambda)$: 
\begin{equation}
	\label{eq:S0equation}
-\Delta S^\alpha_j = 0 \; \text{in $\R^2\setminus (\ol{D}+\Lambda)$}, \quad S^\alpha_j = \delta_{ij} \; \text{in $D_i$}, \quad i,j \in \{1,2\}.
\end{equation}
Note that $S^\alpha_1$ and $S^\alpha_2$ are clearly independent, and their norms in $L^2(Y)$ and $H^1(Y)$ are all comparable to $1$.

Define the capacitance matrix $C^\alpha = \left(C_{i j}(\alpha)\right)$ by
\begin{equation}
\label{eq:capacitance}
	C_{i j}(\alpha):=-\int_{\partial D_i} \psi_j^\alpha \,d \sigma, \quad i, j=1,2 .
\end{equation}
In fact, the expression above agrees with the more standard definition below featuring the name `capacitance':
\begin{equation*}
	C_{ij}(\alpha) := \int_{Y\setminus \ol D} \overline{\nabla S_i^\alpha(x)} \cdot \nabla S_j^\alpha(x)\,dx = \int_{Y} \overline{\nabla S_i^\alpha(x)} \cdot \nabla S_j^\alpha(x)\,dx.
\end{equation*}
Here the last equality holds because $S_j^\alpha$ and $S_i^\alpha$ are extended by constants in $D$. 

Due to the honeycomb symmetry \eqref{eq:hcsym}, it was shown in \cite{ammari_honeycomb-lattice_2020} that the capacitance matrix $C^\alpha$ is Hermitian and that its entries
\begin{equation}
\label{eq:Capp}
	c_1^\alpha:=C_{11}(\alpha) =C_{22}(\alpha), \quad c_2^\alpha:=C_{12}(\alpha)=\overline{C_{21}(\alpha)},
\end{equation}
as functions of $\alpha$, are differentiable and satisfy
\begin{equation}
\label{eq:gradC}
	\left.\nabla_\alpha c_1^\alpha\right|_{\alpha=\alpha^*}=\begin{pmatrix}
		0 \\
		0
	\end{pmatrix},\left.\quad \nabla_\alpha c_2^\alpha\right|_{\alpha=\alpha^*}=c\begin{pmatrix}
		1 \\
		-\mathrm{i}
	\end{pmatrix},
\end{equation}
where
\begin{equation}\label{eq:cdef}
	c:=\left.\frac{\partial c_2^\alpha}{\partial \alpha_1}\right|_{\alpha=\alpha^*}. 
\end{equation}
Note that, in the equation above, $\alpha_1$ is the first component of the vector variable $\alpha$. Moreover, it was proved in \cite[Lemma 3.5]{ammari_honeycomb-lattice_2020} that $c\ne 0$. We give another formula for $\nabla_\alpha c^\alpha_2\rvert_{\alpha=\alpha^*}$ in Theorem \ref{thm:gradalphaC12} below.


\subsection{Operators related to symmetries of the honeycomb structure}


We introduce some operators that distinguish the honeycomb symmetry and plays an important role in the analysis of partial differential operators with such symmetry. Recall the rotation and reflection operators $R,R_1,R_2,$ and $R_0$ for vectors in $\R^2$ defined in \eqref{eq:R12def}. 

Let $\mathcal{R}$ be a `rotation' operator of functions defined by
\begin{equation}\label{eq:Rfdef}
	\mathcal{R}f (x) =\left\{ \begin{aligned}
		& e^{-\mathrm{i}\alpha^* \cdot l_1} f(R_1 x) , &&x \in Y_1, \\
		& e^{-2 \mathrm{i}\alpha^* \cdot l_1} f(R_2 x) , && x \in Y_2,
	\end{aligned}\right. \quad\quad \mathrm{for\ any\ function\ } f : Y\to \bC,
\end{equation}
where $Y_1$ is the left half part of $Y$ and $Y_2$ is the right part; see Figure \ref{fig:honeycomb1}. Via extension cell by cell we also view $\mathcal{R}$ as an operator acting on $f: \R^2\to Y$.
One can check that $\mathcal{R}$ is well-defined and maps $L^2_\delta(\alpha^*)$ to $L^2_\delta(\alpha^*)$ isometrically. Indeed, it suffices to check $\mathcal{R}f$ at $l_2 + tl_1$, at $l_1 + tl_2$ and at $x_0+se_2$ for $t\in [0,1]$ ans $s\in [-\frac12,\frac12]$. A direct computation shows
\begin{equation*}
	\mathcal{R}f(tl_1 + l_2) = e^{\mi \alpha^*\cdot l_2} \mathcal{R}f(tl_1), \quad 
	\mathcal{R}f(tl_2 + l_1) = e^{\mi \alpha^*\cdot l_1} \mathcal{R}f(tl_2), \quad \forall t\in (0,1).
\end{equation*}
Similarly, the values of $\mathcal{R}f$ on $x_0+se_2$ agree when computed from $Y_1$ and $Y_2$; that is,
\begin{equation*}
	e^{-\mi \alpha^*\cdot l_1} f(R_1(x_0+se_2)) = e^{-\mi \alpha^* \cdot l_1} f(R_2(x_0+se_2)), \quad \forall s\in (-1/2,1/2).
\end{equation*}

We observe that $\mathcal{R}^3 = I$, so the spectrum of $\mathcal{R}$ consists of 
\begin{equation*}
	\{z\in \mathbb{C}: z^3 = 1\} = \{1, \tau , \tau^2 = \overline{\tau}\},
\end{equation*}
where $\tau : = e^{\mathrm{i}\frac{2\pi}{3}}$. Moreover, we can check that $\mathcal{R}$ is a normal operator. Hence, by the spectral theorem,  $L_{\delta}^2(\alpha^*)$ admits the following orthogonal decomposition:
\begin{equation}\label{eq:L2delta_orthogonal}
	L_{\delta}^2(\alpha^*) = L_{\delta}^2(\alpha^*,1) \oplus L_{\delta}^2(\alpha^*,\tau) \oplus L_{\delta}^2(\alpha^*,\overline{\tau})  ,
\end{equation}
where, for each $\nu \in \{ 1, \tau,\overline{\tau} \}$, $L_{\delta}^2(\alpha^*, \nu)$ is the corresponding eigenspace of $\mathcal{R}$, that is, 
$$
L_\delta^2(\alpha^*,\nu) := \{ g \in L_{\delta}^2(\alpha^*): \mathcal{R} g = \nu g\}.
$$
For $\delta=1$, when the inner product reduces to the standard one, we simply write them as $L^2(\alpha^*,\nu)$ for $\nu=1,\tau,\overline{\tau}$. Suppose $g\in L^2_\delta(\alpha^*,\nu)$, with $\nu \in \{\tau,\overline{\tau}\}$, then $\mathcal{R}(\mathcal{R}g) = \mathcal{R}(\nu g) = \nu \mathcal{R}g$. Hence, we see that
\begin{equation}
	\label{eq:Rtao} 
	\mathcal{R} \quad :\quad  L^2_\delta(\alpha^*,\tau)  \to L^2_\delta(\alpha^*,\tau)
\end{equation}
is an isometric isomorphism.

Let $\mathcal{P}$ be the `parity-inversion' operator with respect to the center $x_0$ of the unit cell $Y$ and let $\mathcal{C}$ be the complex-conjugation operator, both acting on functions in $Y$. The operators $\mathcal{P}$ and $\mathcal{C}$ are defined by
\begin{equation*}
\mathcal{P}f(x) = f(2x_0-x), \quad \mathcal{C}f(x) = \overline{f(x)}.
\end{equation*}
It is clear that $\mathcal{P}$ and $\mathcal{C}$ are isometric isomorphisms from $L^2_\delta(\alpha)$ to $L^2_\delta(-\alpha)$ for all $\alpha\in Y^*$, and that $[\mathcal{P},\mathcal{C}] = 0$. Moreover, $\mathcal{P}^2=\mathcal{C}^2=I$. The composition of the two operators is then
\begin{equation*}
	\mathcal{P}\mathcal{C} :L^2_\delta(\alpha) \rightarrow L^2_\delta(\alpha), \quad \quad \mathcal{P} \mathcal{C}f (x) = \overline{ f (2x_0 -x)},
\end{equation*}
and it satisfies $(\mathcal{P}\mathcal{C})^2 =I$. It is also routine to check that $[\mathcal{R},\mathcal{P}\mathcal{C}] = 0$, $[\mathcal{L}_\delta,\mathcal{PC}] = 0$, and
\begin{equation}
	\label{eq:PCtao} 
	\mathcal{P} \mathcal{C} \quad :\quad  L^2_\delta(\alpha^*,\tau)  \to L^2_\delta(\alpha^*,\overline{\tau})
\end{equation}
is an isometric isomorphism.

\subsection{Subwavelength band structure and the Bloch eigenmodes at the Dirac points} 

For the honeycomb structure, it was known that the first two subwavelength bands form a conical dispersion relation near the Dirac point $\alpha^* = \alpha_1^*$ (in fact, this happens also at the other Dirac point $\alpha^*_2$ in $Y^*$). Such a conical dispersion is referred to as a \textit{Dirac cone}. The following theorem was essentially proved in \cite{ammari_honeycomb-lattice_2020, ammari_high-frequency_2020}.

\begin{theorem}\label{thm:DiracModes}
	For $\delta$ sufficiently small, the two subwavelength bands of $\mathcal{L}_\delta$ have the following properties:
	\begin{itemize}
		\item[{\upshape (i)}] The first two bands form a Dirac cone in the following sense:
	\begin{equation}
 \label{eq:cone}
 \left\{
		\begin{aligned}
			& 	\omega_{1,\delta}(\alpha) =\omega_{\delta}^* - \lambda_{\delta} |\alpha - \alpha^*|\big[ 1 + \mathcal{O} (|\alpha - \alpha^*|)\big],  \\
			& 	\omega_{2,\delta}(\alpha) =\omega_{\delta}^* + \lambda_{\delta} |\alpha - \alpha^*|\big[ 1 + \mathcal{O} (|\alpha - \alpha^*|)\big], 
		\end{aligned}\right. \quad\quad  \mathrm{as}\ \alpha \rightarrow \alpha^*,
	\end{equation}
	where the error term is uniform for $\delta$ in a neighborhood of $0$. The constants $\omega_{\delta}^*$ and $\lambda_{\delta} $ are independent of $\alpha$ and are given by
	\begin{equation}\label{eq:omegastar}
			\omega_{\delta}^* = \sqrt{\frac{  c_1^{\alpha^*} }{|D_1|}} \delta^{\frac{1}{2}} + \mathcal{O}(\delta), \quad \quad \lambda_{\delta}  = \lambda_0|c| \delta^{\frac{1}{2}}+ \mathcal{O}(\delta),\quad \quad  \lambda_0 =  \frac{1}{2}   \sqrt{\frac{1}{|D_1| c_1^{\alpha^*}}} ,
	\end{equation}
	as $\delta \rightarrow 0$. Here $|D_1|$ denotes the volume of the bubble $D_1$, $c_1^{\alpha^*}$ and $c$ are constants respectively defined in \eqref{eq:Capp} and \eqref{eq:cdef};

\item[{\upshape (ii)}] For the resonant frequency $\omega^*_\delta$ at the Dirac point $\alpha^*$, the operator $\mathcal{L}_{\delta}(\alpha^*) - (\omega^*_{\delta})^2 I$ has a two-dimensional kernel. We can find a basis $\{S_{1,\delta},S_{2,\delta}\}$ for it. Moreover, $S_{1,\delta}$ and $S_{2,\delta}$ are related through the following symmetric transforms:
		\begin{equation}
		\label{eq:symSd12}
			\mathcal{R}S_{1,\delta} = \tau S_{1,\delta}  ,\quad \quad  \mathcal{R}S_{2,\delta} = \overline{\tau} S_{2,\delta}  , \quad \quad \mathcal{PC}  S_{1,\delta} =  S_{2,\delta},
		\end{equation} 
and are well approximated in $H^1(Y)$ by $S_1 = S_1^{\alpha^*}$ and $S_2 = S_2^{\alpha^*}$ defined in \eqref{eq:Sjdef} by
		\begin{equation}\label{eq:SjH1error}
			S_{j,\delta} - S_j  = \mathcal{O}_{H^1(Y)}(\delta), \quad \quad  \mathrm{for}\ j=1,2.
		\end{equation}
\end{itemize}
\end{theorem}
The results above show that the first two bands are in the subwavelength regime and they intersect conically at $\alpha^*$ where the two eigenvalues coincide with the same value $\omega^*_{\delta}$. Moreover, the eigenspace is of two dimension and two basis functions with related symmetries can be found. The results were proved essentially in Theorem 3.2, Theorem 4.1 and Lemma 4.3 of \cite{ammari_honeycomb-lattice_2020}. For the sake of completeness, we will reproduce a proof here.

\begin{remark}\label{rem:SHdel}
By construction, $\nabla S_j = 0$ in $D$ for $j=1,2$ because $S_j$ are constants in each of the inclusions in $D$. Hence,
\begin{equation*}
\|\nabla S_{j,\delta}\|_{L^2(D)} = \|\nabla S_{j,\delta} - \nabla S_j\|_{L^2(D)} \le C\delta.
\end{equation*}
It then follows that $\|\nabla S_{j,\delta}\|_{L^2_\delta(D)} \le C\sqrt{\delta}$ and $\|\nabla S_{j,\delta}\|_{L^2_\delta(Y)} \le C$ for $j=1,2$. Such an estimate of $L^2_\delta$ norms will be useful in our analysis of the wave packets propagation.
\end{remark}

In the following, we present a proof of Theorem \ref{thm:DiracModes} with some details referred to \cite[Theorems 3.2, Proposition 3.4, Theorem 4.1 and Lemma 4.3]{ammari_honeycomb-lattice_2020}. The proof explores the layer potential theory for the eigenvalue problem associated to $\mathcal{L}_{\delta}(\alpha^*)$ and the symmetries of 
the honeycomb structure.

\begin{proof}[Proof of Theorem \ref{thm:DiracModes}]
\emph{Step 1. Layer potential representation of the eigenvalue problem.} Using the quasiperiodic single layer potential operator $\mathcal{S}^{\alpha,k}_D$ defined in \eqref{eq:Salphak}, we can solve the eigenvalue problem
\begin{equation}
	\label{eq:egproblem}
\mathcal{L}_\delta(\alpha) \Phi = \omega^2 \Phi, \quad \Phi\in H^1_{\delta,\Lambda^\alpha}(\R^2/\Lambda), 
\end{equation}
by seeking a non-zero functions $\phi\in L^2(\partial D)$ such that (where the wave number $k=\omega$ since the sound speed $v$ and $v_b$ are set to be one)
\begin{equation}
\label{eq:egprobM}
	\left(-\frac12 I + (\mathcal{K}_D^{-\alpha,k})^*\right) \phi = \delta\left(\frac12 I + (\mathcal{K}_D^{-\alpha,k})^*\right) \phi.
\end{equation}
The integral equation above is equivalent to the eigenvalue problem because, setting
\begin{equation*}
	\Phi(x) = 
		\mathcal{S}_D^{\alpha,k} \phi(x), \quad x\in Y,
\end{equation*}
we see that $\Phi$ automatically satisfy the $\alpha$-quasiperiodic Helmholtz equation inside and outside $D$; we can then use the trace formula \eqref{eq:Salphatrace} to get the system above, and vice versa. Note that, because we have assumed that the sound speeds $v$ and $v_b$ (outside and inside the bubbles, respectively) are both equal to $1$, the layer potential representation is simpler than the setting of \cite{ammari_honeycomb-lattice_2020} with two wave speeds.

For Theorem \ref{thm:DiracModes}, we are interested in eigenvalue problems in the subwavelength regime, \emph{i.e.}, for $k = \omega = o(1)$. Then using the controlled approximation of $\mathcal{S}_D^{\alpha,k}, (\mathcal{K}_D^{-\alpha,k})^*$ by $\mathcal{S}_D^{\alpha,0}$ and $(\mathcal{K}_D^{-\alpha,0})^*$ in \eqref{eq:Sk0error} and \eqref{eq:Kk0error}, we see that such eigenmodes must be constructed with
\begin{equation*}
 	\left(-\frac12 I + (\mathcal{K}_D^{-\alpha,0})^*\right)\phi = \mathcal{O}(\omega^2). 
 \end{equation*}
Given $\alpha \in Y^*$, in order to find $\omega$ and $\phi$ with $\|\phi\|_{L^2(\partial D)} =1 $ solving the equation above, we let $\{\psi_1^\alpha, \psi_2^\alpha\}$ be defined as in \eqref{eq:Salpha0}. Then we can seek for $a,b$ (depending on $\delta$) such that
\begin{equation*}
	\phi = a\psi^\alpha_1 + b\psi^\alpha_2 + h, \quad h \in \left(\mathrm{ker}(-\frac12 I +  (\mathcal{K}_D^{-\alpha,0})^*)\right)^{\perp}.
\end{equation*}
We then get $h = \mathcal{O}(\omega^2)$ in $L^2(\partial D)$ and, plugging the above to \eqref{eq:egprobM} and using \eqref{eq:Kk0error}, we also deduce that
\begin{equation*}
\left(\omega^2 \mathcal{K}_{D,1}^{\alpha} - \delta I\right) [a\psi_1^\alpha + b\psi_2^\alpha] + (-\frac12 I + (\mathcal{K}_D^{-\alpha,0})^*)h = \mathcal{O}(\omega^4 + \delta\omega^2).
\end{equation*}
Integrate the above equation over $\partial D_1$ and $\partial D_2$, respectively. Observe that the integral of the second term on the left hand side above always vanishes. Using the fact (see \cite[Lemma 2.1]{ammari_honeycomb-lattice_2020}) that
\begin{equation*}
	\int_{\partial D_j} \mathcal{K}_{D,1}^{\alpha,0} [\psi] \,d\sigma = - \int_{D_j} \mathcal{S}_D^{\alpha,0}[\psi]\,d\sigma, \quad \forall \psi \in L^2(\partial D), \, j=1,2,
\end{equation*}
we deduce
\begin{equation*}
	\begin{aligned}
		-\omega^2 |D_1| a + \delta(aC^\alpha_{11} + bC^\alpha_{12}) &= \mathcal{O}(\omega^4+\delta\omega^2), \\
		-\omega^2 |D_2| b + \delta(aC^\alpha_{21} + bC^\alpha_{22}) &= \mathcal{O}(\omega^4+\delta\omega^2).
	\end{aligned}
\end{equation*}
Using the fact that $|D_1|=|D_2|$ and the structure of the capacitance matrix $C^\alpha$ in \eqref{eq:Capp}, we rewrite the above as
\begin{equation}
\label{eq:omega_ab}
	\begin{pmatrix} \delta c_1^\alpha - \omega^2|D_1| & \delta c^\alpha_2 \\ \delta \overline{c_2^\alpha} & \delta c_1^\alpha - \omega^2|D_1| \end{pmatrix} \begin{pmatrix} a \\ b \end{pmatrix} = \mathcal{O}(\omega^4 + \delta\omega^2).
\end{equation}
For the system to have unit size solutions, we deduce from the above equation that $(\omega^2/\delta)|D_1|$ approximates the eigenvalue of the capacitance matrix $C^\alpha$, and hence, $\omega = \mathcal{O}(\sqrt{\delta})$. It is also clear that there are at most two subwavelength bands. Moreover, at the Dirac point $\alpha^*$, it is known (see \cite[Lemma 3.3]{ammari_honeycomb-lattice_2020}) that $C^{\alpha^*} = c_1^{\alpha^*} I$, so if $\omega^*_\delta$ is the resonant frequency at this point, we must have $\omega^*_\delta = \mathcal{O}(\sqrt{\delta})$ and then, further, 
\begin{equation}
\label{eq:omegastarsquare}
	\delta c_1^{\alpha^*} - (\omega^*_\delta)^2|D_1| = \mathcal{O}(\delta^2).
\end{equation}
The expression of $\omega^*_\delta$ in \eqref{eq:omegastar} also follows immediately.

\smallskip

To establish the conical intersection behavior at $\alpha^*$ described in \eqref{eq:cone} is more involved, we refer to \cite[Theorem 4.1]{ammari_honeycomb-lattice_2020} and the proof there. We will turn to the precise asymptotics \eqref{eq:cone} later in Theorem \ref{thm:ModesExpansion}.

\smallskip

From our analysis above, we also obtain that for any $\alpha \in Y^*$, if $\omega^\alpha \in \{\omega_1^\alpha,\omega_2^\alpha\}$ is a subwavelength eigenvalue determined above, then the solution of the eigenvalue problem \eqref{eq:egproblem} can be approximated by a certain combination of $S_1^\alpha$ and $S_2^\alpha$ defined in \eqref{eq:Sjdef}. More precisely, for the subwavelength bands $\omega^\alpha$ and an $L^2(Y)$ normalized eigenvector $\Phi$, we must have
\begin{equation}
	\label{eq:Phialpha}
	\Phi(x,\alpha) = a S_1^\alpha(x) + b S_2^\alpha(x) + \mathcal{O}_{H^1(Y)}((\omega^\alpha)^2).
\end{equation}
We may seek for solutions that are `normalized' in the sense that $|a|^2+|b|^2=1$.

\emph{Step 2. We show that the eigenspace associated to the frequency $\omega^*_\delta$ at the Dirac point $\alpha^*$ is two dimensional}. We denote  by $N_{\delta}$ the kernel of  $\mathcal{L}_{\delta}(\alpha^*) - (\omega^*_{\delta} )^2  I$. We already know that for sufficiently small $\delta$ the dimension of $N_{\delta}$ is at most $2$. To prove by contradiction that the dimension is equal to two, we suppose that $\omega^*_{\delta} $ is a simple eigenvalue of $\mathcal{L}_{\delta}(\alpha^*) - (\omega^*_{\delta})^2  I$ and that the eigenspace is spanned by $u$. In view of the honeycomb symmetry of the coefficients in $\mathcal{L}_{\delta}$, one can check that 
\begin{equation}
\label{eq:RPCL}
	\mathcal{L}_{\delta}(\alpha^* )  \mathcal{R} = \mathcal{R} \mathcal{L}_{\delta}(\alpha^* ), \quad \quad \mathcal{L}_{\delta}(\alpha^* )  \mathcal{P}\mathcal{C} = \mathcal{P}\mathcal{C} \mathcal{L}_{\delta}(\alpha^* ).
\end{equation}
That is, both $\mathcal{R}$ and $\mathcal{PC}$ commute with $\mathcal{L}_\delta$. Hence, $\mathcal{R}u$ and $\mathcal{P} \mathcal{C} u$ are both in $N_\delta$, and one can find $\nu \in \{1 ,\tau , \overline{\tau}\}$ and $ \eta \in \{1,-1\}$ such that
\begin{equation*}
	\mathcal{R} u = \nu u, \quad \quad 	\mathcal{P} \mathcal{C} u = \eta u.
\end{equation*}
The fact that $[\mathcal{P}\mathcal{C},\mathcal{R}]  = 0$ and \eqref{eq:PCtao} then show that
\begin{equation*}
	\overline{\nu }(\eta u) =\overline{ \nu} 	\mathcal{P} \mathcal{C} u  = 	\mathcal{P} \mathcal{C} ( \nu u ) = \mathcal{P}\mathcal{C} \mathcal{R} u = \mathcal{R}\mathcal{P} \mathcal{C}u = \mathcal{R} (\eta u) = \nu( \eta u),
\end{equation*}
and therefore, $\nu =1$. 

To complete the proof by contradiction we show that $\mathcal{R} u = u$ is impossible for non-zero $u \in N_\delta$. Using the result from Step 1, we know that 
\begin{equation}\label{eq:diracmodeSL}
			u = a S_1 + b S_2 +  \mathcal{O}_{H^1(Y)} (\delta),
\end{equation}
and can assume that $|a|^2+|b|^2=1$, where $S_j = S_j^{\alpha^*}$'s are defined in \eqref{eq:Sjdef}. Since $S_1=1$ and $S_2=0$ in $D_1$, by the definition of $\mathcal{R}$, we have 
\begin{equation*}
u = a +\mathcal{O}(\delta) \quad \text{and} \quad \mathcal{R}u = e^{-\mi\alpha^*\cdot l_1} a + \mathcal{O}(\delta), \quad \text{in} \, D_1.
\end{equation*}
To have $u=\mathcal{R}u$ we must have $a  = \mathcal{O}(\delta)$. Similarly, $b =\mathcal{O}(\delta)$. As a result, we have $u=\mathcal{O}_{H^1(Y)}(\delta)$, which is a contradiction to our choice of $u$. 
	
\medskip

\emph{Step 3. We describe the required eigenmodes $S_{1,\delta}$ and $S_{2,\delta}$ at the Dirac point $\alpha^*$.} We can easily check that $\mathcal{L}_{\delta}(\alpha^* )  \mathcal{R} = \mathcal{R} \mathcal{L}_{\delta}(\alpha^* ) $, which implies that $N_{\delta} $ is an invariant subspace of $\mathcal{R}$. From the orthogonal  decomposition \eqref{eq:L2delta_orthogonal}, we obtain that for sufficiently small $\delta>0$ 
the following holds:
	\begin{equation}\label{ortho decom 2}
		N_{\delta} = \{N_{\delta} \cap L^2(\alpha^*,1) \} \oplus \{N_{\delta}\cap L^2(\alpha^*,\tau) \} \oplus \{N_{\delta}\cap L^2(\alpha^*,\overline{\tau}) \}.
	\end{equation}
For a proof we refer to \cite[Corollary 5.2]{MR4347316}.

In previous steps we have shown that $N_{\delta}\cap L^2(\alpha^*,1) $ is trivial and $\mathrm{dim}\,N_{\delta}=2$. It is not hard to verify that $\mathcal{PC}$ maps $L^2(\alpha^*,\tau)$ to $L^2(\alpha^*,\overline{\tau})$; see \eqref{eq:PCtao}. Hence, both $N_\delta \cap L^2(\alpha^*,\tau)$ and $N_\delta \cap L^2(\alpha^*,\overline{\tau})$ has dimension $1$, and if we choose $S_{1,\delta}$ from the first subspace with unit norm in $L^2(Y)$  and set $S_{2,\delta} = \mathcal{PC}S_{1,\delta}$, then automatically we have
	\begin{equation*}
		N_{\delta}\cap L^2(\alpha^*,\tau)  = \mathrm{Span}\,\{S_{1,\delta}\}, \quad \quad N_{\delta}\cap L^2(\alpha^*,\overline{\tau})  = \mathrm{Span}\,\{S_{2,\delta}\},
	\end{equation*}
and the required symmetries \eqref{eq:symSd12} hold immediately. We refer to \cite[Lemma 4.3]{ammari_honeycomb-lattice_2020} for an alternative proof of a choice $\{S_{1,\delta},S_{2,\delta}\}$ satisfying \eqref{eq:symSd12}.

It suffices to check that we can make sure that $\{S_{1,\delta},S_{2,\delta}\}$ also satisfy \eqref{eq:SjH1error}. By Step 1 (in particular, by \eqref{eq:Phialpha}) we can find $a,b$ such that 
\begin{equation}
\label{eq:S1delta}
	S_{1,\delta} = aS_1 + bS_2 + \mathcal{O}_{H^1(Y)}(\delta),
\end{equation}
where $S_1 = S_1^{\alpha^*}$ and $S_2 = S^{\alpha^*}_2$ are given by \eqref{eq:Sjdef}. We have the freedom to assume that $S_{1,\delta}$ is `normalized' in the sense that $|a|^2+|b^2|=1$.

Next, we derive some symmetry properties for $S_1,S_2$. This is most easily seen from the equation \eqref{eq:S0equation} that they satisfy. Explore the problem \eqref{eq:S0equation} under the honeycomb symmetry, we deduce that
\begin{equation}
\label{eq:S12PC}
	\mathcal{PC}S_1 = S_2, \quad \mathcal{PC}S_2 = S_1, 
\end{equation}
and
\begin{equation}
\label{eq:S12R}
	\mathcal{R}S_1 = \tau S_1, \quad \mathcal{R} S_2= \ol\tau S_2.
\end{equation}
Indeed, since $[\mathcal{PC},-\Delta] = [\mathcal{R},-\Delta]= 0$, one only needs to inspect how the boundary conditions on $\partial D_1$ and $\partial D_2$ change with respect to $\mathcal{PC}$ and $\mathcal{R}$ to verify the above results.

It follows that
\begin{equation*}
	S_{2,\delta} = \mathcal{PC} S_{1,\delta} = b S_1 + aS_2 + \mathcal{O}_{H^1(Y)}(\delta).
\end{equation*}
On the other hand, $S_1 \in L^2_\delta(\alpha^*,\tau)$ and $S_2\in L^2_\delta(\alpha^*,\ol \tau)$. Take $L^2_\delta(Y)$ inner product against $S_2$ on \eqref{eq:S1delta}. Note that $S_{1,\delta}$ and $S_1$ belong to $L^2_\delta(\alpha^*,\tau)$. In view of the orthogonality in \eqref{eq:L2delta_orthogonal}, we get
\begin{equation*}
	b \|S_2\|_{L^2_\delta(Y)}^2 = \mathcal{O}(\sqrt{\delta}).
\end{equation*}
This shows that $b=\mathcal{O}(\delta^{3/2})$. Plugging this estimate into \eqref{eq:S1delta}, since $S_{1,\delta}$ is `normalized', we get $a^2-1$ is small. We have the freedom to choose $a$ close to $1$, and further get $a-1=\mathcal{O}(\delta)$. We hence conclude that
\begin{equation*}
	S_{1,\delta} = S_1 + \mathcal{O}_{H^1(Y)}(\delta), \quad S_{2,\delta} = S_2 + \mathcal{O}_{H^1(Y)}(\delta).
\end{equation*}
This proves \eqref{eq:SjH1error} and completes the proof.
\end{proof}

\begin{remark} The modes $S_{1,\delta}, S_{2,\delta}$ found above satisfy the following orthogonal properties (they are clearly satisfied by the approximate modes $S_1,S_2$):
    \begin{equation}
        \label{eq: S_j,delta orthogonal}
        \int_{Y\setminus D} S_{1,\delta}(x) \ol{S_{2,\delta}(x)}\,dx= \int_D S_{1,\delta}(x) \ol{S_{2,\delta}(x)}\,dx =0.
    \end{equation}
To show this, consider $\mathcal{R}$ as an operator on the two-dimensional Hilbert space $N_{\delta}$ with inner product $\langle \cdot,\cdot\rangle_{L^2_{\delta'}(Y)}$ for any fixed $\delta' > 0$, it is easy to check that $\mathcal{R}\mathcal{R}^*=I$ and $\mathcal{R}^3 = I$, so $\mathcal{R}$ is unitary and has three eigenvalues $1,\tau,\tau^2 =\overline{\tau}$. By spectral theorem for normal (unitary) operators, one has the orthogonal decomposition 
$$N_{\delta} = \{N_{\delta} \cap L_{\delta'}^2(\alpha^*,1) \} \oplus \{N_{\delta}\cap L_{\delta'}^2(\alpha^*,\tau) \} \oplus \{N_{\delta}\cap L_{\delta'}^2(\alpha^*,\overline{\tau}) \}$$
under norm $\|\cdot\|_{L^2_{\delta'}(Y)}$. Then \eqref{eq: S_j,delta orthogonal} immediately follows from the above by the fact that we can choose $\delta'$ arbitrarily.
\end{remark}



\subsection{Expansion of Bloch eigenmodes near the Dirac point}

In this subsection, we explore the behavior of eigenmodes for resonant frequencies $\omega^\alpha$ where $\alpha$ is near the Dirac point $\alpha^*$. This can be thought as the augmentation of the conical behavior of $\omega^\alpha$ near $\omega^*$.

In \cite[Section 3.2]{ammari_high-frequency_2020}, the authors there essentially proved the following result.
\begin{theorem}\label{thm:ModesExpansion}
	Let $S_{1,\delta}$ and $S_{2,\delta}$ be the Bloch eigenmodes associated to the frequency $\omega^*_\delta$ at the Dirac point $\alpha^*$. For nearby points in $Y^*$, \emph{i.e.}, for $\alpha = \alpha^*+\beta$ for $|\beta|$ sufficiently small, we can find Bloch eigenmodes $u^\alpha_{j,\delta}(x)$, associated to the subwavelength bands $(\omega_{j,\delta}^{\alpha^*+\beta})^2$, for $j=1,2$, \emph{i.e.},
	\begin{equation*}
		\mathcal{L}_{\delta}(\alpha) u^{\alpha}_{j,\delta} = \omega^2_{j,\delta}(\alpha )  u^{\alpha}_{j,\delta} ,\quad \quad j=1,2,
	\end{equation*}
	whose $L^2(Y)$ norms are comparable to $1$ and, more importantly, they satisfy the following expansion:
	\begin{equation}\label{eq:ModesExpan}
		\begin{aligned}
			& u^\alpha_{1,\delta}(x) =\frac{A(\beta)}{\sqrt{2}} e^{ \mathrm{i}x\cdot \beta} S_{1,\delta}(x) - \frac{1}{\sqrt{2}} e^{\mathrm{i}x\cdot \beta}  S_{2,\delta}(x) +\mathcal{O}_{H^1(Y)}(\delta+|\beta|) , \\
			&u^\alpha_{2,\delta}(x) =\frac{A(\beta)}{\sqrt{2}} e^{\mathrm{i}x\cdot \beta} S_{1,\delta}(x) + \frac{1}{\sqrt{2}}  e^{\mathrm{i}x\cdot \beta}  S_{2,\delta}(x) +\mathcal{O}_{H^1(Y)}(\delta+|\beta|),
		\end{aligned}
	\end{equation}
	where
	\begin{equation}
 \label{eq:Abetadef}
		A(\beta) = \frac{c}{|c|} \frac{\beta_1 -\mathrm{i} \beta_2 }{|\beta|},
	\end{equation}
	and $c$ is defined in \eqref{eq:cdef}.
\end{theorem}

\begin{proof}[Proof of Theorem \ref{thm:ModesExpansion}] In item one of Theorem \ref{thm:DiracModes} we have seen that near the Dirac point $\alpha^*$, there are two subwavelength bands (resonant frequencies) $\omega_-^\alpha = \omega_{1,\delta}(\alpha)$ and $\omega_+^\alpha = \omega_{2,\delta}(\alpha)$, and they intersect at the Dirac point $\alpha^*$ (and also at $\alpha^*_2$).
The goal here is to show that we can choose eigenpairs $(u_{1,\delta}^\alpha,(\omega_-^\alpha)^2)$ and $(u_{2,\delta}^\alpha,(\omega_+^\alpha)^2)$ to the eigenvalue problem \eqref{eq:egproblem} such that \eqref{eq:ModesExpan} holds.

We now follow Step 1 in the proof of Theorem \ref{thm:DiracModes} to find the eigenfunctions of \eqref{eq:egproblem} associated to the eigenvalues $(\omega_\pm^\alpha)^2$. According to the result there, let $S^\alpha_1$ and $S^\alpha_2$ be defined by \eqref{eq:Sjdef}, then any eigenmode $u^\alpha_+$ (respectively, $u^\alpha_-$) corresponding to $(\omega_+^\alpha)^2$ (respectively, $(\omega_-^\alpha)^2$) must be of the form (see \eqref{eq:Phialpha})
\begin{equation*}
	u^\alpha = a S^\alpha_1 + b S^\alpha_2 + \mathcal{O}_{H^1(Y)}(\delta).
\end{equation*}
Here, we used the fact that $(\omega_\pm^\alpha)^2 =\mathcal{O}(\delta)$. Moreover, the vector $(a,b)$ satisfies the $2$ by $2$ system \eqref{eq:omega_ab}, which now becomes
\begin{equation*}
	\begin{pmatrix} \delta c_1^\alpha - (\omega_\pm^\alpha)^2|D_1| & \delta c^\alpha_2 \\ \overline{c_2^\alpha} & \delta c_1^\alpha - (\omega_\pm^\alpha)^2|D_1| \end{pmatrix} \begin{pmatrix} a_\pm \\ b_\pm \end{pmatrix} = \mathcal{O}(\delta^2).
\end{equation*}
We expand the entries in this system, from the differentiability of the capacitance matrix $C^\alpha$ (in particular by \eqref{eq:gradC}) and the fact that $C^{\alpha^*}$ is $c_1^{\alpha^*}I$,
\begin{equation*}
	c_1^\alpha = c_1^{\alpha^*} + \mathcal{O}(|\beta|^2),\quad  c_2^\alpha = c(\beta_1-\mi \beta_2) + \mathcal{O}(|\beta|^2), \quad \text{as }\, |\beta| \to 0.
\end{equation*}
Here, $\beta = (\beta_1,\beta_2)$ in standard coordinate system. The entry on the diagonal of the 2 by 2 system can be written as
\begin{equation*}
  	\delta c_1^{\alpha^*} - (\omega^*_\delta)^2|D_1| + \left((\omega^*_\delta)^2 - (\omega^\alpha_\pm)^2\right)|D_1| + \mathcal{O}(\delta|\beta|^2).
\end{equation*}  
Using also the estimate \eqref{eq:omegastarsquare}, we see that the system for $(a_\pm,b_\pm)$ has the form
\begin{equation*}
	\begin{pmatrix} -|D_1|((\omega^\alpha_\pm)^2-(\omega^*_\delta)^2) + \mathcal{O}(\delta|\beta|^2) & \delta c(\beta_1-\mi \beta_2) + \mathcal{O}(\delta|\beta|^2)\\ \delta \ol{c}(\beta_1+\mi \beta_2) + \mathcal{O}(\delta|\beta|^2) & -|D_1|((\omega^\alpha_\pm)^2-(\omega^*_\delta)^2) + \mathcal{O}(\delta|\beta|^2) \end{pmatrix} \begin{pmatrix} a_\pm \\ b_\pm \end{pmatrix} = \mathcal{O}(\delta^2).
\end{equation*}
We see that the eigenvalues $(\omega^\alpha_\pm)^2$ should satisfy
\begin{equation*}
	(\omega^\alpha_\pm)^2 = (\omega^*_\delta)^2 \pm \delta |c||\beta||D_1|^{-1} + \mathcal{O}(\delta^2 + \delta|\beta|^2).
\end{equation*}
From this, we can identify the main term in the asymptotic behavior \eqref{eq:cone}. Moreover, we can select the normalized solution to the system with $\omega^\alpha_+$ and, respectively, with $\omega^\alpha_-$, as follows:
\begin{equation*}
	\begin{pmatrix} \displaystyle\frac{c(\beta_1-\mi \beta_2)}{\sqrt{2}|c||\beta|} \\ \\ \frac{1}{\sqrt{2}} \end{pmatrix} + \mathcal{O}(\delta+|\beta|^2) \quad
	\text{and, respectively, } \quad
	\begin{pmatrix} \displaystyle\frac{c(\beta_1-\mi \beta_2)}{\sqrt{2}|c||\beta|} \\ \\  -\frac{1}{\sqrt{2}} \end{pmatrix} + \mathcal{O}(\delta+|\beta|^2).
\end{equation*}
This shows that we can find eigenmodes for $\alpha = \alpha^* + \beta$ and $|\beta|$ small:
\begin{equation}
\label{eq:ModesExp2}
	\begin{aligned}
		u^\alpha_{1,\delta} &= \frac{A(\beta)}{\sqrt{2}} S^\alpha_1 - \frac{1}{\sqrt{2}} S^\alpha_2 + \mathcal{O}_{H^1(Y)}(\delta+|\beta|^2),\\
		u^\alpha_{2,\delta} &=\frac{A(\beta)}{\sqrt{2}} S^\alpha_1 + \frac{1}{\sqrt{2}} S^\alpha_2 + \mathcal{O}_{H^1(Y)}(\delta+|\beta|^2).
	\end{aligned}
\end{equation}
where $A(\beta)$ is defined in \eqref{eq:Abetadef} and $S^\alpha_j = \mathcal{S}_D^{\alpha,0}[\psi^\alpha_j]$, $j=1,2$, are defined by \eqref{eq:Sjdef}. 

To achieve \eqref{eq:ModesExpan} it remains to replace $S^\alpha_j$ by $e^{\mi x\cdot\beta} S_{j,\delta}$, and control the error. In \cite{ammari_high-frequency_2020}, it was proved (see Lemma 3.1 there) that in a neighborhood of $\alpha^*$, \emph{i.e.}, for $\alpha = \alpha^* + \beta$ with $|\beta|$ small, and for $j=1,2$,
\begin{equation}
\label{eq:SalphaExp}
\begin{aligned}
S^\alpha_j &= \mathcal{S}^{\alpha,0}_D[\psi_j^\alpha] = e^{\mi \beta \cdot x} \mathcal{S}_D^{\alpha^*,0}[\xi^\beta_j] + \mathcal{O}_{H^1(Y)}(|\beta|^2),\\
	\xi_j^\beta &= (\mathbb{S}_D^{\alpha^*,0})^{-1}[e^{-\mi \beta\cdot y} \mathbf{1}_{\partial D_j}(y)].
	\end{aligned}
\end{equation}
In the second line above, by an abuse of notation, $\mathbb{S}_D^{\alpha^*,0}$ refers to the trace on $\partial D$ of the single layer operator $\mathcal{S}_D^{\alpha^*,0}$. In \cite[Equation (3.4), Remark 1]{ammari_high-frequency_2020}, $\xi_j^\beta$ is shown to be $\psi_j^{\alpha^*}+\mathcal{O}_{L^2(\partial D)}(|\beta|)$. Hence, we get
\begin{equation*}
	S_j^\alpha = e^{\mi\beta\cdot x} S_j + \mathcal{O}_{H^1(Y)}(|\beta|), \qquad j=1,2.
\end{equation*}
In view of \eqref{eq:SjH1error}, we could replace $S_j$ by $S_{j,\delta}$ above with the error term replaced by $\mathcal{O}_{H^1(Y)}(\delta+|\beta|)$. Plugging this estimate into \eqref{eq:ModesExp2}, we finish the proof.
\end{proof} 

\begin{remark}
	For the analysis of wavepackets propagation in this paper, in particular in order to make use of the Floquet-Bloch theory associated to $\mathcal{L}_\delta$ reviewed in the previous section, we need to use eigenmodes normalized in the weighted space $L^2_\delta(Y)$. Note that for $u^\alpha_{j,\delta}$'s, their $L^2_\delta(Y)$ norm is of order $\delta^{-\frac12}$ due to the weights. By setting
		\begin{equation}
		\label{eq:Phi12choice}
		\Phi_{j,\delta}(x,\alpha^*) = \frac{S_{j,\delta}}{\|S_{j,\delta} \|_{L^2_{\delta}(Y)}}, \quad \quad  \Phi_{j,\delta}(x,\alpha)  = \frac{u^{\alpha}_{j,\delta}}{\|u^{\alpha}_{j,\delta} \|_{L^2_{\delta}(Y)}},
	\end{equation}
for $j=1,2$, we can use them as the orthonormal basis of the eigenspace associated the problem \eqref{eq:egproblem} with eigenvalue $(\omega^\alpha_j)^2$. Then we can rewrite \eqref{eq:ModesExpan} in the following form:
	\begin{equation}\label{eq:ModesExpan1}
	\begin{aligned}
		& \Phi_{1,\delta}(x,\alpha) =\frac{A(\beta)}{\sqrt{2}} e^{\mathrm{i}x\cdot \beta}  \Phi_{1,\delta}(x,\alpha^*)(x) - \frac{1}{\sqrt{2}}e^{\mathrm{i}x\cdot \beta}  \Phi_{2,\delta}(x,\alpha^*) +\mathcal{O}_{H_{\delta}^1(Y)}(\delta+|\beta|) , \\
		&\Phi_{2,\delta}(x,\alpha) =\frac{A(\beta)}{\sqrt{2}} e^{\mathrm{i}x\cdot \beta}  \Phi_{1,\delta}(x,\alpha^*)(x) + \frac{1}{\sqrt{2}}e^{\mathrm{i}x\cdot \beta}  \Phi_{2,\delta}(x,\alpha^*) +\mathcal{O}_{H_{\delta}^1(Y)}(\delta+|\beta|).
	\end{aligned}
\end{equation} 
The Bloch modes $\Phi_{j}(x,\alpha)$ can also be described as follows. Let $p_j(x,\alpha) = e^{-\mi \alpha \cdot x} \Phi_j(x,\alpha)$. Then $p_j(\cdot,\alpha)$ is $\Lambda$-periodic in $x$, and solves
\begin{equation}
\begin{aligned}
&\widetilde{\mathcal{L}_\delta}(\alpha)[p_j] = \omega^2_j(\alpha) p_j,\\
&p_j(\cdot,\alpha) \in L^2_{\delta,\Lambda}(\R^2).
\end{aligned}
\end{equation}
Here, $\wtcL(\alpha) = e^{-\mi \alpha\cdot x}\mathcal{L}_\delta e^{\mi \alpha \cdot x}$, \emph{i.e.},
\begin{equation*}
	\wtcL(\alpha)[\phi] = -\sigma_\delta (\nabla+\mi \alpha)\cdot\left(\sigma_\delta^{-1} (\nabla+\mi \alpha) \phi\right).
\end{equation*}
Although the operator now varies with $\alpha$, all $p_{j}$'s live in the same space $H^1_{\delta,\Lambda}$. In terms of the $\Lambda$-periodic functions $p_{1,\delta}$ and $p_{2,\delta}$, the expansion of the eigenmodes near the Dirac point $\alpha^*$ can be written as
\begin{equation}\label{eq:p12expansion}
	\begin{aligned}
		& p_{1,\delta}(x,\alpha) =\frac{A(\beta)}{\sqrt{2}} p_{1,\delta}(x,\alpha^*)(x) - \frac{1}{\sqrt{2}}p_{2,\delta}(x,\alpha^*) +\mathcal{O}_{H_{\delta}^1(Y)}(\delta+|\beta|) , \\
		&p_{2,\delta}(x,\alpha) =\frac{A(\beta)}{\sqrt{2}} p_{1,\delta}(x,\alpha^*)(x) + \frac{1}{\sqrt{2}} p_{2,\delta}(x,\alpha^*) +\mathcal{O}_{H_{\delta}^1(Y)}(\delta+|\beta|).
	\end{aligned}
\end{equation}
\end{remark}

To understand the dependence of $p_j(\cdot,\alpha)$ on  $\alpha$, we expand the operator $\wtcL(\alpha)$ for $\alpha$ near any point $\alpha_0$, \emph{i.e.}, $\wtcL(\alpha_0+\kappa)$ for $\kappa$ with small amplitude, and find
\begin{equation}
\label{periodic L}
\wtcL(\alpha_0+\kappa) = \wtcL(\alpha_0) - \mi\kappa \cdot \wtcA(\alpha_0) + |\kappa|^2,
\end{equation}
where for each $\alpha$, 
\begin{equation*}
\wtcA(\alpha)\phi = (\nabla + \mi\alpha)\phi + \sigma_\delta (\nabla + \mi\alpha)\left(\frac{1}{\sigma_\delta}\phi\right).
\end{equation*}
Here, $\wtcA(\alpha)$ is understood as an operator on $H^1_{\delta,\Lambda}$ with value in $H^{-1}_{\delta,\Lambda}$, and $\wtcA(\alpha)\phi$ acts on $\psi$ by
\begin{equation}
\label{eq:Atildepair}
\langle \wtcA(\alpha)\phi, \psi\rangle_{H^{-1}_{\delta,\Lambda},H^1_{\delta,\Lambda}} = \int_{Y} \frac{1}{\sigma_\delta} \left((\nabla + \mi\alpha)\phi \overline{\psi} - \phi\overline{(\nabla+\mi\alpha)\psi}\right)\,dx.
\end{equation}
The following estimate is an immediate consequence of the definition above:
\begin{equation}
\label{eq:cAHm1}
\left|\langle \wtcA(\alpha) \phi,\psi\rangle_{H^{-1}_{\delta,\Lambda},H^1_{\delta,\Lambda}} \right| \le C\left(\|\phi\|_{L^2_\delta(Y)}\|\psi\|_{H^1_\delta(Y)} + \|\phi\|_{H^1_\delta(Y)}\|\psi\|_{L^2_\delta(Y)}\right), \qquad \forall\,\alpha \in Y^*.
\end{equation}
Note also that by the definition of $H^{-1}_\delta$-$H^1_\delta$ pairings, for any $j,\ell \in \{1,2\}$,
\begin{equation}
\label{eq:pjkpairing}
	\langle \wtcA(\alpha^*)p_{j,\delta},p_{\ell,\delta}\rangle = \langle (\nabla + \mi \alpha^*)p_{j,\delta}, p_{\ell,\delta}\rangle_{L^2_\delta(Y)} - \langle p_{j,\delta}, (\nabla + \mi \alpha^*)p_{\ell,\delta}\rangle_{L^2_\delta(Y)}.
\end{equation}

The next theorem turns out very important for the derivation of the effective dynamics of the wave packets propagation problem \eqref{problem 1}. The quantity $a_\delta$ there determines the coefficient in the effective Dirac system \eqref{eq:Dirac}. It shows that the Dirac cone relation at the Dirac point $\alpha^*$ is  captured by the operator $\wtcA(\alpha^*)$. 

\begin{theorem}\label{thm:derivative A} For $j=1,2,$ let $p_{j,\delta} = \Phi_{j,\delta}(\cdot,\alpha^*)e^{-\mi \alpha^*\cdot x}$. Then, the following holds:
	\begin{equation}
	\label{eq:keypair}
		\begin{aligned}
			& 	\big\langle   \wtcA(\alpha^*) p_{1,\delta}, p_{1,\delta} \big\rangle =	\big\langle   \wtcA(\alpha^*) p_{2,\delta}, p_{2,\delta} \big\rangle  =0 ,\\
			& 	\big\langle   \wtcA(\alpha^*) p_{2,\delta}, p_{1,\delta} \big\rangle  = - \overline{\big\langle   \wtcA(\alpha^*) p_{1,\delta}, p_{2,\delta} \big\rangle}=  a_{\delta} \begin{pmatrix}
				1 \\ 
				\mathrm{i}
			\end{pmatrix}.
		\end{aligned}
	\end{equation}
Here, all pairings are understood as $\langle\cdot,\cdot\rangle_{H^{-1}_\delta(Y),H^1_\delta(Y)}$. The constant $a_\delta$ is given by
\begin{equation}\label{eq:adelta}
	a_{\delta} =  \frac{\mathrm{i}c}{|D_1|} \delta+ \mathcal{O}(\delta^{\frac{3}{2}}) , 
\end{equation}
and $c$ defined by \eqref{eq:cdef} is nonzero.
\end{theorem}

Note that by the relation between $\Phi_{j,\delta}$ and $p_{j,\delta}$ and by the formula \eqref{eq:pjkpairing}, the results above says
\begin{equation*}
\begin{aligned}
    \langle \Phi_{1,\delta}(\cdot,\alpha^*),\nabla\Phi_{1,\delta}(\cdot,\alpha^*)\rangle_{L^2_\delta(Y)} &= \langle \Phi_{2,\delta}(\cdot,\alpha^*),\nabla\Phi_{2,\delta}(\cdot,\alpha^*)\rangle_{L^2_\delta(Y)} = \begin{pmatrix}0\\ 0\end{pmatrix}, \\
    -2\langle \Phi_{2,\delta}(\cdot,\alpha),\nabla \Phi_{1,\delta}(\cdot,\alpha)\rangle_{L^2_\delta(Y)} &= a_{\delta} \begin{pmatrix}
				1 \\ 
				\mathrm{i}
			\end{pmatrix}.
    \end{aligned}
\end{equation*}
The fact that $a_\delta \ne 0$ can be interpreted as a non-degeneracy Dirac cone condition, consistent with such non-degeneracy condition in \cite{MR2947949}.

\begin{remark} We emphasize that $\wtcA(\alpha)$ is, in some sense, the derivative of $\wtcL(\alpha)$ with respect to $\alpha$. Indeed, this is most clearly seen from the expansion formula \eqref{periodic L}. The advantage of introducing $\wtcL(\alpha)$ there, other than treating directly the operator $\mathcal{L}_\delta$ in the $\alpha$-varying space $H^1_{\delta,\Lambda^\alpha}$, is to make this point clear. Due to the discontinuity of $\sigma_\delta$, we have to understand $\wtcA(\alpha)$ as an operator from $H^1_\delta(\alpha)$ to $H^{-1}_\delta(\alpha)$, for each $\alpha \in Y^*$.

We also record here the following useful formula: for sufficiently regular $f$ and for $g \in H^1_\delta(\R^2)$, 
\begin{equation}
\label{action fg}
\begin{aligned}
    \mathcal{L}_\delta(fg) - f\mathcal{L}_\delta(g) = -\nabla f\cdot \nabla g -\sigma_\delta \nabla \cdot(\sigma_\delta^{-1} g\nabla f). 
    \end{aligned}
\end{equation}
The above identity holds in $H^{-1}_\delta(\R^2)$, and the right-hand side acts on $\psi \in H^1_\delta(\R^2)$ by
\begin{equation*}
\langle -\nabla f\cdot \nabla g -\sigma_\delta \nabla \cdot(\sigma_\delta^{-1} g\nabla f),\psi\rangle = \int_{\R^2} \frac{1}{\sigma_\delta} \left[g\nabla f\cdot \nabla \overline{\psi} - (\nabla f\cdot \nabla g)\overline{\psi}\right]\,dx.
\end{equation*}
\end{remark}

\begin{proof}[Proof of Theorem \ref{thm:derivative A}]
	Recall the expression \eqref{eq:pjkpairing} for the $H^{-1}_\delta(Y)$-$H^1_\delta(Y)$ pairing of $p_{j,\delta}$'s. The first line of \eqref{eq:keypair}, as well as the first equality in the second line, follow immediately.

	It remains to compute the quantity $\mathbf{a}_\delta := \langle \wtcA(\alpha^*)p_{2,\delta},p_{1,\delta}\rangle$, which is a vector in $\bC^2$. By the relation between $p_{j,\delta}$ and $\Phi_{j,\delta}$, we have
	\begin{equation*}
		\mathbf{a}_\delta = \langle \nabla \Phi_{2,\delta}(\cdot,\alpha^*), \Phi_{1,\delta}(\cdot,\alpha^*)\rangle_{L^2_\delta(Y)} - \langle \Phi_{2,\delta}(\cdot,\alpha^*), \nabla \Phi_{1,\delta}(\cdot,\alpha^*)\rangle_{L^2_\delta(Y)}.
	\end{equation*}
Below we work with $S_{1,\delta}$ and $S_{2,\delta}$ rather than $\Phi_{j,\delta}(\cdot,\alpha^*)$. After factorizing by a constant, we obtain that
	\begin{equation*}
		\|S_{1,\delta}\|^2_{L^2_\delta(Y)} \mathbf{a}_{\delta} = \langle \nabla S_{2,\delta},S_{1,\delta}\rangle_{L^2_\delta(Y)} - \langle S_{2,\delta},\nabla S_{1,\delta}\rangle_{L^2_\delta(Y)}.
	\end{equation*}
Note that
\begin{equation*}
	\|S_1\|_{L^2_\delta(Y)}^2 = \|S_2\|_{L^2_\delta(Y)}^2 = \delta^{-1}|D_1| + \mathcal{O}(1).
\end{equation*}
By \eqref{eq:SjH1error} the above holds for $S_{j,\delta}$, $j=1,2$, as well. For simplicity the subscript $L^2_\delta(Y)$ is henceforth omitted from the bracket of the pairings above. By \eqref{eq:SjH1error} again and Remark \ref{rem:SHdel}, we get
\begin{equation*}
	\|S_{1,\delta}\|^2_{L^2_\delta(Y)} \mathbf{a}_\delta = \mathbf{b} + \mathcal{O}(\sqrt{\delta}), \quad \text{where} \quad \mathbf{b} := \langle \nabla S_2,S_1\rangle - \langle S_2,\nabla S_1\rangle.
\end{equation*}
We will show in Theorem \ref{thm:gradalphaC12} that $\mathbf{b} = \mi c(1,\mi)^\top$, where the superscript $\top$ denotes the transpose. Hence, we get
\begin{equation*}
	\mathbf{a}_\delta = \frac{\mi c\delta}{|D_1|} \begin{pmatrix}
	1\\ \mi 
	\end{pmatrix} + \mathcal{O}(\delta^{\frac32}).
\end{equation*}

Although the above characterization of $\mathbf{a}_\delta$ is sufficient for our analysis of the wavepacket problem later, we do prove that the second statement \eqref{eq:keypair} holds. For doing so, it suffices to show that $R\mathbf{a}_\delta = \tau \mathbf{a}_\delta$, where $\tau = e^{\mi \frac{2\pi}{3}}$, because $(1,\mi)^\top$ is a basis for the eigenspace of $R$ associated to the eigenvalue $\tau$.

Recall that $\mathcal{R}S_{1,\delta} = \tau S_{1,\delta}$ and $\mathcal{R}S_{2,\delta} = \ol\tau S_{2,\delta}$. By a direct computation, $\mathcal{R}\nabla = R^*\nabla \mathcal{R}$, where $R^*$ is the adjoint of the rotation matrix $R$. Hence,
	\begin{equation*}
	\begin{aligned}
		\langle \nabla S_{2,\delta},S_{1,\delta}\rangle &= \langle \tau \nabla \mathcal{R}S_{2,\delta}, S_{1,\delta}\rangle =  \langle  \nabla \mathcal{R}S_{2,\delta}, \ol\tau S_{1,\delta}\rangle =  \langle  \nabla \mathcal{R}S_{2,\delta}, \ol\tau^2\mathcal{R} S_{1,\delta}\rangle = \tau^2 \langle  R\mathcal{R}\nabla S_{2,\delta}, \mathcal{R} S_{1,\delta}\rangle\\
		&= \ol\tau \left( R\int_{Y_1} \frac{1}{\sigma_\delta(y)} \nabla S_{2,\delta}(R_1 y)\overline{S_{1,\delta}(R_1 y)}\,dy  + R\int_{Y_2} \frac{1}{\sigma_\delta(y)} \nabla S_{2,\delta}(R_2 y)\overline{S_{1,\delta}(R_2 y)}\,dy\right)\\
		&=\ol\tau R \langle \nabla S_{2,\delta},S_{1,\delta}\rangle.
		\end{aligned}
	\end{equation*}
That is, $\langle \nabla S_{2,\delta},S_{1,\delta}\rangle$ is an eigenvector of $R$ associated to eigenvalue $\tau$. On the other hand, we have $S_{2,\delta} = \mathcal{PC}S_{1,\delta}$ and $\mathcal{PC}\nabla = -\nabla\mathcal{PC}$. Hence, we get
	\begin{equation*}
	\begin{aligned}
		\langle \nabla S_{2,\delta}, S_{1,\delta}\rangle &= \langle \nabla \mathcal{PC} S_{1,\delta}, \mathcal{PC} S_{2,\delta}\rangle = \langle \mathcal{PC}\nabla S_{1,\delta},\mathcal{PC} S_{2,\delta}\rangle \\
		&= -\int_Y \frac{1}{\sigma_\delta}\overline{\nabla S_{1,\delta}(2x_0 -y)} S_{2,\delta}(2x_0-y)\,dy\\
		&= - \langle S_{2,\delta},\nabla S_{1,\delta}\rangle.
		\end{aligned}
	\end{equation*}
	Hence $\|S_{1,\delta}\|^2_{L^2_\delta(Y)}\mathbf{a}_\delta = 2\langle \nabla S_{2,\delta},S_{1,\delta}\rangle$.
It follows that $R\mathbf{a}_\delta = \tau \mathbf{a}_\delta$, and hence the second statement in \eqref{eq:keypair} with $a_\delta$ being given by \eqref{eq:adelta} holds.
\end{proof}

\begin{lemma}\label{thm:gradalphaS} For each $\alpha \in Y^*$, let $S_1^\alpha$ and $S_2^\alpha$ be defined by \eqref{eq:Sjdef}. Then, for $j=1,2$,
\begin{equation}
	\label{eq:gradalphaS}
\nabla_\alpha S_j^\alpha \rvert_{\alpha = \alpha^*} = \mi \left(xS_j - \mathbf{W}_j\right),
\end{equation}
where $\mathbf{W}_j = \mathcal{S}_D^{\alpha^*,0}[(\mathbb{S}_D^{\alpha^*,0})^{-1} y\mathbf{1}_{\partial D_j}(y)]$, or, equivalently, $\mathbf{W}_j$ is a $\bC^2$-valued function determined by
\begin{equation*}
\left\{
	\begin{aligned}
		&-\Delta \mathbf{W}_j = 0 \quad \text{in }\; Y\setminus (\partial D),\\
&\mathbf{W}_j \text{  is $\alpha^*$-quasiperiodic},\\
&\mathbf{W}_j = y \delta_{ij} \quad \text{on $D_i$, $i=1,2$.}
	\end{aligned}
	\right.
\end{equation*}
\end{lemma}

\begin{proof} To compute $\nabla_\alpha S^\alpha_j$ at $\alpha=\alpha^*$, we use the formula \eqref{eq:SalphaExp} obtained earlier. From the expressions there we have, for $\alpha = \alpha^*+\beta$, and $|\beta|\to 0$,
\begin{equation*}
\begin{aligned}
	\xi^\beta_j &= \psi_j^{\alpha^*} + (\mathbb{S}_D^{\alpha^*,0})^{-1}[(e^{-\mi \beta\cdot y}-1) \mathbf{1}_{\partial D_j}(y)]\\
	&= \psi_j^{\alpha^*} -\mi \beta \cdot (\mathbb{S}_D^{\alpha^*,0})^{-1}[y \mathbf{1}_{\partial D_j}(y)] + \mathcal{O}_{L^2(\partial D)}(|\beta|^2).
\end{aligned}
\end{equation*}
Plugging the above expansion to the expression of $S^\alpha_j$ into \eqref{eq:SalphaExp}, we get
\begin{equation*}
\begin{aligned}
	S^\alpha_j &= e^{\mi \beta\cdot x}S_j(x) - e^{\mi \beta\cdot x} \mi \beta\cdot \mathcal{S}_D^{\alpha^*,0}(\mathbb{S}_D^{\alpha^*,0})^{-1}[y \mathbf{1}_{\partial D_j}(y)]) + \mathcal{O}_{H^1(Y)}(|\beta|^2)\\
	&= S_j + \mi \beta \cdot xS_j - \mi \beta\cdot \mathbf{W}_j +  \mathcal{O}_{H^1(Y)}(|\beta|^2).
	\end{aligned}
\end{equation*}
Hence, we have proved \eqref{eq:gradalphaS}.
\end{proof}

\begin{theorem}\label{thm:gradalphaC12} Under the honeycomb symmetry \eqref{eq:hcsym} of $D$, for the off-diagonal entry of the capacitance matrix $(C_{ij}(\alpha))$, we have 
\begin{equation}
	\label{eq:gradalphaC12}
\nabla_\alpha C_{12}(\alpha)\Big\rvert_{\alpha=\alpha^*} = \begin{pmatrix}
0 & -1\\ -1 & 0
\end{pmatrix}
\int_{Y\setminus \ol D} \nabla S_2(x)\ol{S_1(x)}-S_2(x)\ol{\nabla S_1(x)}\,dx. 
\end{equation}
Moreover, the vector above has the form $c(1,-\mi)^\top$, with $c$ given by
\begin{equation}
\label{eq:c_formula2}
	c = \frac{\mi\sqrt{3}L}{2} \int_{\partial Y \cap \partial Y_1} \ol{S_1(y)} \frac{\partial S_2}{\partial \nu}(y) -  \ol{\frac{\partial S_1}{\partial \nu}(y)} S_2(y)\,dy.
\end{equation}
\end{theorem}
\begin{proof} Note that in this proof, in standard coordinate system we write $x=(x_1,x_2)$, $y=(y_1,y_2)$ for points $x,y\in \R^2$. This should not be confused with the centers $x_1,x_2$ and $x_0$ of $D_1,D_2$ and $Y$.

\emph{Step 1}. By the definition of $C_{12}(\alpha)$, we have
\begin{equation*}
	\nabla_\alpha C_{12}(\alpha) = \int_{Y\setminus \ol D}
	\ol{\nabla_x \nabla_\alpha S^\alpha_1}\cdot \nabla_x S_2^\alpha \,dx + \int_{Y\setminus \ol D}
	\ol{\nabla_x S^\alpha_1}\cdot \nabla_x \nabla_\alpha S_2^\alpha \,dx .
\end{equation*}
It is known (by symmetry of $C^\alpha)$ that $\nabla_\alpha C_{12}(\alpha)\rvert_{\alpha=\alpha^*}$ can be written as $c(1,-\mi)^\top$. By evaluating the above at $\alpha = \alpha^*$, taking the first component, and using \eqref{eq:gradalphaS}, we get
\begin{equation*}
	c:= \frac{\partial c_1^\alpha}{\partial \alpha_1}\Big\rvert_{\alpha=\alpha^*} = -\mi\int_{Y\setminus \ol D}
	\ol{\nabla_x (x_1 S_1 - \mathbf{W}_1^{(1)})}\cdot \nabla_x S_2 \,dx + \mi\int_{Y\setminus \ol D}
	\ol{\nabla_x S_1}\cdot \nabla_x (x_1S_2-\mathbf{W}_2^{(1)}) \,dx .
\end{equation*}
Here, $\mathbf{W}_j^{(1)}$ is the first component of the vector valued function $\mathbf{W}_j$, $j=1,2$. By definition, $\mathbf{W}_j^{(1)}$ is $\alpha^*$-quasiperiodic, and moreover,
\begin{equation*}
	x_1 S_j - \mathbf{W}_j^{(1)} = 0 \quad \text{on $\partial D$}, \quad j=1,2.
\end{equation*}
Using integration by parts and the fact that $S_j$'s are harmonic on $Y\setminus \partial D$, and using the above observation also, we get
\begin{equation*}
	c = \mi\int_{\partial Y} (y_1S_2(y)-\mathbf{W}_2^{(1)})\ol{\frac{\partial S_1}{\partial \nu}(y)}\,dy - \mi\int_{\partial Y} (y_1\ol{S_1}(y)-\ol{\mathbf{W}_1^{(1)}}) \frac{\partial S_2}{\partial \nu}(y)\,dy.
\end{equation*}
Note that for any pair of $\alpha^*$-quasiperiodic $H^1(Y)$ functions $u,v$ we always have
\begin{equation*}
	\int_{\partial Y} u(y)\ol{\frac{\partial v}{\partial \nu}(y)}\,dy = 0.
\end{equation*}
We simplify the expression of $c$ further to obtain that
\begin{equation}
\label{eq:c_formula1}
	c = \mi\int_{\partial Y} y_1S_2(y)\ol{\frac{\partial S_1}{\partial \nu}(y)}\,dy - \mi\int_{\partial Y} y_1\ol{S_1(y)} \frac{\partial S_2}{\partial \nu}(y)\,dy .
\end{equation}

\medskip

\emph{Step 2}. Next, we verify the expression \eqref{eq:gradalphaC12}. Let
\begin{equation*}
	\mathbf{b} = \langle \nabla S_2,S_1\rangle - \langle S_2,\nabla S_1\rangle = \int_{Y\setminus \ol D} \nabla S_2(x)\ol{S_1(x)} - S_2(x) \ol{\nabla S_1(x)}\,dx ,
\end{equation*}
as defined in the proof of Theorem \ref{eq:keypair}. By the same argument there that leads to $R\mathbf{a}_\delta = \tau \mathbf{a}_\delta$, $\tau = e^{\mi 2\pi/3}$, we verify that $R \mathbf{b} =\tau \mathbf{b}$ as well. Hence, $\mathbf{b} = b(1,\mi)^\top$, where $b$ is the first component of $\mathbf{b}$, \emph{i.e.},
\begin{equation*}
	b = \int_{Y\setminus \ol D} \partial_{1} S_2(x) \overline{S_1(x)}\,dx - \int_{Y\setminus \ol D} S_2(x)\overline{\partial_1 S_1(x)}\,dx.
\end{equation*}
Writing $\partial_1 = \partial_{x_1} = (\nabla x_1)\cdot \nabla$, we compute
\begin{equation*}
\begin{aligned}
	b &= \int_{Y\setminus \ol D} \nabla x_1 \cdot \left((\nabla S_2(x))\ol{S_1(x)} - S_2(x)\ol{\nabla S_1(x)}\right)\,dx\\
	&= \int_{Y\setminus \ol D} \nabla \cdot\left[x_1 \left(\nabla S_2\ol{S_1} - S_2\ol{\nabla S_1}\right)\right]\,dx\\
	&= \int_{\partial (Y\setminus \ol D)} y_1\left(\frac{\partial S_2}{\partial \nu}(y)\ol{S_1(y)} - S_2(y)\ol{\frac{\partial S_1}{\partial \nu}(y)}\right)\,dy. 
	\end{aligned}
\end{equation*}
Here, for the second line we have used integration by parts and the fact that $\Delta S_j = 0$, $j=1,2$, in $Y\setminus \ol D$. The boundary $\partial(Y\setminus \ol D)$ is a disjoint union of $\partial Y$, $\partial D_1$ and $\partial D_2$. Since $S_j=\delta_{ij}$ on $\partial D_i$, for $i,j \in \{1,2\}$, the contributions in the last line from $\partial D_1$ and $\partial D_2$ amount to
\begin{equation*}
	\int_{\partial D_2} y_1 \nu(y)\cdot \ol{\nabla S_1(y)}\,dy-\int_{\partial D_1} y_1 \nu(y)\cdot \nabla S_2(y)\,dy ,
\end{equation*}
where $\nu(y)$ refers to the unit outer normal vector along $\partial D$. By the honeycomb symmetry relations between $D_1$ and $D_2$, and the relations between $S_2(x) = \mathcal{PC}S_1(x) = \ol{S_1(2x_0-x)}$, under the change of variable $y\mapsto R_0y = 2x_0-y$, we have $D_2 = R_0(D_1)$, and for all $y\in \partial D_1$,
\begin{equation*}
y+R_0y = 2x_0, \; \nu(y) = -\nu(R_0y), \; (\nabla S_2)(y) = -\ol{\nabla S_1(R_0 y)}.
\end{equation*}
Hence, we get
\begin{equation*}
\begin{aligned}
\int_{\partial D_2} y_1\nu(y)\cdot \ol{\nabla S_1(y)}\,dy &= \int_{\partial D_1} (2x_0-y)_1 \nu(y)\cdot \nabla S_2(y)\,dy\\
&= -2(x_0)_1 C_{12}(\alpha^*) - \int_{\partial D_1} y_1 \nu(y)\cdot \nabla S_2(y)\,dy.
\end{aligned}
\end{equation*}
We know that $C_{12}(\alpha^*) = 0$, so the integrals on $\partial D_1$ and on $\partial D_2$ are opposite of each other, and we only need to compute one of them. Let us focus on the integral over $\partial D_1$ of $y_1 \nu(y)\cdot \nabla S_2(y)$. We use the honeycomb symmetry again, under the rotation $R_1$ (counter clockwise by $2\pi/3$ around the center of $D_1$), $R_1D_1 = R_1^2 D_1 = D_1$. Also, for all $x\in Y_1$, by definition \eqref{eq:Rfdef} and the relation \eqref{eq:S12R},
\begin{equation*}
	\mathcal{R}S_2(x) = \tau S_2(R_1x), \quad S_2(R_1x) = \tau S_2(x), \quad (\nabla S_2)(R_1x) = \tau R\nabla S_2(x).
\end{equation*}
Also, under the mapping $R_1$, for all $y\in \partial D_1$, we have $\nu(R_1 y) = R\nu(y)$. Hence, by a change of variable we get
\begin{equation*}
	\int_{\partial D_1} y_1\nu(y)\cdot \nabla S_2(y)\,dy = \int_{\partial D_1} \tau (R_1 y)_1 \nu(y) \cdot \nabla S_2(y)\,dy.
\end{equation*}
Similarly, we get
\begin{equation*}
	\int_{\partial D_1} y_1\nu(y)\cdot \nabla S_2(y)\,dy = \int_{\partial D_1} \tau^2 (R^2_1 y)_1 \nu(y) \cdot \nabla S_2(y)\,dy.
\end{equation*}
Computing the exact expressions of $(\tau R y)_1$ and $(\tau^2 R^2y)_1$ in terms of $y_1$ and $y_2$, and using the fact that $C_{12}(\alpha^*)=0$, we can get from the two equations above
\begin{equation*}
	\int_{\partial D_1} y_1 \frac{\partial S_2}{\partial \nu}(y)\,dy = \int_{\partial D_1} y_2 \frac{\partial S_2}{\partial \nu}(y)\,dy = 0.
\end{equation*}
We hence have proved
\begin{equation*}
	b = \int_{\partial Y} y_1\left(\frac{\partial S_2}{\partial \nu}(y)\ol{S_1(y)} - S_2(y)\ol{\frac{\partial S_1}{\partial \nu}(y)}\right)\,dy = \mi c.
\end{equation*}
Now that $\mathbf{b} = \mi c (1,\mi)^\top$, $\nabla_\alpha C_{12}(\alpha)|_{\alpha=\alpha^*} = c(1,-\mi)^\top$, we see that \eqref{eq:gradalphaC12} holds.

\medskip

\emph{Step 3}. Finally, we simplify the expression of $c$ in \eqref{eq:c_formula1}. Setting $\gamma_1 = \{tl_1\,:\, t\in [0,1]\}$, $\gamma_2 = \{tl_2\,:\, t\in [0,1]\}$, we can break $\partial Y$ into four pieces: $\gamma_1$, $\gamma_2$, $\gamma_1+l_2$ and $\gamma_2+l_1$. Since $S_1,S_2$ are $\alpha^*$-quasiperiodic, we get
\begin{equation*}
	\int_{\gamma_1+l_2} y_1S_2 \ol{\frac{\partial S_1}{\partial \nu}}\,dy = \int_{\gamma_1} (y_1+l_2^{(1)})S_2(y)e^{\mi \alpha^*\cdot l_2}\ol{\left(-\frac{\partial S_1}{\partial \nu}(y) e^{-\mi \alpha^*\cdot l_2}\right)}\,dy,
\end{equation*}
where $l_2^{(1)} = \sqrt{3}L/2$ is the first component of $l_2$, and $L$ is the lattice constant of $\Lambda$; the minus sign on the right is due to the fact that the outer normal for $\gamma_1$ is the opposite of that of $\gamma_1+l_2$. The integral on the other pieces of $\partial Y$ can be similarly analyzed. Note also $l_1^{(1)} = l_2^{(1)}$. We hence get
\begin{equation*}
	\int_{\partial Y} y_1S_2 \ol{\frac{\partial S_1}{\partial \nu}}\,dy = -\frac{\sqrt{3}L}{2} \int_{\partial Y\cap \partial Y_1} S_2 \ol{\frac{\partial S_1}{\partial \nu}}\,dy.
\end{equation*}
The remaining term in \eqref{eq:c_formula1} can be similarly computed, and we get the desired formula \eqref{eq:c_formula2}.
\end{proof}

\section{Outline of the proofs for the main theorems} \label{sec:proofs}

Let $w_{\varepsilon,\delta}$ be the solution of problem \eqref{problem 1}. We define $u_{\varepsilon,\delta} (x,t) = \eps w_{\eps,\delta}(\eps x,\eps t)$. 
Note that after this rescaling, the space and time in $u_{\eps,\delta}$ correspond to the fast ones; the extra $\eps$-rescaling in the amplitude makes the transformation unitary in $L^2(\R^2_x)$, 
$$
\|w_{\eps,\delta}(\cdot,t)\|_{L^2(\R^2)} = \|u_{\eps,\delta}(\cdot,t/\eps)\|_{L^2(\R^2)}.
$$
Then $u_{\varepsilon,\delta}$ satisfies 
\begin{equation}\label{problem 2}
	\left\{
	\begin{aligned}
	  & \left( \partial_t^2 + \mathcal{L}_{\delta} \right) u_{\varepsilon,\delta} = 0   \quad  &&(x,t)\in \mathbb{R}^2\times (0,\infty), \\
		& u_{\varepsilon,\delta}(x,0)=   \eps \left( F_1(\eps x) S_1(x)+ F_2(\eps x) S_2(x)\right) \quad &&x\in \R^2,\\
		&\partial_t u_{\varepsilon,\delta}(x,0) = \mathrm{i} \omega_{\delta}^* \eps \left( F_1(\eps x) S_1(x)+ F_2(\eps x) S_2(x)  \right) \qquad  &&x\in \R^2.
	\end{aligned}\right.
\end{equation}
Replacing $S_j$ by $S_{j,\delta}$ above, we also consider the solution $\widetilde{u}_{\varepsilon,\delta}$ to the following problem: 
\begin{equation}\label{problem 3}
	\left\{
	\begin{aligned}
		& \left( \partial_t^2 + \mathcal{L}_{\delta} \right) \widetilde{u}_{\varepsilon,\delta}= 0   \quad  &&(x,t)\in \mathbb{R}^2\times (0,\infty), \\
		& \widetilde{u}_{\varepsilon,\delta}(x,0)=   \eps \left( F_1(\eps x) S_{1,\delta}(x)+  F_2(\eps x) S_{2,\delta}(x)\right) \qquad &&x\in \R^2,\\
		&\partial_t \widetilde{u}_{\varepsilon,\delta}(x,0) = \mathrm{i} \omega_{\delta}^*\eps \left(  F_1(\eps x) S_{1,\delta}(x)+ F_2(\eps x) S_{2,\delta}(x)  \right) \qquad &&x\in \R^2.
	\end{aligned}\right.
\end{equation}
Define 
\begin{equation}
\label{ved}
	\eta_{\varepsilon,\delta}(x,t) = \widetilde{u}_{\varepsilon,\delta} (x,t) - e^{-\mathrm{ i} \omega_{\delta}^* t } \left(\eps V_1 (\eps x,\eps t) S_{1,\delta}(x)+  \eps V_2(\eps x, \eps t) S_{2,\delta}(x) \right).
\end{equation}
By a direct computation we see that, for the original approximation error $r_{\varepsilon,\delta}$ defined in \eqref{ansatz}, the following holds:
\begin{equation*}
		\| r_{\varepsilon,\delta}(\cdot,t) \|_{L^2(\mathbb{R}^2)}  =	\| \eps r_{\varepsilon,\delta}(\eps x,t) \|_{L^2(\mathbb{R}_x^2)}. 
\end{equation*}
Hence, we have
\begin{equation*}
\|r_{\eps,\delta}(\cdot,t)\|_{L^2(\mathbb{R}^2)} \le \sum_{j=1}^3 \|I_j(x,t)\|_{L^2_x(\mathbb{R}^2)}.
\end{equation*}
where the terms $I_j(x,t)$, $j=1,2,3$, are defined by
\begin{equation}
\label{eq:I123def}
	\begin{aligned}
		& I_1(x,t) = \eps V_1 (\eps x,t) (S_1- S_{1,\delta} )(x) + \eps V_2(\eps x, t) (S_2 - S_{2,\delta} )(x),\\ 
		& I_2(x,t) = u_{\varepsilon,\delta} (x, t/\eps)-\widetilde{u}_{\varepsilon,\delta} (x, t/\eps),\\ 
		&I_3(x,t) = \eta_{\varepsilon,\delta} (x, t/\eps).
	\end{aligned}
\end{equation}
Therefore, Theorem \ref{thm:main} is proved once we have finished the estimates of $I_1, I_2$ and $I_3$ in the subsequent sections of this paper:

In Section \ref{sec:I1} (Proposition \ref{prop:I1}), we prove that
\begin{equation*}
\|I_1(\cdot,t)\|_{L^2(\R^2)}\le C_{\bm{F}} \delta.
\end{equation*}

In Section \ref{subs:I2} (Proposition \ref{prop:I2}), we prove that
 \begin{equation*}
 \|I_2(\cdot,t)\|_{L^2_\delta(\R^2)} \leq C_{\bm{F}}(\delta^{\frac12}+t\eps\delta).
 \end{equation*}

 In Section \ref{subs:I3}, we divide $I_3$ into three parts: 
\begin{equation*}
	\|I_3(\cdot,t)\|_{L^2(\mathbb{R}^2)} \leq  \sum_{k=1}^3 \| \eta^{(k)}_{\varepsilon,\delta}(\cdot,t/\eps)\|_{L^2(\mathbb{R}^2)} .
\end{equation*}

In Propositions \ref{prop:eta1bound} and \ref{prop:eta2bound}, we show that 
\begin{equation*}
	\|\eta^{(1)}_{\varepsilon,\delta} (\cdot,t/\eps) \|_{L^2(\mathbb{R}^2)} + 	\|\eta^{(2)}_{\varepsilon,\delta} (\cdot,t/\eps) \|_{L_{\delta}^2(\mathbb{R}^2)} \leq C_{\bm F} (\eps +t \sqrt{\delta}+ t\eps + t^2\eps \sqrt{\delta} ).
\end{equation*}
In Proposition \ref{prop:eta3bound} we show that
\begin{equation*}
\begin{aligned}
	\|\eta^{(3)}(\cdot,t/\varepsilon)\|_{L^2(\R^2)} &\le C_{\bm{F}} \eps\delta^{-\frac12} + C_{\bm{F}} t(\eps^{N+\nu-1}\delta^{-1} + \varepsilon^{\nu-1} \delta^{1/2}+ \varepsilon^{2\nu-1} \delta^{-1/2} )\\
    &\quad\, +C_{\bm{F}} t^2 (
    \varepsilon^{\nu-2} \delta^{3/2} + \varepsilon^{2\nu-2} \delta^{1/2} + \eps^{(1-\nu)N-2}). 
 \end{aligned}
\end{equation*}

Combine all the estimates above and choose $\delta = \eps^\kappa$, $\kappa \in (0,2)$, we get
\begin{equation*}
\begin{aligned}
\|r_{\eps,\delta}(\cdot,t)\|_{L^2(\R^2)} &\le C_{\bm{F}}\Big(\eps^{\frac{\kappa}{2}} + \eps^{1-\frac{\kappa}{2}} + t\eps^{\min\{N+\nu-1-\kappa, \nu-1+\frac{\kappa}{2}, 2\nu-1-\frac{\kappa}{2}\}} \\
&\quad + t^2\varepsilon^{\min\{ \nu-2+\frac{3\kappa}{2}, 2\nu-2+\frac{\kappa}{2}, (1-\nu)N-2\}}\Big).
\end{aligned}
\end{equation*}
We hence have established the main estimate in Theorem \ref{thm:main}.

\subsection{Estimate of $I_1$}
\label{sec:I1}

Recall the definition of $I_1$ in \eqref{eq:I123def}, which accounts for using $S_j$ rather than $S_{j,\delta}$, $j=1,2$, in the initial wave packets. Here, $S_{j,\delta}$ as defined in Theorem \ref{thm:DiracModes} are the Bloch modes associated to the Dirac point $\alpha^*$ of the band structure of $\mathcal{L}_\delta$, and $S_j$ as defined in \eqref{eq:Sjdef} are the zero-frequency limit of $S_{j,\delta}$; see also \eqref{eq:SjH1error}. Using $S_j$ rather than $S_{j,\delta}$ makes sense from a practical point of view, as that data are then independent of the contrast variable $\delta$ and are hence easier to prepare.

\begin{lemma}\label{lem:SpPL2bound}
	Let $\Gamma \in \mathscr{S}(\mathbb{R}^2)$ and let $\phi \in L_{\rm loc}^2(\mathbb{R}^2)$ be $\Lambda$-periodic. Then, for all positive number $s >1$, there is a constant $C=C(s) >0$ such that, for all $\varepsilon>0$, we have
	\begin{equation*}
	\|	\eps \Gamma(\eps x) \phi(x)\|_{L^2(\mathbb{R}^2)} \leq C\|\phi\|_{L^2(Y)} \sup_{x\in \mathbb{R}^2} \big| (1+|x|^2 )^{\frac{s}{2}} \Gamma(x) \big| .
	\end{equation*}
	Moreover, the above inequality still holds if we replace $L^2(\R^2)$ by $L^2_\delta(\R^2)$ on the left-hand side and replace $L^2(Y)$ by $L^2_\delta(Y)$ on the right-hand side.
\end{lemma}
\begin{proof}
	Recall that $Y_{\rm B}$ denotes the Wigner-Seitz cell of $\Lambda$ containing $(0,0)$. Since $\phi$ is $\Lambda$-periodic, we have
	\begin{equation}\label{eq:GphiL2-1}
		\begin{aligned}
			\int_{\mathbb{R}^2} | \eps \Gamma(\eps x) \phi(x) |^2 \,dx &= \varepsilon^2 \sum_{\mathbf{m} \in \mathbb{Z}^2 } \int_{Y_{\rm B}} | \Gamma(\varepsilon x + \eps\mathbf{m}\bm{l}) |^2|\phi(x)|^2\,dx \\
			&\leq  \varepsilon^2 \|\phi\|^2_{L^2(Y)} \sum_{\mathbf{m} \in \mathbb{Z}^2 } \max_{x \in Y_{\rm B}}  | \Gamma(\varepsilon x + \eps\mathbf{m}\bm{l}) |^2.
		\end{aligned}
	\end{equation}
We can find a constant $c\in (0,1)$ independent of $\{\eps, s\}$ such that
\begin{equation*}
	\min_{x\in Y_{\rm B}} |x+\mathbf{m}\bm{l}| \ge c|\mathbf{m}|, \quad \forall \mathbf{m} \in \Z^2.
\end{equation*}
Since $\Gamma \in \mathscr{S}(\mathbb{R}^2)$, 
	for any $s>1$, we have
	\begin{equation*}
		\max_{x \in Y_{\rm B}}  | \Gamma(\varepsilon x + \eps\mathbf{m}\bm{l}) | \le \frac{1}{c^s(1+\eps^2|\mathbf{m}|^2)^{s/2}}\sup_{x\in \mathbb{R}^2} \big| (1+|x|^2 )^{\frac{s}{2}} \Gamma(x) \big|, \quad \forall \, \mathbf{m}\in \Z^2.
	\end{equation*} 
We then get
	\begin{equation*}
			\int_{\mathbb{R}^2} | \eps \Gamma(\eps x) \phi(x) |^2 \,dx  \leq  C(\varepsilon ,s)\|\phi\|^2_{L^2(Y)} 	\left(\sup_{x\in \mathbb{R}^2} \ (1+|x|^2 )^{\frac{s}{2}} \Gamma(x) \right)^2 , 
	\end{equation*}
where
\begin{equation*}
	C(\varepsilon ,s) = \frac{1}{c^{2s}}\sum_{\mathbf{m} \in \mathbb{Z}^2 }  \frac{ \varepsilon^2 }{(1+\varepsilon^2 |\mathbf{m}|^2 )^s}. 
\end{equation*}
The desired result for $L^2(\R^2)$ estimates then holds because $C(\eps,s)$ can be bounded independently of $\eps$ for $s>1$ (as $\eps\to 0$, it approximates the integral $\int_{\R^2} (1+|x|^2)^{-s}$). The same result for $L^2_\delta(\R^2)$ also holds if we simply change $dx$ to $\sigma^{-1}(x)dx$ in \eqref{eq:GphiL2-1}.
\end{proof}

\begin{proposition}\label{prop:I1}
	Under the assumption of Theorem \ref{thm:main}, there is a universal constant $C>0$, such that, for any $t>0$,
	\begin{equation*} 
		\|I_1(\cdot,t)\|_{L^2(\R^2)} \leq C\delta \sum_{j=1,2} \sup_{x\in \R^2} (1+|x|^2)|F_j(x)|. 
	\end{equation*}
\end{proposition}
\begin{proof}
	Choosing $s=2$ in Lemma \ref{lem:SpPL2bound}, we get
	\begin{equation*}
		\|I_1 (\cdot,t)\|_{L^2(\mathbb{R}^2)} \leq C\sum_{j=1,2}\|S_j - S_{j,\delta} \|_{L^2(Y)} \sup_{x\in \mathbb{R}^2} \big| (1+|x|^2 )V_j(x,t) \big|.
	\end{equation*}
The conclusion then follows from the estimates \eqref{eq:SjH1error} and \eqref{eq:decayDirac} by choosing $N=2$.
\end{proof}

\subsection{Estimate of $I_2$}
\label{subs:I2}
In this subsection, we control the $L^2_\delta(\R^2)$ norm of $I_2(\cdot,t)$. This term accounts for the difference of using $S_{j,\delta}$ versus $S_j$ in the Cauchy data of the wave equations \eqref{problem 2} and \eqref{problem 3}. 

Let $v_{\eps,\delta} = u_{\eps,\delta} - \widetilde u_{\eps,\delta}$. Then by definition \eqref{eq:I123def} we know that $I_2(x,t)=v_{\eps,\delta}(x,t/\eps)$. Clearly, $v_{\eps,\delta}$ solves the wave equation 
\begin{equation*}
\left\{\begin{aligned}
&\partial_t^2 v_{\eps,\delta} + \mathcal{L}_\delta v_{\eps,\delta} =0, \qquad &(x,t) \in \R^2\times (0,\infty),\\
&v_{\eps,\delta}(x,0) = \eps F_1(\eps x) (S_1-S_{1,\delta})(x) +  \eps F_2(\eps x) (S_2-S_{2,\delta})(x), \qquad &x\in \R^2,\\
&\partial_t v_{\eps,\delta}(x,0) = \mathrm{i}\omega^*_\delta \left( \eps F_1(\eps x) (S_1-S_{1,\delta})(x) + \eps F_2(\eps x) (S_2-S_{2,\delta})(x) \right), \qquad &x\in \R^2.
\end{aligned}
\right.
\end{equation*}
For notational simplicity, we write $v$ for $v_{\eps,\delta}$ in this section. 

Using the Floquet-Bloch states representation, we get the following formula for $v$:
\begin{equation}
\label{eq:v_FB}
	\begin{aligned}
		v(x,T) =\,&\sum_{j\geq 1} \int_{Y^*}\cos(\omega_{j,\delta}(\alpha)T) \langle v(x,0)  , \Phi_{j,\delta}(\cdot,\alpha) \rangle_{L^2_{\delta}(\mathbb{R}^2)} \Phi_{j,\delta}(x,\alpha)\,d\alpha \\
		& + \mathrm{ i} \omega_{\delta}^* \sum_{j\geq 1} \int_{Y^*}\frac{\sin(\omega_{j,\delta}(\alpha)T) }{\omega_{j,\delta}(\alpha)} \langle v(x,0)  , \Phi_{j,\delta}(\cdot,\alpha) \rangle_{L^2_{\delta}(\mathbb{R}^2)} \Phi_{j,\delta}(x,\alpha)\,d\alpha.
	\end{aligned}
\end{equation}
In view of \eqref{plancherel}, as long as the $L^2_\delta(\R^2)$ norm is concerned, the first item on the right-hand side above is controlled by the initial value $\|v(\cdot,0)\|_{L^2_\delta(\R^2)}$, which is of order $\mathcal{O}(\delta^{1/2})$. More precisely, by using Lemma \ref{lem:SpPL2bound} (with $s=2$),  we have
\begin{equation*}
\|v(\cdot,0)\|_{L^2_\delta(\mathbb{R}^2)} \le \delta^{-1/2}\|v(\cdot,0)\|_{L^2(\mathbb{R}^2)} \le C\delta^{\frac12} \|\widehat{\bm{F}}\|_{W^{2,1}}.
\end{equation*} 

To control the $L^2_\delta$ norm of the second term, we break the integral over $Y^*$ to regions near and away from the Dirac point $\alpha^*$. In view of the conical structure of the first two Bloch bands near the Dirac point $\alpha^*$, we can find a radius $r>0$ and a constant $C>0$, both independent of $\delta$, such that 
\begin{equation}\label{eq:Y*cut}
	\min_{|\alpha -\alpha^*|<r}\omega_{j,\delta}(\alpha) > \frac{1}{2} \omega_{\delta}^*  > C\delta^{\frac{1}{2}}, \quad j=1,2.
\end{equation}
Moreover, in view of \eqref{eq:bandall}, there is a constant $C>0$ such that
\begin{equation}\label{eq:Y*nocut}
	\min_{\alpha \in Y^*}\omega_{j,\delta}(\alpha) > C, \quad j\geq 3.
\end{equation}
Since $\omega^*_\delta \sim \sqrt{\delta}$, for any $T>0$, the $L^2_\delta(\R^2_x)$ norm of 
\begin{equation*}
    \mathrm{ i} \omega_{\delta}^* \left( \sum_{j=1,2} \int_{Y^*\cap \{|\alpha-\alpha^*|< r\}} + \sum_{j\geq 3} \int_{Y^*} \right)\frac{\sin(\omega_{j,\delta}(\alpha)T) }{\omega_{j,\delta}(\alpha)} \langle v(x,0)  , \Phi_{j,\delta}(\cdot,\alpha) \rangle_{L^2_{\delta}(\mathbb{R}^2)} \Phi_{j,\delta}(x,\alpha)\,d\alpha 
\end{equation*}
is bounded by $\|v(\cdot,0)\|_{L^2_\delta(\R^2)}$ as well. Hence, we  have
\begin{equation}\label{eq:I2bound}
	\begin{aligned}
			\|I_2(\cdot,t)&\|_{L^2_{\delta}(\mathbb{R}^2)} = \|v(\cdot,t/\eps)\|_{L^2_\delta(\R^2)} \leq C \|v(\cdot,0)\|_{L^2_{\delta}(\mathbb{R}^2)} \\
			&+ \left\| \mathrm{ i} \omega_{\delta}^* \sum_{j=1,2} \int_{Y^*\cap \{|\alpha-\alpha^*|\ge r\}}\frac{\sin(\omega_{j,\delta}(\alpha)t/\eps) }{\omega_{j,\delta}(\alpha)} \langle v(x,0)  , \Phi_{j,\delta}(\cdot,\alpha) \rangle_{L^2_{\delta}(\mathbb{R}^2)} \Phi_{j,\delta}(x,\alpha)\,d\alpha\right\|_{L^2_\delta(\mathbb{R}^2)}\\
			&\leq C\delta^{\frac{1}{2}} \|\widehat{\bm{F}}\|_{W^{2,1}} + C\delta^{\frac12} \frac{t}{\eps} \sum_{j=1,2} \max_{|\alpha-\alpha^*|\geq r} |\langle v (\cdot,0)  , \Phi_{j,\delta}(\cdot,\alpha) \rangle_{L^2_{\delta}(\mathbb{R}^2)} | .
	\end{aligned}
\end{equation}
To proceed,  we need the following lemma first used in \cite{fefferman_wave_2014}.
	\begin{lemma}\label{f.w.lemma}
	The following results hold:
	\begin{itemize}
		\item[(i)] 	Let $\Gamma \in \mathscr{S}(\mathbb{R}^2)$, let $\Psi $ be $\alpha^*$-quasiperiodic, and let $\Phi$ be $\alpha$-quasiperiodic. Then,  
		\begin{equation*}
			\begin{aligned}
				&\langle  \Gamma(\eps\,\cdot) \Psi(\cdot,\alpha^*), \Phi(\cdot,\alpha)\rangle_{L^2_{\delta}(\mathbb{R}^2)}   \\
				& = \varepsilon^{-2} \sum_{\mathbf{m} \in \mathbb{Z}^2}  \widehat{\Gamma} \left( \frac{\mathbf{m}\bm{\alpha} + \alpha - \alpha^*}{\varepsilon} \right) \int_Y \sigma_{\delta}^{-1}(x) e^{\mathrm{i} (\mathbf{m}\bm{\alpha}+ \alpha - \alpha^* ) \cdot x} \Psi(x) \overline{\Phi(x)} \,dx.
			\end{aligned}
		\end{equation*}
		\item[(ii)] For any $r>0$, there is a constant $C>0$, which depends only on $r$, such that
		\begin{equation}
			|\mathbf{m}\bm{\alpha} + \alpha - \alpha^*| \geq C(1 + |\mathbf{m}|) ,\quad \forall \mathbf{m} = (m_1,m_2) \in \mathbb{Z}^2,\  |\alpha - \alpha^*|>r.
		\end{equation}
		\item[(iii)] For any $0 < \nu <1$, there is a constant $c=c(\nu)$ such that for any $\varepsilon>0$ sufficiently small, 
		\begin{equation}
			\begin{aligned}
				&|\mathbf{m}\bm{\alpha} + \alpha - \alpha^*| \geq c|\mathbf{m}| ,\quad &\forall \mathbf{m}  \in \mathbb{Z}^2,\  |\alpha - \alpha^*|<\varepsilon^{\nu}, \\
				&|\mathbf{m}\bm{\alpha} + \alpha - \alpha^*| \geq c\varepsilon^{\nu} (1 + |\mathbf{m}|) ,\quad &\forall \mathbf{m}  \in \mathbb{Z}^2,\  |\alpha - \alpha^*|\geq \varepsilon^{\nu} .
			\end{aligned}
		\end{equation}
	\end{itemize}
\end{lemma}

\begin{proof}
	We refer to \cite[Propositions 7.1 and 7.2]{fefferman_wave_2014} for the proof. Note that for (i), although the proof in \cite{fefferman_wave_2014} only considers a special weight $\sigma_{\delta}\equiv  1$ and uses Poisson summation formula, the method there is also valid for any $\Lambda$-periodic weight $\sigma_{\delta}$.
\end{proof}

\begin{proposition}\label{prop:I2}
	There is a constant $C>0$ such that for any $t>0$, 
	\begin{equation*}
		\|I_2(\cdot,t)\|_{L^2_\delta(\R^2)} \leq C\left(\delta^{\frac12}\|\widehat{\bm{F}}\|_{W^{2,1}} + t\eps \delta \|\bm{F}\|_{W^{3,1}}\right).
	\end{equation*}
\end{proposition}
\begin{proof} In view of \eqref{eq:I2bound}, we only need to estimate
\begin{equation*}
		\max_{|\alpha - \alpha^*| >r } \left|  \big\langle 
	\varepsilon \Gamma(\eps x ) \Psi(x) ,  \Phi_{j,\delta}(x,\alpha)
		\big\rangle_{L_{\delta}^2(\mathbb{R}^2_x)}   \right|
	\end{equation*}
	for $\Gamma=F_1, \Psi=S_1-S_{1,\delta}$, and for $\Gamma=F_2$, $\Psi=S_2-S_{2,\delta}$. Let us only estimate the first pair as the other one can be done in exactly the same way. 

	We apply Lemma \ref{f.w.lemma}(a) with our choice of $\Gamma$, $\Psi$ and $\Phi = \Phi_{j,\delta}(\cdot,\alpha)$, and get 
	\begin{equation*}
			\big|\langle  \eps \Gamma(\eps x) \Psi(\cdot), \Phi_{j,\delta}(\cdot,\alpha)\rangle_{L^2_{\delta}(\mathbb{R}^2)}  \big|  \leq  \varepsilon^{-1} \| \Psi \|_{L^2_{\delta}(Y)} \sum_{\mathbf{m} \in \mathbb{Z}^2}  \left| \widehat{\Gamma} \left( \frac{\mathbf{m}\bm{\alpha}+ \alpha - \alpha^*}{\varepsilon} \right) \right|.
	\end{equation*}
Using the fact that $\hat\Gamma \in \mathscr{S}(\R^2_\xi)$ and that $\|\Psi\|_{L^2(Y)} \le  C\delta$ by \eqref{eq:SjH1error}, we can control the above by
\begin{equation*}
	C\eps^{-1}\delta^{\frac12} \Big(\sup_{\xi \in \mathbb{R}^2} |\xi|^3 |\widehat{\Gamma}(\xi) |\Big) \sum_{\mathbf{m} \in \mathbb{Z}^2} \frac{\varepsilon^3}{|\mathbf{m}\bm{\alpha} + \alpha - \alpha^*|^3}.
\end{equation*}
Now for $|\alpha-\alpha^*|>r$, we can use Lemma \ref{f.w.lemma}(ii) to obtain that
		\begin{equation*}
		\max_{|\alpha - \alpha^*| >r } \left|  \big\langle 
		\varepsilon  F_1(\eps \,\cdot) (S_1 - S_{1,\delta})(\cdot) ,  \Phi_{j,\delta}(\cdot,\alpha)
		\big\rangle_{L_{\delta}^2(\mathbb{R}^2)}   \right| < C\varepsilon^2 \delta^{\frac12} \| F_1\|_{W^{3,1}} .
	\end{equation*}
Using the above estimate in \eqref{eq:I2bound}, we get the desired result.
\end{proof}

\section{Estimate of $I_3$} \label{sec:i3}
\label{subs:I3}

In this section, we control the $L^2(\R^2)$ norm of $I_3(\cdot,t)$. This term corresponds to the approximation error $\eta_{\eps,\delta}$ defined by \eqref{ved}. It is the approximation error of the evolution of the wave packets with base modes $S_{1,\delta}$ and $S_{2,\delta}$. 

Applying the operator 
$\partial_t^2 + \mathcal{L}_{\delta}$ on $\eta_{\eps,\delta}$,   we obtain in view of \eqref{problem 3}\eqref{ved} and \eqref{action fg} that
\begin{equation}\left\{
	\begin{aligned}
	  &(\partial_t^2 + \mathcal{L}_\delta)\eta_{\eps,\delta} = -e^{\mathrm{i}\omega^*_\delta t}\left(f_{\eps,\delta}^{(2)} + f_{\eps,\delta}^{(3)}\right), \qquad &&(x,t)\in \R^2\times (0,\infty),\\  &\eta_{\varepsilon,\delta} (x,0)= 0, \qquad &&x\in \R^2,\\
	  & \partial_t \eta_{\varepsilon,\delta}  (x,0) = -f_{\varepsilon,\delta}^{(1)}(x), \qquad &&x\in \R^2. 
 \end{aligned}\right.
\end{equation}
Here, the initial datum $f^{(1)}_{\eps,\delta}$ for $\partial_t \eta_{\eps,\delta}$ and the source terms $f^{(2)}_{\eps,\delta}, f^{(3)}_{\eps,\delta}$ are given by
\begin{equation}
\label{eq:f123}
	\begin{aligned}
  f^{(1)}_{\eps,\delta} =\,& \eps^2 (\partial_t V_1)(\eps x,0) S_{1,\delta}(x) + \eps^2 (\partial_t V_2)(\eps x,0) S_{2,\delta}(x),\\
  f^{(2)}_{\eps,\delta} =\,& \eps^3 (\partial^2_t V_1)(\eps x,\eps t) S_{1,\delta}(x) + \eps^3 (\partial^2_t  V_2)(\eps x,\eps t) S_{2,\delta}(x), \\
   f^{(3)}_{\eps,\delta} =\,&\eps^2 \sum_{j=1,2} \big\{ 2\mathrm{i} \omega^*_\delta (\partial_t V_j)(\eps x,\eps t) S_{j,\delta}(x) - (\nabla V_j)(\eps x,\eps t)\cdot \nabla S_{j,\delta}\\
   &\qquad\qquad \qquad\qquad- \sigma_\delta\nabla \cdot \left(\sigma_\delta^{-1} (\nabla V_j)(\eps x,\eps t) S_{j,\delta}(x)\right)\big\}.
	\end{aligned}
\end{equation}
Note that the last term in $f^{(3)}_{\eps,\delta}$ is understood as an element of $H^{-1}_\delta(\R^2)$.
By linearity, we further write $\eta_{\varepsilon,\delta}$ as the sum of $\eta^{(1)}_{\varepsilon,\delta}$, $\eta^{(2)}_{\varepsilon,\delta}$ and $\eta^{(3)}_{\varepsilon,\delta} $ which accounts for the contributions of $f_{\eps,\delta}^{(1)}$, $f_{\eps,\delta}^{(2)}$ and $f_{\eps,\delta}^{(3)}$, respectively. These terms are defined by
	\begin{equation}
 \label{eq:errorf1}
		\left\{
		\begin{aligned}
		  & (\partial_t^2 + \mathcal{L}_{\delta}) \eta^{(1)}_{\varepsilon,\delta} = 0 && \,(x,t)\in \mathbb{R}^2 \times (0,\infty) , \\
		  & \left(\eta^{(1)}_{\varepsilon,\delta}, \partial_t \eta^{(1)}_{\varepsilon,\delta}\right) (x,0) = \left(0,-f^{(1)}_{\varepsilon,\delta}(x)\right) &&  \, x\in \mathbb{R}^2,
		\end{aligned}\right.
	\end{equation}
	\begin{equation}
 \label{eq:errorf2}
			\left\{
		\begin{aligned}
		  & (\partial_t^2 + \mathcal{L}_{\delta}) \eta^{(2)}_{\varepsilon,\delta} =   -   e^{ \mathrm{i} \omega_{\delta}^* t} f_{\varepsilon,\delta}^{(2)}(x,t) &&(x,t) \in \mathbb{R}^2 \times (0,\infty)
			, \\
			& \left(\eta^{(2)}_{\varepsilon,\delta}, \partial_t \eta^{(2)}_{\varepsilon,\delta}\right)  (x,0) = (0,0) &&x\in \mathbb{R}^2,
		\end{aligned}\right.  
	\end{equation}
and
\begin{equation}
\label{eq:errorf3}
\left\{
\begin{aligned}
  & (\partial_t^2 + \mathcal{L}_{\delta}) \eta^{(3)}_{\varepsilon,\delta} =   -   e^{\mathrm{i}  \omega_{\delta}^* t} f_{\varepsilon,\delta}^{(3)}(x,t)  &&(x,t) \in \mathbb{R}^2 \times (0,\infty) , \\
& \left(\eta^{(3)}_{\varepsilon,\delta},\partial_t \eta^{(3)}_{\varepsilon,\delta}\right)  (x,0) = (0,0) &&x\in \mathbb{R}^2.
\end{aligned}\right.
\end{equation}
For simplicity, we write $\eta^{(j)}, f^{(j)}$ for $\eta_{\varepsilon,\delta}^{(j)}, f_{\varepsilon,\delta}^{(j)}$, respectively, in the following sections.

Those equations can be solved using the standard functional analysis framework and using the Floquet-Bloch states representation. More precisely, for such $\mathcal{L}_\delta$-wave equation with general right hand side $f$ and initial data $(u(0),\partial_t u(0)) = (g,h)$, by the Hille-Yosida theory and Duhamel's principle, the solution is given by
\begin{equation*}
u(x,T) = \cos(\sqrt{\mathcal{L}_\delta}T)g + \frac{\sin(\sqrt{\mathcal{L}_\delta} T)}{\sqrt{\mathcal{L}_\delta}} h + \int_0^T \frac{\sin(\sqrt{\mathcal{L}_\delta} (T-t))}{\sqrt{\mathcal{L}_\delta}} f(t)\,dt.
\end{equation*}
Using the Floquet-Bloch states $\{\omega_{j,\delta}(
\alpha),\Phi_{j,\delta}(\cdot,\alpha)\}_{j,\alpha}$, we further get
\begin{equation*}
\begin{aligned}
u(x,T) = &\int_{Y^*} \cos(\omega_{j,\delta}(\alpha)T)\langle g,\Phi_{j,\delta}(\cdot,\alpha)\rangle_{L^2_\delta(\R^2)} \Phi_{j,\delta}(x,\alpha)\,d\alpha \\
& \; + \int_{Y^*} \frac{\sin(\omega_{j,\delta}(\alpha)T)}{\omega_{j,\delta}(\alpha)} \langle h,\Phi_{j,\delta}(\cdot,\alpha)\rangle_{L^2_\delta(\R^2)}\Phi_{j,\delta}(x,\alpha)\,d\alpha\\
&\; + \int_{Y^*} \left(\int_0^T \frac{\sin(\omega_{j,\delta}(\alpha) (T-t))}{\omega_{j,\delta}(\alpha)} \langle f,\Phi_{j,\delta}(\cdot,\alpha)\rangle_{L^2_\delta(\R^2)}\,dt\right)\Phi_{j,\delta}(x,\alpha)\,d\alpha.
\end{aligned}
\end{equation*}
For $f$ in $L^2_\delta(\R^2)$ or $H^{-1}_\delta(\R^2)$, we have explained the meaning of the pairing $\langle f,\Phi_{j,\delta}(\cdot,\alpha)\rangle_{L^2(\R^2)}$ in Section \ref{sec:FBstates}.

\subsection{Some estimates for wave equations}

In the next two subsections, we derive $L^2(\R^2)$ (to emphasize, this is not $L^2_\delta(\R^2)$) estimates for the terms $\eta^{(1)}$ and $\eta^{(2)}$. We will mainly use the following estimates for the $\mathcal{L}_\delta$-wave equation.

\begin{theorem}\label{thm:L2wave}
	Let $r>0$ be chosen so that \eqref{eq:Y*cut} holds. Then there is a universal constant $C>0$ such that, if $u = u_\delta$ solves
		\begin{equation}\label{eq:exteriorL}
		\left\{
		\begin{aligned}
			& (\partial_t^2 + \mathcal{L}_{\delta}) u =  0 \qquad  &(x,t)\in \mathbb{R}^2\times [0,\infty), \\
			& u(x,0) = 0, \; \partial_t u(x,0) = g(x), \qquad &x\in \R^2,
		\end{aligned}\right.  
	\end{equation}
	then we can find $C>0$ such that uniformly in $T>0$,
	\begin{equation}
 \label{eq:L2Rwave}
		\| u(\cdot, T) \|_{L^2(\mathbb{R}^2)}  \leq  C \|g\|_{L^2_\delta(\R^2)} + CT\sqrt{\delta}\sum_{j=1,2} \max_{|\alpha-\alpha^*|>r} |\langle g(\cdot) , \Phi_{j,\delta}(\cdot,\alpha) \rangle_{L^2_{\delta}(\mathbb{R}^2)} | +CT\sqrt{\delta}\|g\|_{H^1_\delta(\R^2)}.
	\end{equation}
\end{theorem}

The proof of Theorem \ref{thm:L2wave} is provided in the appendix; see Section \ref{sec:L2wave}. The main idea is to control the $L^2(\R^2)$ norm of the wave equation in the honeycomb structure by exploring the wave operator outside the inclusions array with help from estimates of the energy in the inclusions. Let $\Omega^\Lambda = \Omega+\Lambda$ be this exterior domain, we have the following result.

\begin{proposition}\label{prop: exterior nonzero}
Let $v \in H^1(\Omega^\Lambda)$ be the solution of
	\begin{equation}
 \label{eq:wave_in0}
		\left\{
		\begin{aligned}
			& (\partial_t^2 - \Delta ) v =  f && \mathrm{in} \ \Omega^{\Lambda} \times [0,\infty) , \\
			& v(x,0) =0, \quad  \partial_t v (x,0) =0 && \mathrm{on} \ \Omega^{\Lambda}  , \\
			& v(x,t) = h(x,t) && \mathrm{on} \ \partial \Omega^{\Lambda}  \times [0,\infty).
		\end{aligned}\right.  
	\end{equation}
Here, $h(x,t)$ is a sufficiently regular function on $\partial \Omega^\Lambda$ (i.e., on $\Lambda+\partial D$) and satisfies $h(\cdot,0)=0$ and $\partial_t h(\cdot,0)=0$. Then, there is a universal constant $C>0$ such that for any $T>0$,
	\begin{equation}\label{formula:prop 6.4}
	\| v(\cdot, T) \|_{L^2(\Omega^{\Lambda})}  \leq C \| h(\cdot, T) \|_{H^{\frac{1}{2}}(\partial \Omega^{\Lambda})} +CT\max_{0\leq t\leq T} \| f \|_{L^2(\Omega^{\Lambda})} + CT  \max_{0\leq s \leq T}\| \partial^2_{t} h(\cdot, s) \|_{H^{\frac{1}{2}}(\partial \Omega^{\Lambda})},
\end{equation}
and, also,
	\begin{equation}\label{formula:prop 6.4 II}
	\| v(\cdot, T) \|_{L^2(\Omega^{\Lambda})}  \leq C \| h(\cdot, T) \|_{H^{\frac{1}{2}}(\partial \Omega^{\Lambda})} +CT\max_{0\leq t\leq T} \| f \|_{L^2(\Omega^{\Lambda})} + CT  \max_{0\leq s \leq T}\| \partial_{t} h(\cdot, s) \|_{H^{\frac{1}{2}}(\partial \Omega^{\Lambda})}.
\end{equation}
\end{proposition}

The proof of this proposition is postponed also to the Appendix; see section \ref{sec:proof:propextwave}. 
\subsection{Estimate of $\eta^{(1)}$}

The item $\eta^{(1)} = \eta^{(1)}_{\varepsilon,\delta}$ from  \eqref{eq:errorf1} solves the $\mathcal{L}_\delta$-wave equation with zero source and initial data
\begin{equation*}
	\eta^{(1)}  (x,0)  = 0 , \quad \quad \partial_t \eta^{(1)}(x,0) = -f^{(1)}(x),
\end{equation*}
with $f^{(1)} = f^{(1)}_{\varepsilon,\delta}$ given in \eqref{eq:f123}. From the formula \eqref{eq:f123} and in view of \eqref{eq:Dirac}, we have
\begin{equation}
\label{eq:f1formula}
	f^{(1)}(x) = -\eps^2 \eta_{\#,\delta}\left((\partial_1+\mi \partial_2)F_2\right)(\eps x) S_{1,\delta}(x) - \eps^2 \overline{\eta_{\#,\delta}}\left((\partial_1-\mi \partial_2)F_1\right)(\eps x) S_{2,\delta}(x). 
\end{equation}
Applying Lemma \ref{lem:SpPL2bound}, we get
\begin{equation*}
\|f^{(1)}\|_{L^2_\delta(\mathbb{R}^2)}\le C\eps|\eta_{\#,\delta}|\Big(\sum_{\ell=1,2}\|S_{\ell,\delta}\|_{L^2_\delta(Y)}\Big)\|(1+|x|^2)\nabla \bm{F}\|_{L^\infty}.
\end{equation*}
Since $\|S_{\ell,\delta}\|_{L^2_\delta(Y)} =\mathcal{O}(\delta^{-1/2})$ and $|\eta_{\#,\delta}| \sim \delta^{\frac12}$, the above is bounded by $C_{\bm{F}} \eps$. 

\medskip

Repeating the method in our analysis of $I_2(\cdot,T)$, we may try to control $\|\eta^{(1)}(\cdot,T)\|_{L^2_\delta(\R^2)}$ using the Floquet-Bloch theory. The representation of $\eta^{(1)}$ is
\begin{equation}
\label{eq:etaI-FB}
\eta^{(1)}(x,T) = -\sum_{j\ge 1}\int_{Y^*} \frac{\sin(\omega_{j,\delta}(\alpha)T)}{\omega_{j,\delta}(\alpha)} \langle f^{(1)}(x),\Phi_{j,\delta}(x,\alpha)\rangle_{L^2_\delta(\R^2)}\Phi_{j,\delta}(x,\alpha)\,d\alpha. 
\end{equation}
In view of \eqref{eq:bandall}, for all $j\ge 3$ we have $\omega_j(\alpha)\ge 1/C$ for some constant $C>0$. Hence, the contribution from $j\ge 3$ can be bounded as follows:
\begin{equation*}
\left\| \sum_{j\ge 3}\int_{Y^*} \frac{\sin(\omega_{j,\delta}(\alpha)T)}{\omega_{j,\delta}(\alpha)} \langle f^{(1)}(x),\Phi_{j,\delta}(x,\alpha)\rangle_{L^2_\delta(\R^2)}\Phi_{j,\delta}(x,\alpha)\,d\alpha \right\|_{L^2_\delta(\R^2)} \le C\|f^{(1)}\|_{L^2_\delta(\R^2)}.
\end{equation*}
It turns out, however, that the same method does not yield good $L^2_\delta$ estimates for the band $j=1,2$. Hence, we appeal to the estimate in Theorem \ref{thm:L2wave}.

\begin{proposition}\label{prop:eta1bound}
	Under the assumption of Theorem \ref{thm:main}, we can find a constant $C>0$ so that
	\begin{equation*}
		\|\eta^{(1)}(\cdot,T)\|_{L^2(\R^2)} \le C(\eps+T\eps\sqrt{\delta} )\left(\|\langle x\rangle^2 (|\nabla^2\bm{F}|+|\nabla\bm{F}|)\|_{L^\infty} + \|\langle \xi\rangle^4\widehat{\bm{F}}\|_{L^\infty}\right). 
	\end{equation*}
\end{proposition}
\begin{proof}
	We apply Theorem \ref{thm:L2wave} with $g=-f^{(1)}$. We have checked above that $\|f^{(1)}\|_{L^2_\delta(\mathbb{R}^2)} \le C \eps\|\langle x\rangle^2\nabla \bm{F}\|_{L^\infty}$. For the $L^2_\delta$ norm of $\nabla f^{(1)}$, we have
	\begin{equation*}
	\begin{aligned}
	-\nabla f^{(1)}(x) &= \eps^3 \eta_{\#,\delta}\left((\partial_1+\mi \partial_2)\nabla F_2\right)(\eps x) S_{1,\delta}(x) + \eps^3 \overline{\eta_{\#,\delta}}\left((\partial_1-\mi \partial_2)\nabla F_1\right)(\eps x) S_{2,\delta}(x)\\
	&\quad + \eps^2 \eta_{\#,\delta}\left((\partial_1+\mi \partial_2)F_2\right)(\eps x) (\nabla S_{1,\delta})(x) + \eps^2 \overline{\eta_{\#,\delta}}\left((\partial_1-\mi \partial_2)F_1\right)(\eps x) (\nabla S_{2,\delta})(x). 
	\end{aligned}
\end{equation*}
Applying Lemma \ref{lem:SpPL2bound}, we see that the contribution to $\|\nabla f^{(1)}\|_{L^2_\delta(\mathbb{R}^2)}$ from the $\eps^3$-terms can be bounded by
\begin{equation*}
C\eps^2|\eta_{\#,\delta}|\Big(\sum_{\ell=1,2}\|S_{\ell,\delta}\|_{L^2_\delta(Y)}\Big)\|(1+|x|^2)\nabla^2 \bm{F}\|_{L^\infty(\R^2)},
\end{equation*}
which is of order $\mathcal{O}(\eps^2)$. Similarly, the contribution from the $\eps^2$-terms can be bounded by
\begin{equation*}
C\eps|\eta_{\#,\delta}|\Big(\sum_{\ell=1,2}\|\nabla S_{\ell,\delta}\|_{L^2_\delta(Y)}\Big)\|(1+|x|^2)\nabla \bm{F}\|_{L^\infty(\R^2)},
\end{equation*}
which is of order $\mathcal{O}(\eps \sqrt{\delta})$ because $\|\nabla S_{j,\delta}\|_{L^2_\delta(Y)}$ can be bounded by a constant; see Remark \ref{rem:SHdel}. In summary, we have
\begin{equation}
\label{eq:f1Hdel}
\|\nabla f^{(1)}\|_{L^2_\delta(\R^2)} \le C\eps^2 \|(1+|x|^2)D^2 \bm{F}\|_{L^\infty(\R^2)} + C\eps\sqrt{\delta} \|(1+|x|^2)D \bm{F}\|_{L^\infty(\R^2)} .
\end{equation}

\smallskip

Next, with $r>0$ fixed so that \eqref{eq:Y*cut} holds, we estimate $\langle f^{(1)}, \Phi_{j,\delta}(\cdot,\alpha)\rangle_{L^2_{\delta}(\mathbb{R}^2)}$ for $j=1,2$ and for $\alpha \in Y^*$ satisfying $|\alpha-\alpha^*|>r$.
We only consider the first term in the formula \eqref{eq:f1formula}. 
Using Lemma \ref{f.w.lemma}(i) we get, for $\ell=1,2$ (in fact, the estimates below hold for all $\ell \in \N^*$),
\begin{equation*}
\begin{aligned}
&\langle \eps^2 \eta_{\#,\delta}\left((\partial_1+\mi \partial_2)F_2\right)(\eps x) S_{1,\delta}(x),\Phi_\ell(x,\alpha)\rangle_{L^2_\delta(\R^2)}\\
= \,&\eta_{\#,\delta}\sum_{\bm{m}\in \Z^2} [\mathscr{F}_{x\mapsto \xi}(\partial_1+\mi \partial_2)F_2]\left(\frac{\mathbf{m}\bm{\alpha}+\alpha-\alpha^*}{\eps}\right) \int_Y \sigma_\delta^{-1}(x) e^{\mi(\mathbf{m}\bm{\alpha}+\alpha-\alpha^*)\cdot x} S_{1,\delta}(x)\overline{\Phi_\ell(x)}\,dx.
\end{aligned}
\end{equation*}
Using the fact that $F_2\in \mathscr{S}(\R^2_x)$ and $\widehat{F}_2 \in \mathscr{S}(\R^2_\xi)$ and the fact that $|\eta_{\#,\delta}| = \mathcal{O}(\delta^{\frac12})$, we can bound the right-hand side above by (where $N \in \N_{>2}$ is to be chosen and $C$ depends on $N$)
\begin{equation*}
	C\delta^{\frac12}\||\xi|^{N+1}\widehat{F}_2\|_{L^\infty(\R^2_\xi)}\sum_{\mathbf{m}\in \Z^2} \frac{\eps^N}{|\mathbf{m}\bm{\alpha}+\alpha-\alpha^*|^N} \left|\int_Y \sigma_\delta^{-1}(x) e^{\mi(\mathbf{m}\bm{\alpha}+\alpha-\alpha^*)\cdot x} S_{1,\delta}(x)\overline{\Phi_\ell(x)}\,dx\right|.
\end{equation*}
Using Lemma \ref{f.w.lemma}(ii) and choosing $N>2$, note also that $\|S_{1,\delta}\|_{L^2_\delta(Y)} \le C\delta^{-\frac12}$ and $\|\Phi_\ell(\cdot,\alpha)\|_{L^2_\delta(Y)} = 1$, we verify that the above is bounded by $C_{\bm{F},N,r} \eps^N$. To conclude, we have
\begin{equation*}
\max_{|\alpha-\alpha^*|>r} |\langle f^{(1)}(\cdot) , \Phi_{j,\delta}(\cdot,\alpha) \rangle_{L^2_{\delta}(\mathbb{R}^2)} | \le C\eps^N \||\xi|^{N+1}\widehat{\bm{F}}\|_{L^\infty(\R^2_\xi)} .
\end{equation*}

Finally, we apply Theorem \ref{thm:L2wave} with the above estimates for $g=-f^{(1)}$, and conclude with the desired bound by choosing $N=3$ above.
\end{proof}

\subsection{Estimate of $\eta^{(2)}$}

The item $\eta^{(2)} = \eta^{(2)}_{\eps,\delta}$ solves the non-homogeneous $\mathcal{L}_\delta$-wave equation \eqref{eq:errorf2} 
with zero initial data and source term $-e^{\mi\omega^*_\delta t}f^{(2)} = -e^{\mi\omega^*_\delta t}f^{(2)}_{\eps,\delta}$ given in \eqref{eq:f123}. From \eqref{eq:f123}, we get
\begin{equation*}
	\begin{aligned}
		f^{(2)}(x,t) &= \eps^3 (\partial^2_t V_1)(\eps x,\eps t) S_{1,\delta}(x) + \eps^3 (\partial^2_t  V_2)(\eps x,\eps t) S_{2,\delta}(x),\\
		\partial_t f^{(2)}(x,t) &= \eps^4 (\partial^3_t V_1)(\eps x,\eps t) S_{1,\delta}(x) + \eps^3 (\partial^3_t  V_2)(\eps x,\eps t) S_{2,\delta}(x).
	\end{aligned}
\end{equation*}

Using the inequality in Lemma \ref{lem:SpPL2bound} and the inequality \eqref{eq:decayDirac} in Proposition \ref{prop:DiracFourierEstimate}, we get
\begin{equation*}
\begin{aligned}
\|f^{(2)}(x,t)\|_{L^2_\delta(\R^2_x)} &\le C\eps^2 \sum_{j=1,2}\|S_{j,\delta}\|_{L^2_\delta(Y)} \left\|\langle x\rangle^2 \partial_t^2 V_j(x,\eps t)\right\|_{L^\infty(\R^2_x)}\\
&\le C\eps^2 \sqrt{\delta}\|(I-\Delta_\xi)\langle\xi\rangle^2\widehat{\bm{F}}\|_{L^1(\R^2_\xi)}.
\end{aligned}
\end{equation*}
Similarly, we have
\begin{equation*}
\|\partial_t f^{(2)}(x,t)\|_{L^2_\delta(\R^2_x)} \le C\eps^3 \delta\|(I-\Delta_\xi)\langle\xi\rangle^3\widehat{\bm{F}}\|_{L^1(\R^2_\xi)}.
\end{equation*}
Note that those estimates are uniform in $t$.

\medskip

Using the Floquet-Bloch theory and Duhamel's formula, we have the representation
\begin{equation*}
	\eta^{(2)}(x,T) = -\sum_{j\ge 1}\int_{Y^*} \left(\int_0^T\frac{\sin(\omega_{j,\delta}(\alpha)(T-t))}{\omega_{j,\delta}(\alpha)} e^{\mi\omega^*_\delta t} \langle f^{(2)}(\cdot,t),\Phi_{j,\delta}(\cdot,\alpha)\rangle_{L^2_\delta(\R^2)}\,dt\right)\Phi_{j,\delta}(x,\alpha)\,d\alpha.
\end{equation*}
Using integration by parts in time we rewrite the time integral inside the bracket as
\begin{equation*}
\begin{aligned}
	-&\frac{\sin(\omega_{j,\delta}(\alpha)T)}{\mi \omega_{\delta}^*\omega_{j,\delta}(\alpha)}\langle f^{(2)}(\cdot,0),\Phi_{j,\delta}(\cdot,\alpha)\rangle_{L^2_\delta(\R^2)} \\
 &+ \int_0^T \frac{\cos(\omega_{j,\delta}(\alpha)(T-t))  e^{\mi \omega_{\delta}^* t}}{ \mi \omega_{\delta}^*} \langle f^{(2)}(\cdot,t),\Phi_{j,\delta}(\cdot,\alpha)\rangle_{L^2_\delta(\R^2)}\,dt\\
&-\int_0^T \frac{\sin(\omega_{j,\delta}(\alpha)(T-t)) e^{\mi \omega_{\delta}^*t}}{\mi \omega_{\delta}^*\omega_{j,\delta}(\alpha)} \langle \partial_t f^{(2)}(\cdot,t),\Phi_{j,\delta}(\cdot,\alpha)\rangle_{L^2_\delta(\R^2)}\,dt.
	\end{aligned}
\end{equation*}
We can bound $\sin(\omega_{j,\delta}(\alpha)(T-t))/\omega_{j,\delta}(\alpha)$ by $CT$, and bound $1/\omega_{\delta}^*$ by $C\delta^{-\frac12}$. It then follows that
\begin{equation*}
	\|\eta^{(2)}(\cdot,T)\|_{L^2_\delta(\mathbb{R}^2)} \le CT\delta^{-\frac12} \max_{0\le t\le T}\|f^{(2)}(\cdot,t)\|_{L^2_\delta(\mathbb{R}^2)} + CT^2\delta^{-\frac12} \max_{0\le t\le T}\|\partial_t f^{(2)}(\cdot,t)\|_{L^2_\delta(\mathbb{R}^2)}.
\end{equation*}

In view of the estimates obtained earlier for $f^{(2)}$ and $\partial_t f^{(2)}$, we hence have proved the following result.
\begin{proposition}\label{prop:eta2bound}
	Under the assumption of Theorem \ref{thm:main}, we can find a constant $C>0$ so that
	\begin{equation*}
		\|\eta^{(2)}(\cdot,T)\|_{L^2_{\delta}(\R^2)} \le C(T\eps^2 + T^2\eps^3 \delta^{\frac12})\textstyle\sum\limits_{k=2,3}\|(I-\Delta_\xi)\langle\xi\rangle^k\widehat{\bm{F}}\|_{L^1}. 
	\end{equation*}
\end{proposition}

\subsection{Estimate of $\eta^{(3)}_{\varepsilon,\delta}$}
\label{sec: estimate E3 final}

As before, we suppress the dependence on $\eps,\delta$ from the notations for $f^{(3)}_{\eps,\delta}$, $\eta^{(3)}_{\eps,\delta}$ and from the quantities derived by them. By Duhamel's principle and Floquet-Bloch representation, we have the formula
\begin{equation}
\begin{aligned}
		\eta^{(3)}(x,T) &= \sum_{j\ge 1} \int_{Y^*}  \widetilde{\eta^{(3)}}(j,\alpha;T) \Phi_{j,\delta}(x,\alpha)\,d\alpha;\\
  -\widetilde{\eta^{(3)}}(j,\alpha;T) &= \int_0^{T} \frac{\sin \big(  \omega_{j,\delta}(\alpha) (T-t) \big)  }{   \omega_{j,\delta}(\alpha)}   e^{\mathrm{i}\omega_{\delta}^* \tau} \langle f^{(3)}(x,t),\Phi_{j,\delta}(x,\alpha)\rangle_{L^2_{\delta}(\mathbb{R}^2)} \,dt.
\end{aligned}
\end{equation}
By the definition of $f^{(3)}$ and its action on $H^1_\delta$ via the Floquet transform, we have
	\begin{equation*}
		\begin{aligned}
		&\big\langle f^{(3)}(\cdot,t) , \Phi_{j,\delta}(\cdot) \big\rangle_{L^2_\delta(\R^2)} = \eps^2 \sum_{k=1,2} \int_{\R^2} \frac{dx}{\sigma_\delta(x)} \Big\{2\mi \omega^*(\partial_t V_k)(\eps x,\eps t)S_{k,\delta}(x) \overline{\Phi_{j}(x,\alpha)} \\
		&\qquad\qquad-(\nabla V_k)(\eps x,\eps t)\cdot \nabla S_{k,\delta}(x)\overline{\Phi_j(x,\alpha)} + (\nabla V_k)(\eps x,\eps t)S_{k,\delta}(x)\overline{\nabla \Phi_j(x,\alpha)} \Big\}.
		\end{aligned}
	\end{equation*}
In view of Lemma \ref{f.w.lemma}, for fixed $\alpha \in Y^*$, the integral over $\R^2_x$ in the last line can be computed in terms of the Fourier transforms of $\partial_t V_k$ and $\nabla V_k$. 

To simply the notation, we introduce the following pair of functions
\begin{equation}
  \label{eq:gammaPsi1}
\begin{aligned}
\gamma_1 &= 2\mi \omega^* \partial_t V_1, \quad &&\Psi_1 = \Psi_5 = S_{1,\delta},\\
\gamma_2 &= 2\mi \omega^* \partial_t V_2, \quad &&\Psi_2 = \Psi_6 = S_{2,\delta},\\
-\gamma_3 &= \gamma_5 = \nabla_x V_1, \quad &&\Psi_3 = \nabla_x S_{1,\delta},\\
-\gamma_4 &= \gamma_6 = \nabla_x V_2, \quad &&\Psi_4 = \nabla_x S_{2,\delta}.
\end{aligned}
\end{equation}
The pairing of $\langle f^{(3)}(\cdot,t),\Phi_j(\cdot,\alpha)\rangle$, for each fixed $j$, is then written as
\begin{equation}
  \label{eq:f3pairing}
\eps^2 \sum_{\ell=1}^4 
\langle \gamma_\ell(\eps\,\cdot,\eps t) \Psi_\ell(\cdot), \Phi_{j}(\cdot,\alpha)\rangle_{L^2_\delta} + \eps^2 \sum_{\ell=5}^6 \langle \gamma_\ell(\eps\,\cdot,\eps t) \Psi_\ell(\cdot), \nabla \Phi_j(\cdot,\alpha)\rangle_{L^2_\delta}.
\end{equation}
To be precise, the above pairings are not $L^2_\delta$ inner products \emph{per se}, since $\Psi_\ell$ and $\Phi_j$ are not elements in $L^2_\delta(\R^2)$. Nevertheless, their meaning in terms of weighted integrals are clear, and they can always be rewritten as $L^2_\delta$ inner product if we split $\gamma_\ell$ to both $\Psi_\ell$ and $\Phi_j$.


Combining the computation above, for any $T>0$, we have
\begin{equation*}
\begin{aligned}
-\widetilde{\eta^{(3)}}(j,\alpha,T) &= \int_0^{T} \frac{\sin(\omega_j(\alpha)(T-\tau))}{\omega_j(\alpha)}e^{\mi \omega^*_\delta \tau} \\
&\qquad \eps^2 \left(\sum_{\ell=1}^4 
\langle \gamma_\ell(\eps x,\eps \tau) \Psi_\ell(x), \Phi_{j}(x,\alpha)\rangle_{L^2_\delta} + \sum_{\ell=5}^6 \langle \gamma_\ell(\eps x,\eps\tau) \Psi_\ell(x), \nabla \Phi_j(x,\alpha)\rangle_{L^2_\delta}\right)\,d\tau.
\end{aligned}
\end{equation*}

We break the right-hand side above into three parts and call their contributions $\eta^{(3)}_{\rm I}, \eta^{(3)}_{\rm II}$ and $\eta^{(3)}_{\rm III}$, which are given by
 \begin{equation*}
 \begin{aligned} 
 \eta^{(3)}_{\rm I}(x,T) &= \sum_{j=1,2} \int_{Y^*} \mathbf{1}_{\{|\alpha-\alpha^*| < \eps^\nu \}}  \widetilde{\eta^{(3)}}(j,\alpha;T)\Phi_{j,\delta}(x,\alpha)\,d\alpha,\\
 \eta^{(3)}_{\rm II}(x,T) &= \sum_{j=1,2} \int_{Y^*} \mathbf{1}_{\{|\alpha-\alpha^*| \geq \eps^\nu  \}}  \widetilde{\eta^{(3)}}(j,\alpha;T)\Phi_{j,\delta}(x,\alpha)\,d\alpha,\\
   \eta^{(3)}_{\rm III}(x,T) &= \sum_{j\ge 3} \int_{Y^*} \widetilde{\eta^{(3)}}(j,\alpha;T)\Phi_{j,\delta}(x,\alpha)\,d\alpha,
	\end{aligned}
\end{equation*}
where $\nu \in (0,1)$ will be determined later. In the spectral domain, we see that the three parts in turn correspond to contributions of the sources from the first two subwavelength bands and close to the Dirac point, those from the first two subwavelength bands but with distance away from the Dirac point $\alpha^*$, and those from higher bands. 

We first focus on the term $\eta^{(3)}_{\rm I}$ and argue that for it to be small, $V_1(\cdot,t)$ and $V_2(\cdot,t)$ must evolve according to the Dirac equation \eqref{eq:Dirac}.  

\subsubsection{Estimate of $\eta^{(3)}_{\rm I}$}

By  formula \eqref{eq:f3pairing}, we have
\begin{equation}
  \label{eq:eta3I}
  \begin{aligned}
  &-\eta^{(3)}_{\rm I}(x,T) \\
  = & \sum_{j=1,2} \int_{\{|\alpha-\alpha^*|\le \eps^\nu \}} \int_0^T 
  \frac{\sin(\omega_j(\alpha)(T-\tau))}{\omega_j(\alpha)} e^{\mi \omega^* \tau} \Big\{\eps^2\sum_{\ell=1}^4 \langle \gamma_\ell(\eps\,\cdot,\eps\tau)\Psi_\ell(\cdot),\Phi_j(\cdot,\alpha)\rangle_{L^2_\delta} \\
  &\qquad \qquad \qquad + \eps^2\sum_{\ell=5,6} \langle \gamma_\ell(\eps\,\cdot,\eps\tau)\Psi_\ell(\cdot),\nabla\Phi_j(\cdot,\alpha)\rangle_{L^2_\delta}\Big\}\,dt\, \Phi_j(x,\alpha)\,d\alpha ,
\end{aligned}
\end{equation}
where $\gamma_\ell, \Psi_\ell$, for $\ell =1,\dots,6$, are given in \eqref{eq:gammaPsi1}. Using  formula \eqref{f.w.lemma}, for $\ell=1,2$ and $j=1,2$, we have 
\begin{equation}
\label{eq:eta3pair1}
  \begin{aligned}
	&\eps^2\langle \gamma_\ell(\eps\cdot,\eps\tau)\Psi_\ell(\cdot),\Phi_j(\cdot,\alpha)\rangle_{L^2_\delta} \\
	= &\sum_{\mathbf{m}\in \Z^2} 2\mi \omega^* \left(\partial_t \widehat{V}_\ell\right)\left(\frac{\mathbf{m}\bm{\alpha}+\alpha-\alpha^*}{\eps},\eps \tau\right)\int_Y \frac{e^{\mi(\mathbf{m}\bm{\alpha}+\alpha-\alpha^*)\cdot x}}{\sigma_\delta(x)} S_{\ell,\delta}(x)\overline{\Phi_j(x,\alpha)} \,dx.
  \end{aligned}
\end{equation}
Similarly, for $\ell=3,4$ and $j=1,2$, we have
\begin{equation}
\label{eq:eta3pair2}
  \begin{aligned}
&\eps^2\langle \gamma_\ell(\eps\cdot,\eps\tau)\Psi_\ell(\cdot),\Phi_j(\cdot,\alpha)\rangle_{L^2_\delta}\\
=&-\sum_{\mathbf{m}\in \Z^2}\widehat{\nabla V}_{\ell-2}\left(\frac{\mathbf{m}\bm{\alpha}+\alpha-\alpha^*}{\eps},\eps \tau\right)\cdot\int_Y \frac{e^{\mi(\mathbf{m}\bm{\alpha}+\alpha-\alpha^*)\cdot x}}{\sigma_\delta(x)}\nabla_x S_{\ell-2,\delta}(x)\overline{\Phi_j(x,\alpha)}\,dx\\
= &\frac{-\mi}{\eps}\sum_{\mathbf{m}\in \Z^2} \widehat{V}_{\ell-2}\left(\frac{\mathbf{m}\bm{\alpha}+\alpha-\alpha^*}{\eps},\eps \tau\right)\int_Y \frac{e^{\mi(\mathbf{m}\bm{\alpha}+\alpha-\alpha^*)\cdot x}}{\sigma_\delta(x)}(\mathbf{m}\bm{\alpha}+\alpha-\alpha^*)\cdot \nabla S_{\ell-2,\delta}(x)\overline{\Phi_j(x,\alpha)}\,dx.
\end{aligned}
\end{equation}
For $\ell=5,6$ and $j=1,2$, we have
\begin{equation}
\label{eq:eta3pair3}
  \begin{aligned}
&\eps^2\langle \gamma_\ell(\eps\cdot,\eps\tau)\Psi_\ell(\cdot),\nabla\Phi_j(\cdot,\alpha)\rangle_{L^2_\delta}\\
=&-\sum_{\mathbf{m}\in \Z^2}\widehat{\nabla V}_{\ell-4}\left(\frac{\mathbf{m}\bm{\alpha}+\alpha-\alpha^*}{\eps},\eps \tau\right)\cdot\int_Y \frac{e^{\mi(\mathbf{m}\bm{\alpha}+\alpha-\alpha^*)\cdot x}}{\sigma_\delta(x)}S_{\ell-4,\delta}(x)\overline{\nabla_x \Phi_j(x,\alpha)}\,dx\\
= &\frac{\mi}{\eps}\sum_{\mathbf{m}\in \Z^2} \widehat{V}_{\ell-4}\left(\frac{\mathbf{m}\bm{\alpha}+\alpha-\alpha^*}{\eps},\eps \tau\right)\int_Y \frac{e^{\mi(\mathbf{m}\bm{\alpha}+\alpha-\alpha^*)\cdot x}}{\sigma_\delta(x)} S_{\ell-4,\delta}(x)\overline{(\mathbf{m}\bm{\alpha}+\alpha-\alpha^*)\cdot \nabla \Phi_j(x,\alpha)}\,dx.
  \end{aligned}
\end{equation}

For the sums over $\mathbf{m}\in \Z^2$, we single out the contributions from $\mathbf{m}=(0,0)$ in \eqref{eq:eta3pair1}\eqref{eq:eta3pair2}\eqref{eq:eta3pair3} to $\eta^{(3)}_{\rm I}$ and call it $\eta^{(3)}_{{\rm I},0}$. Let $M_{\delta} := \| S_{1,\delta} \|_{L^2_{\delta}(Y)}  =  \| S_{2,\delta} \|_{L^2_{\delta}(Y)}$; recall that $M_\delta = \mathcal{O}(\delta^{-\frac{1}{2}})$ and for $j=1,2$, $S_{j,\delta}(x) = M_\delta \Phi_{j,\delta}(x,\alpha^*)$. Recall also that $\Phi_j(\cdot,\alpha)$ is $\alpha$-quasiperiodic, and $p_j(\cdot,\alpha) = \Phi_j(\cdot,\alpha)e^{-\mi \alpha\cdot x}$ is $\Lambda$-periodic. We denote the contribution from $\mathbf{m}=\mathbf{0}$ mentioned above by 
\begin{equation}
\label{eq:eta3I0def}
-\eta^{(3)}_{{\rm I},0}(x,T) := \sum_{j=1,2} M_\delta\int_{\{|\alpha-\alpha^*|\le \eps^\nu  \}} \int_0^T \frac{\sin(\omega_j(\alpha)(T-\tau))}{\omega_j(\alpha)} e^{\mi \omega^*\tau} K(\alpha,j,\tau)\,d\tau \, \Phi_j(x,\alpha) d\alpha,
\end{equation}
where
\begin{equation}
\label{eq:etaI0}
    \begin{aligned}
        &K(\alpha,j,\tau) \,:= \\
        &\;2\mi \omega^* \partial_t \widehat{V}_1(\frac{\beta}{\eps},\eps\tau) \left\langle p_{1,\delta}(\cdot,\alpha^*),p_{j,\delta}(\cdot,\alpha)\right\rangle 
  +2\mi \omega^* \partial_t \widehat{V}_2(\frac{\beta}{\eps},\eps\tau)\left\langle p_{2,\delta}(\cdot,\alpha^*),p_{j,\delta}(\cdot,\alpha)\right\rangle \\
  &\;-\widehat{V}_1(\frac{\beta}{\eps},\eps\tau) \frac{\mi\beta}{\eps}\cdot\Big(\left\langle (\mi\alpha^*+\nabla)p_{1,\delta}(\cdot,\alpha^*),p_{j,\delta}(\cdot,\alpha)\right\rangle - \left\langle p_{1,\delta}(\cdot,\alpha^*),(\mi\alpha+\nabla)p_{j,\delta}(\cdot,\alpha)\right\rangle\Big)\\
    &\;-\widehat{V}_2(\frac{\beta}{\eps},\eps\tau) \frac{\mi\beta}{\eps}\cdot\Big(\left\langle (\mi\alpha^*+\nabla)p_{2,\delta}(\cdot,\alpha^*),p_{j,\delta}(\cdot,\alpha)\right\rangle-\left\langle p_{2,\delta}(\cdot,\alpha^*),(\mi\alpha+\nabla)p_{j,\delta}(\cdot,\alpha)\right\rangle\Big) . 
    \end{aligned}
\end{equation}
Here and below, $\beta := \alpha-\alpha^*$ and it has a small amplitude since it ranges in $\{|\beta|\le \eps^\nu\}$. We use the expansion formula \eqref{eq:p12expansion} and the $L^2_\delta(Y)$ orthogonality of $\{p_{j,\delta}\}_j$ to get
\begin{equation*}
    \begin{aligned}
        &\langle p_{1,\delta}(\cdot,\alpha^*),p_{j,\delta}(\cdot,\alpha^*+\beta)\rangle = \frac{\overline{A(\beta)}}{\sqrt{2}} + \mathcal{O}(|\beta|+\delta), \qquad j=1,2.\\
        &\langle p_{2,\delta}(\cdot,\alpha^*),p_{j,\delta}(\cdot,\alpha^*+\beta)\rangle = \frac{(-1)^j}{\sqrt{2}} + \mathcal{O}(|\beta|+\delta), \qquad j=1,2.
    \end{aligned}
\end{equation*}
Using also \eqref{eq:keypair}, we have for $j=1,2$,
\begin{equation*}
    \begin{aligned}
        &\langle(\mi\alpha^*+\nabla)p_{1,\delta}(\cdot,\alpha^*),p_{j,\delta}(\cdot,\alpha)\rangle - \langle p_{1,\delta}(\cdot,\alpha^*),(\mi\alpha^*+\nabla)p_{j,\delta}(\cdot,\alpha)\rangle\\ 
        = &\frac{\overline{A(\beta)}}{\sqrt{2}}\langle \wtcA(\alpha^*)p_{1,\delta}(\cdot,\alpha^*),p_{1,\delta}(\cdot,\alpha^*)\rangle + \frac{(-1)^j}{\sqrt{2}} \langle \wtcA(\alpha^*)p_{1,\delta}(\cdot,\alpha^*),p_{2,\delta}(\cdot,\alpha^*)\rangle + \mathcal{O}(|\beta|+\delta)\\
        = &\frac{(-1)^j}{\sqrt{2}} \overline{a_\delta} (-1,\mi)^\top + \mathcal{O}(|\beta|+\delta).
    \end{aligned}
\end{equation*}
Similarly, for $j=1,2$, we also have
\begin{equation*}
    \langle(\mi\alpha^*+\nabla)p_2(\cdot,\alpha^*),p_j(\cdot,\alpha)\rangle - \langle p_2(\cdot,\alpha^*),(\mi\alpha^*+\nabla)p_j(\cdot,\alpha)\rangle = \frac{\overline{A(\beta)}}{\sqrt{2}}a_\delta (1,\mi)^\top + \mathcal{O}(|\beta|+\delta).
\end{equation*}
Note that $\mi (1,\mi)\cdot \xi= \mi \xi_1-\xi_2$ is the Fourier symbol of the differential operator $\partial_{x_1} + \mi\partial_{x_2}$. Similarly, $\mi (-1,\mi)\cdot \xi$ equals $-\mi\xi_1-\xi_2$ and is the Fourier symbol of $-\partial_{x_1}+\mi\partial_{x_2}$. Hence, if $(V_1,V_2)$ satisfies the Dirac system \eqref{eq:Dirac}, \emph{i.e.},
\begin{equation*}
\left\{
	\begin{aligned}
		&2\mi \omega^* \partial_t \widehat{V}_1(\xi,t) - a_\delta(\mi\xi_1-\xi_2) \widehat{V}_2(\xi,t) = 0,\\
		&2\mi \omega^* \partial_t \widehat{V}_2(\xi,t) + \overline{a_\delta}(\mi\xi_1+\xi_2) \widehat{V}_1(\xi,t) = 0,
	\end{aligned}
	\right.
\end{equation*}
many cancellations will result in the expression of $\eta^{(3)}_{{\rm I},0}$ and make this term small. More precisely, let $K(\alpha,j,\tau)$ be written as the sum of $K_1$ and $K_2$, where $K_1$ is defined by replacing the $p_{j,\delta}(\cdot,\alpha)$ in \eqref{eq:etaI0} by the error term 
\begin{equation*}
	e(p_{j,\delta}) = p_{j,\delta}(\cdot,\alpha) - \left(\frac{A(\beta)}{\sqrt{2}} p_{1,\delta}(\cdot,\alpha^*) + \frac{(-1)^j}{\sqrt{2}} p_{2,\delta}(\cdot,\alpha^*) \right)
\end{equation*}
of \eqref{eq:p12expansion}. That is,
\begin{equation*}
    \begin{aligned}
        K_1 (\alpha,j,t)&= 
  2\mi \omega^* \partial_t \widehat{V}_1(\frac{\beta}{\eps},\eps t) \left\langle p_{1,\delta}(\cdot,\alpha^*),e(p_{j,\delta}) \right\rangle 
  +2\mi \omega^* \partial_t \widehat{V}_2(\frac{\beta}{\eps},\eps t)\left\langle p_{2,\delta}(\cdot,\alpha^*),e(p_{j,\delta})\right\rangle \\
  &\quad\, -\widehat{V}_1(\frac{\beta}{\eps},\eps t) \frac{\mi\beta}{\eps}\cdot\left(\left\langle (\mi\alpha^*+\nabla)p_{1,\delta}(\cdot,\alpha^*),e(p_{j,\delta})\right\rangle - \left\langle p_{1,\delta}(\cdot,\alpha^*),(\mi\alpha+\nabla)e(p_{j,\delta}) \right\rangle\right)\\
    &\quad\, -\widehat{V}_2(\frac{\beta}{\eps},\eps t) \frac{\mi\beta}{\eps}\cdot\left(\left\langle (\mi\alpha^*+\nabla)p_{2,\delta}(\cdot,\alpha^*),e(p_{j,\delta})\right\rangle-\left\langle p_{2,\delta}(\cdot,\alpha^*),(\mi\alpha+\nabla)e(p_{j,\delta})\right\rangle\right) .
    \end{aligned}
\end{equation*}
The remaining term $K_2=K-K_1$ is then
\begin{equation*}
\begin{aligned}
  K_2(\alpha,j,t) &= \widehat{V}_1(\frac{\beta}{\eps},\eps t) \frac{|\beta|^2}{\eps}\cdot \left\langle p_{1,\delta}(\cdot,\alpha^*), \frac{A(\beta)}{\sqrt{2}}p_{1,\delta}(\cdot,\alpha^*)+\frac{(-1)^j}{\sqrt{2}}p_{2,\delta}(\cdot,\alpha^*) \right\rangle \\
  &\quad\, + \widehat{V}_2(\frac{\beta}{\eps},\eps t) \frac{|\beta|^2}{\eps} \cdot \left\langle p_{2,\delta}(\cdot,\alpha^*), \frac{A(\beta)}{\sqrt{2}}p_{1,\delta}(\cdot,\alpha^*)+\frac{(-1)^j}{\sqrt{2}}p_{2,\delta}(\cdot,\alpha^*) \right\rangle\\
  &=\widehat{V}_1(\frac{\beta}{\eps},\eps t) \frac{|\beta|^2}{\eps} \frac{\overline{A(\beta)}}{\sqrt{2}} + \widehat{V}_2(\frac{\beta}{\eps},\eps t) \frac{|\beta|^2}{\eps}  \frac{(-1)^j}{\sqrt{2}} .
  \end{aligned}
\end{equation*}
Using the bound $\|e(p_{j,\delta})\|_{H^1_\delta(Y)} \le C(\delta + |\beta|)$ in \eqref{eq:p12expansion}, the estimate $\|(\nabla+\mi \alpha^*) p_\ell(\cdot,\alpha^*)\|_{L^2_\delta(Y)} = |\omega^*| \le C\sqrt{\delta}$ and the estimate \eqref{eq:DiracF-2}, we see that the largest contribution to $|K_1|$ above come from the terms with $(\nabla+\mi \alpha)e(p_{j,\delta}(\cdot,\alpha))$ and get
\begin{equation*}
    |K_1(\alpha,j,t)| \le C(\delta+|\beta|)\|\bm{F}\|_{W^{2,1}} .
\end{equation*}
Taking derivatives of $K_1$ in $t$, and using the same argument above, we deduce also
\begin{equation*}
    \begin{aligned}
        | \partial_t^k K_1(\alpha,j,t )| \leq C \varepsilon^k \delta^{k/2} (\delta+|\beta|) \|\bm{F}\|_{W^{1+k,1}}, \qquad \forall\, k\in \N.
        \end{aligned}
\end{equation*}
By the same reasoning and using \eqref{eq:decayDirac}, for $K_2$, we also have
\begin{equation*}
    |\partial_t^k K_2(\alpha,j,\tau)|\leq C \varepsilon^{1+k} \sum_{\ell=1,2} \|\partial_t^k \widehat{V}_\ell(\xi)|\xi|^2\|_{L^\infty_\xi}  \leq C\varepsilon^{1+k} \delta^{k/2} \| \bm{F} \|_{W^{2+k,1}}, \qquad \forall \ k\in \mathbb{N}.
\end{equation*}
In view of $K=K_1+K_2$, combining the above estimates we get
\begin{equation}\label{eq:Kbound}
    |\partial_t^k K(\alpha,j,t)| \leq C\varepsilon^k \delta^{k/2}(\varepsilon + \delta + |\beta|)\| \bm{F} \|_{W^{2+k,1}}, \qquad \forall\,k\in \N, 
\end{equation}
uniformly in $t$ and $\alpha$.


\medskip

Now we control $\|\eta^{(3)}_{{\rm I},0}(x,T)\|_{L^2(\R^2)}$ (note, again, this control is not for the $L^2_\delta(\R^2)$ norm). In view of the definition \eqref{eq:eta3I0def} of $\eta^{(3)}_{{\rm I},0}(x,t)$ and using the Floquet-Bloch representation of $\mathcal{L}_\delta$-wave equation, we see that
\begin{equation}
\label{eq:wave_etaI0}
\left\{
    \begin{aligned}
         & (\partial_t^2 + \mathcal{L}_{\delta}) \eta^{(3)}_{{\rm I},0} = - e^{\mi\omega^* t} \sum_{j=1,2} M_\delta\int_{\{|\alpha-\alpha^*|\le \eps^\nu  \}} K(\alpha,j,t ) \Phi_j(x,\alpha)\,d\alpha, \\
         & \eta^{(3)}_{{\rm I},0}\big|_{t=0}=\partial_t  \eta^{(3)}_{{\rm I},0}\big|_{t=0} =0.
    \end{aligned}\right.
\end{equation}

Using the bound \eqref{eq:Kbound} for $k=0$ in the representation formula \eqref{eq:eta3I0def}, we obtain that
$$
\| \eta^{(3)}_{{\rm I},0}  \|_{H^1_{\delta}(\R^2)} \leq C_{\bm{F}} T \delta^{-1} \varepsilon^{\nu } ( \delta + \varepsilon^{\nu} )
$$
and, thus, 
\begin{equation}\label{eta3 interior}
    \|\eta^{(3)}_{{\rm I},0}\|_{H^1(D^{\Lambda})} \leq C_{\bm{F}} T( \varepsilon^{\nu} \delta^{1/2}+ \varepsilon^{2\nu} \delta^{-1/2} ).
\end{equation}
To control $\|\eta^{(3)}_{{\rm I},0}\|_{L^2(\Omega^{\Lambda})} $, we utilize Proposition \ref{prop: exterior nonzero} for the wave equation \eqref{eq:wave_etaI0} viewed from the exterior domain $\Omega^\Lambda$. For the source term there, we have 
\begin{equation}\label{control f in prop A2}
    T\max_{0\leq t\leq T} \left\|  e^{\mi\omega^* t} \sum_{j=1,2} M_\delta\int_{\{|\beta|\le \eps^\nu  \}} K(\alpha,j,t ) \Phi_j(x,\alpha)\,d\alpha\right\|_{L^2(\Omega^{\Lambda})} \leq C_{\bm{F}}T (\varepsilon^{\nu}\delta^{1/2} + \varepsilon^{2\nu}\delta^{-1/2} ).
\end{equation}
For the boundary term $h = \eta^{(3)}_{{\rm I},0}\rvert_{\partial \Omega^\Lambda}$, we control $\|h(\cdot,t)\|_{H^{\frac12}(\partial \Omega^\Lambda)}$ directly by \eqref{eta3 interior}; we still need to control  $\partial_t^2 h$ according to \eqref{formula:prop 6.4 II}. 
By a direct computation, $\partial_t^2 \eta^{(3)}_{{\rm I},0}$ satisfies problem
\begin{equation*}\left\{
    \begin{aligned}
         & (\partial_t^2 + \mathcal{L}_{\delta}) \partial_t^2  \eta^{(3)}_{{\rm I},0} = - e^{\mi\omega^* t} \sum_{j=1,2} M_\delta\int_{\{|\beta|\le \eps^\nu  \}} \partial_t^2 K(\alpha,j,t ) \Phi_j(x,\alpha)\,d\alpha  \\
        &\qquad \qquad\qquad\qquad \  - 2\mi \omega^* e^{\mi\omega^* t} \sum_{j=1,2} M_\delta\int_{\{|\beta|\le \eps^\nu  \}}  \partial_t K(\alpha,j,t ) \Phi_j(x,\alpha)\,d\alpha \\
        &\qquad \qquad\qquad\qquad \  +  (\omega^* )^2 e^{\mi\omega^* t} \sum_{j=1,2} M_\delta\int_{\{|\beta|\le \eps^\nu  \}}  K(\alpha,j,t) \Phi_j(x,\alpha)\,d\alpha ,\\
        &  \partial_t^2 \eta^{(3)}_{{\rm I},0}\big|_{t=0} =  -\sum_{j=1,2} M_\delta\int_{\{|\beta|\le \eps^\nu  \}}  K(\alpha,j,0 ) \Phi_j(x,\alpha)\,d\alpha,\\ 
        &  \partial_t\partial_t^2 \eta^{(3)}_{{\rm I},0}\big|_{t=0} = -   \sum_{j=1,2} 3M_\delta\int_{\{|\beta|\le \eps^\nu  \}} \big(  \mi\omega^* K(\alpha,j,0)+ \partial_t K(\alpha,j,0 ) \big) \Phi_j(x,\alpha)\,d\alpha.
    \end{aligned}\right.
\end{equation*}
We view this wave equation from the interior of $D^\Lambda$ (on each inclusion inside $D^\Lambda$). By the standard estimates for the wave operator $\partial_t^2+\mathcal{L}_\delta$ in $\R^2$ with respect to the $L^2_\delta$ and $H^1_\delta$ norms, we have
\begin{equation*}
\begin{aligned}
\|\partial_t^2 \eta^{(3)}_{{\rm I},0}(\cdot,T)\|_{H^1_\delta} \le \, &\|\partial_t^2 \eta^{(3)}_{{\rm I},0}(\cdot,0)\|_{H^1_\delta} + \|\partial_t\partial_t^2 \eta^{(3)}_{{\rm I},0}(\cdot,0)\|_{L^2_\delta} + \left\|\frac{\sin(\sqrt{\mathcal{L}_\delta} T)}{\sqrt{\mathcal{L}_\delta}}\partial_t\partial_t^2 \eta^{(3)}_{{\rm I},0}(\cdot,0)\right\|_{L^2_\delta} \\
& \; + T \max_{0\le t \le T} \|f(\cdot,t)\|_{L^2_\delta} + T\max_{0\le t \le T} \left\|\frac{\sin(\sqrt{\mathcal{L}_\delta}(T-t))}{\sqrt{\mathcal{L}_\delta}}f(\cdot,t)\right\|_{L^2_\delta} , 
\end{aligned}
\end{equation*}
where $f$ denotes the right-hand side in the wave equation above for $\partial_t^2 \eta^{(3)}_{{\rm I},0}$. By \eqref{eq:Kbound} and the estimates of $\omega^*$, $M_\delta$ and $|\omega_j(\alpha)|\ge \sqrt{\delta}/C$ for $j=1,2$ and $|\alpha-\alpha^*| \le \eps^\nu$, the above yields
\begin{equation*}
    \begin{aligned}
        \| &\partial_t^2  \eta^{(3)}_{{\rm I},0}(x,T)  \|_{H^1(D^{\Lambda})} \\
        &\leq  C\varepsilon^{\nu} \sum_{j=1,2} \max_{|\beta|\le \varepsilon^{\nu}} |K(\alpha,j,0)| +C \delta^{-1/2} \varepsilon^{\nu}
        \sum_{j=1,2}\big( \delta^{1/2} \max_{|\beta|\le \varepsilon^{\nu} }| K(\alpha,j,0)|+ \max_{|\beta|\le \varepsilon^{\nu} }| \partial_t K(\alpha,j,0)|\big) \\
        & \quad \, +CT\delta^{-1/2} \varepsilon^{\nu}\sum_{j=1,2}\big( \max_{0\leq t\leq T
        \atop |\beta|\le \varepsilon^{\nu}}|\partial_t^2 K(\alpha,j,t)|  + \delta^{1/2} \max_{0\leq t\leq T \atop |\beta|\le \varepsilon^{\nu} }|\partial_t K(\alpha,j,t)|+\delta  \max_{0\leq t\leq T \atop |\beta|\le \varepsilon^{\nu} }| K(\alpha,j,t)|\big) \\
        &\leq  C_{\bm{F}} ( \varepsilon^{\nu} \delta + \varepsilon^{2\nu} ) +C_{\bm{F}} T ( \varepsilon^{\nu} \delta^{3/2} + \varepsilon^{2\nu} \delta^{1/2}  ) .
    \end{aligned}
\end{equation*}
Therefore,
\begin{equation}\label{control boundary in A2}
        T\max_{0\leq t\leq T} \|\partial_t^2 \eta^{(3)}_{{\rm I},0}(\cdot,t)\|_{H^{\frac{1}{2}}(\partial \Omega^{\Lambda})} \leq CT\max_{0\leq t\leq T} \|\partial_t^2 \eta^{(3)}_{{\rm I},0}(\cdot,t)\|_{H^1(D^{\Lambda})} \leq C_{\bm{F}} T\big(1+  T \delta^{1/2}\big) (\varepsilon^{\nu} \delta + \varepsilon^{2\nu}) . 
\end{equation}

    
We combine \eqref{eta3 interior}\eqref{control f in prop A2}\eqref{control boundary in A2} and use Proposition \ref{prop: exterior nonzero} to conclude that
\begin{equation}
    \| \eta^{(3)}_{{\rm I},0}(\cdot,T)\|_{L^2(\mathbb{R}^2)}  \leq C_{\bm{F}} T(\delta^{-1/2}+T\delta^{1/2}) (\varepsilon^{\nu}\delta + \varepsilon^{2\nu}) .
\end{equation}

Next, we control the contributions from $\mathbf{m} \ne 0$ in \eqref{eq:eta3pair1}, \eqref{eq:eta3pair2} and \eqref{eq:eta3pair3}. Let $\eta^{(3)}_{{\rm I},1} = \eta^{(3)}_{\rm I} - \eta^{(3)}_{{\rm I},0}$, which accounts for the total contribution. To control it we use the fact from Lemma \ref{f.w.lemma}(iii), that is, we can find a small number $c=c(\nu)$ so that $|\mathbf{m}\bm{\alpha}+\alpha-\alpha^*|\ge c|\mathbf{m}|$. Then, we can use the decay of $\widehat{V}_\ell$ from \eqref{eq:DiracF-2}. Take the quantity in \eqref{eq:eta3pair1} for instance. The contribution from $\mathbf{m}\ne 0$ can be bounded by
\begin{equation*}
	\frac{C}{c^N}|\omega^*| \|S_{\ell,\delta}\|_{L^2_\delta(Y)} \delta^{\frac12} \|\bm{F}\|_{W^{N+1,1}} \sum_{\mathbf{m} \ne 0} \frac{\eps^N}{|\mathbf{m}|^N} \le C\delta^{\frac12}\eps^N,
\end{equation*}
by choosing $N\ge 3$. 

Similarly, we can control the contribution from $\mathbf{m}\ne 0$ in \eqref{eq:eta3pair2}. We identify the term $\xi \widehat{V}_\ell(\xi,\eps \tau)$ with $\xi = \eps^{-1}(\mathbf{m}\bm{\alpha}+\alpha-\alpha^*)$ there (or equivalently, inspect only the first equality there) and use the estimate \eqref{eq:DiracF-2}. We see then the contribution is bounded by
\begin{equation*}
	\frac{C}{c^N} \|\nabla S_{\ell,\delta}\|_{L^2_\delta(Y)} \|\bm{F}\|_{W^{N+1,1}} \sum_{\mathbf{m} \ne 0} \frac{\eps^N}{|\mathbf{m}|^N} \le C\eps^N
\end{equation*}
for $N\ge 3$. Note that we have also used the fact that $\|\nabla S_{\ell,\delta}\|_{L^2_\delta(Y)}$ can be uniformly bounded in $\delta$; see Remark \ref{rem:SHdel}. The control of \eqref{eq:eta3pair3} is similar, but we use the fact that $\|\nabla \Phi_j(\cdot,\alpha)\|_{L^2_\delta(Y)}$ is of order $\omega_j(\alpha)$ and bounded by $C\sqrt{\delta}$ because $j \in \{1,2\}$. Hence, the contribution from $\mathbf{m}\ne 0$ is bounded by
\begin{equation*}
	\frac{C}{c^N} \|S_{\ell,\delta}\|_{L^2_\delta(Y)} \|\nabla \Phi_j(\cdot,\alpha)\|_{L^2_\delta(Y)}\|\bm{F}\|_{W^{N+1,1}} \sum_{\mathbf{m} \ne 0} \frac{\eps^N}{|\mathbf{m}|^N} \le C\eps^N.
\end{equation*}
Using those estimates, and the bound of $\omega_j(\alpha)$ from below by $c\sqrt{\delta}$, we get
\begin{equation}
\label{eq:etaI1est}
	\begin{aligned}
	\|\eta^{(3)}_{{\rm I},1}(x,T)\|_{L^2_\delta(\R^2)} &\le C_{\bm{F}}M_\delta \delta^{-\frac12} T \left(\int_{Y^*} \mathbf{1}_{|\beta|\le \eps^\nu} \eps^{2N}\,d\beta\right)^{\frac12}
	\le C_{\bm{F}}T\eps^{N+\nu}\delta^{-1}.
	\end{aligned}
\end{equation}

\subsubsection{Estimate of $\eta^{(3)}_{\rm II}$}

By definition of $\eta^{(3)}_{\rm II}$, the formula \eqref{eq:f3pairing} and Minkowski inequality, we have
\begin{equation*}
	\begin{aligned}
		&\|\eta^{(3)}_{\rm II} (\cdot,T) \|^2_{L^2_{\delta}(\mathbb{R}^2)} \\
		= &\sum_{j=1,2} \int_{\{|\alpha-\alpha^*| > \eps^\nu\}} \Big|\int_0^T 
  \frac{\sin(\omega_j(\alpha)(T-\tau))}{\omega_j(\alpha)} e^{\mi \omega^* \tau} \Big\{\eps^2\sum_{\ell=1}^4 \langle \gamma_\ell(\eps\,\cdot,\eps\tau)\Psi_\ell(\cdot),\Phi_j(\cdot,\alpha)\rangle_{L^2_\delta} \\
  &\qquad \qquad \qquad + \eps^2\sum_{\ell=5,6} \langle \gamma_\ell(\eps\,\cdot,\eps\tau)\Psi_\ell(\cdot),\nabla\Phi_j(\cdot,\alpha)\rangle_{L^2_\delta}\Big\}\,dt\Big|^2\,d\alpha .
	\end{aligned}
\end{equation*}

We will control the above quantity by the same method as the one we have used for $\eta^{(3)}_{{\rm I},1}$. By Lemma \ref{f.w.lemma}(iii), we can find a constant $c = c(\nu) >0$ such that $|\mathbf{m}\bm{\alpha}+\alpha-\alpha^*| \ge c\eps^{\nu}(1+|\mathbf{m}|)$, and appeal to the decay of $\widehat{V}_\ell$. We use the crude bound
\begin{equation*}
\left|\frac{\sin(\omega_j(\alpha) (T-\tau))}{\omega_j(\alpha)}\right| \le C(T-\tau), \qquad \forall j, \alpha.
\end{equation*}
Then again we only need to estimate the pairings in \eqref{eq:eta3pair1}, \eqref{eq:eta3pair2}, and \eqref{eq:eta3pair3}. The estimates are then exactly the same as those in the last part of the previous subsection; we only need to replace $|\mathbf{m}|$ there by $\eps^\nu(1+|\mathbf{m}|)$. Hence, for $|\alpha-\alpha^*|>\eps^\nu$, the quantity in \eqref{eq:eta3pair1} is bounded by
\begin{equation*}
C\delta \|S_{\ell,\delta}\|_{L^2_\delta(Y)} \|\bm{F}\|_{W^{N+1,1}} \sum_{\mathbf{m} \in \Z^2} \frac{\eps^N}{\eps^{\nu N}(1+|\mathbf{m}|)^N} \le C\delta^{\frac12}\eps^{(1-\nu)N}.
\end{equation*}
The quantity in \eqref{eq:eta3pair2} is bounded by
\begin{equation*}
C\|\nabla S_{\ell,\delta}\|_{L^2_\delta(Y)} \|\bm{F}\|_{W^{N+1,1}} \sum_{\mathbf{m} \in \Z^2} \frac{\eps^N}{\eps^{\nu N}(1+|\mathbf{m}|)^N} \le C\eps^{(1-\nu)N}.
\end{equation*}
The quantity in \eqref{eq:eta3pair3} is bounded by
\begin{equation*}
C\|S_{\ell,\delta}\|_{L^2_\delta(Y)} \|\nabla \Phi_j(\cdot,\alpha)\|_{L^2_\delta(Y)}\sum_{\mathbf{m} \in \Z^2} \frac{\eps^N}{\eps^{\nu N}(1+|\mathbf{m}|)^N} \le C\eps^{(1-\nu)N}.
\end{equation*}

We hence conclude that for $N\ge 3$ and some $C>0$ depending on $N$,
\begin{equation}
	\label{eq:eta3IIest}
\|\eta^{(3)}_{\rm II}(\cdot,T)\|_{L^2_\delta(\R^2)} \le CT^2 \eps^{(1-\nu)N} \|\bm{F}\|_{W^{N+1,1}}.
\end{equation}

\subsubsection{Estimate of $\eta^{(3)}_{\rm III}$}

The term $\eta^{(3)}_{\rm III}$ accounts for contributions to $\eta^{(3)}$ from the higher ($j\ge 3$) bands. Then, $\omega_j(\alpha)$ is bounded from below by a constant uniformly for $j\ge 3$ and $\alpha \in Y^*$; see \eqref{eq:bandall}. 

By the formula \eqref{eq:f3pairing}, the term $\eta^{(3)}_{\rm III}$ is represented by
\begin{equation}
  \label{eq:eta3III}
  \begin{aligned}
  -\eta^{(3)}_{\rm III}(x,T) 
  = & \sum_{j\ge 3} \int_{Y^*} \int_0^T 
  \frac{\sin(\omega_j(\alpha)(T-\tau))}{\omega_j(\alpha)} e^{\mi \omega^* \tau} \Big\{\eps^2\sum_{\ell=1}^4 \langle \gamma_\ell(\eps\,\cdot,\eps\tau)\Psi_\ell(\cdot),\Phi_j(\cdot,\alpha)\rangle_{L^2_\delta} \\
  &\qquad \qquad \qquad + \eps^2\sum_{\ell=5,6} \langle \gamma_\ell(\eps\,\cdot,\eps\tau)\Psi_\ell(\cdot),\nabla\Phi_j(\cdot,\alpha)\rangle_{L^2_\delta}\Big\}\,d\tau\, \Phi_j(x,\alpha)\,d\alpha .
\end{aligned}
\end{equation}
Here and below, $\{\Gamma_\ell, \Psi_\ell\}_{\ell=1,\dots,6}$ are defined in \eqref{eq:gammaPsi1}. Using Euler's identity for $\sin$ function, we can rewrite the right hand side above as
\begin{equation*}
	\begin{aligned}
  \sum_{j\ge 3} \int_{Y^*} &\int_0^T 
  \frac{e^{-\mi(\omega_j(\alpha)-\omega^*)\tau} e^{\mi \omega_j(\alpha)T} - e^{\mi(\omega_j(\alpha)+\omega^*)\tau }e^{-\mi\omega_j(\alpha)T}}{2\mi \omega_j(\alpha)}\times \\
  &\Big\{\eps^2\sum_{\ell=1}^4 \langle \gamma_\ell(\eps\,\cdot,\eps\tau)\Psi_\ell(\cdot),\Phi_j(\cdot,\alpha)\rangle_{L^2_\delta} +\eps^2\sum_{\ell=5,6} \langle \gamma_\ell(\eps\,\cdot,\eps\tau)\Psi_\ell(\cdot),\nabla\Phi_j(\cdot,\alpha)\rangle_{L^2_\delta}\Big\}\,dt\, \Phi_j(x,\alpha)\,d\alpha.
  \end{aligned}
\end{equation*}
We see that $\|\eta^{(3)}_{\rm III}(x,T)\|_{L^2_\delta(\R^2_x)}^2 $ is bounded by the sum of $J_1$ and $J_2$, where 
\begin{equation*}
	\begin{aligned}
  J_1 := \sum_{j\ge 3} \int_{Y^*} \Big|\int_0^T 
  \frac{e^{-\mi(\omega_j(\alpha)-\omega^*)t}}{2\mi \omega_j(\alpha)} &\Big\{\eps^2\sum_{\ell=1}^4 \langle \gamma_\ell(\eps\,\cdot,\eps t)\Psi_\ell(\cdot),\Phi_j(\cdot,\alpha)\rangle_{L^2_\delta} \\
  &+ \eps^2\sum_{\ell=5,6} \langle \gamma_\ell(\eps\,\cdot,\eps t)\Psi_\ell(\cdot),\nabla\Phi_j(\cdot,\alpha)\rangle_{L^2_\delta}\Big\}\,dt\Big|^2 d\alpha
  \end{aligned}
\end{equation*}
and $J_2$ is defined with the complex exponential replaced by $e^{\mi(\omega_j(\alpha)+\omega^*)t}$. Note that $\omega^* \sim \sqrt{\delta}$ and $\omega_j(\alpha)\ge c$ for some constant $c>0$ uniformly for all $j\ge 3$ and $\alpha\in Y^*$. Hence, the complex exponentials in $J_1$ and $J_2$ are oscillatory in $t$. We bound $J_1$ and $J_2$ using this fact. Because of the treatments of these terms are exactly the same, we only consider $J_1$.

Using integration by parts in time, the time integral in $J_1$ can be written as
\begin{equation*}
	\begin{aligned}
		\frac{e^{-\mi(\omega_j(\alpha)-\omega^*)T}-1}{2(\omega_j(\alpha)-\omega^*)\omega_j(\alpha)} &\Big\{\eps^2\sum_{\ell=1}^4 \langle \gamma_\ell(\eps\,\cdot,\eps T)\Psi_\ell(\cdot),\Phi_j(\cdot,\alpha)\rangle_{L^2_\delta} + \eps^2\sum_{\ell=5,6} \langle \gamma_\ell(\eps\,\cdot,\eps T)\Psi_\ell(\cdot),\nabla\Phi_j(\cdot,\alpha)\rangle_{L^2_\delta}\Big\}\\
&-\int_0^T 
  \frac{(e^{-\mi(\omega_j(\alpha)-\omega^*)t}-1)}{2(\omega_j(\alpha)-\omega^*)\omega_j(\alpha)} \Big\{\eps^3\sum_{\ell=1}^4 \langle (\partial_t\gamma_\ell)(\eps\,\cdot,\eps t)\Psi_\ell(\cdot),\Phi_j(\cdot,\alpha)\rangle_{L^2_\delta} \\
  &\qquad\qquad\qquad\qquad\;+ \eps^3\sum_{\ell=5,6} \langle (\partial_t \gamma_\ell)(\eps\,\cdot,\eps t)\Psi_\ell(\cdot),\nabla\Phi_j(\cdot,\alpha)\rangle_{L^2_\delta}\Big\}\,dt\Big|^2 d\alpha .
	\end{aligned}
\end{equation*}
We first focus on the terms with $\ell=1,\dots,4$. Each of their contributions to $J_1$ can be bounded by the sum of two terms:
\begin{equation}
\label{eq:J1l1-1}
\sum_{j\ge 3} \int_{Y^*}\frac{\eps^4}{\omega_j^2(\alpha)(\omega_j(\alpha)-\omega^*)^2}  \left|\langle \gamma_\ell(\eps\,\cdot,\eps T)\Psi_\ell(\cdot),\Phi_j(\cdot,\alpha)\rangle_{L^2_\delta}\right|^2\,d\alpha,
\end{equation}
which takes care of the value at $T$ of the time integral representation, and 
\begin{equation}
\label{eq:J1l1-2}
\sum_{j\ge 3} \int_{Y^*}\frac{\eps^6}{\omega_j^2(\alpha)(\omega_j(\alpha)-\omega^*)^2}  \left|\int_0^T \langle (\partial_t \gamma_\ell)(\eps\,\cdot,\eps t)\Psi_\ell(\cdot),\Phi_j(\cdot,\alpha)\rangle_{L^2_\delta}\,dt\right|^2\,d\alpha.
\end{equation}
Because $j\ge 3$ and $\omega_j$ can be bounded from below by a positive constant uniformly in $j$ and $\alpha$, the above two terms can be bounded by
\begin{equation*}
C\sum_{j\ge 3} \int_{Y^*}\frac{\eps^4}{(1+\omega_j^2(\alpha))^2}  \left|\langle \gamma_\ell(\eps\,\cdot,\eps T)\Psi_\ell(\cdot),\Phi_j(\cdot,\alpha)\rangle_{L^2_\delta}\right|^2\,d\alpha,
\end{equation*}
and 
\begin{equation*}
C\sum_{j\ge 3} \int_{Y^*}\frac{\eps^6}{(1+\omega_j^2(\alpha))^2}  \left|\int_0^T \langle (\partial_t \gamma_\ell)(\eps\,\cdot,\eps t)\Psi_\ell(\cdot),\Phi_j(\cdot,\alpha)\rangle_{L^2_\delta}\,dt\right|^2\,d\alpha,
\end{equation*}
respectively. In view of the Floquet-Bloch representation of $H^{-1}_\delta(\R^2)$ norm in \eqref{eq:H-1exp}, the first term is bounded by $C\|\gamma_\ell(\eps\,\cdot,\eps T)\Psi_\ell(\cdot)\|_{H^{-1}_\delta(\R^2)}^2$. The second term can be bounded similarly. Using the simple fact that  $L^2_\delta$ is embedded in $H^{-1}_\delta$ continuously, and using also the triangle inequality, we see that the contribution to $J_1$ from the terms involving $\gamma_\ell$, $\ell=1,\dots,4$, is bounded by
\begin{equation}
\label{eq:J1l124}
C\eps^4\|\gamma_\ell(\eps\,\cdot,\eps T)\Psi_\ell(\cdot)\|_{L^2_\delta(\R^2)}^2 + C\eps^6 T\int_0^T  \|(\partial_t \gamma_\ell)(\eps\,\cdot,\eps t)\Psi_\ell(\cdot)\|^2_{L^2_\delta(\R^2)}\,dt.
\end{equation}
For the terms involving $\gamma_\ell$ with $\ell=5,6$, using integration by parts in $x$ (or equivalently, the definition of the $H^{-1}_\delta$ Floquet-Bloch transform and the characterization of $H^{-1}_\delta(\R^2)$ norm), we have
\begin{equation*}
	\begin{aligned}
	\langle \gamma_\ell(\eps\,\cdot,\eps t)\Psi_\ell(\cdot),\nabla\Phi_j(\cdot,\alpha)\rangle_{L^2_\delta} &= -\langle \sigma_\delta(\cdot)\nabla \left(\sigma_\delta^{-1} \gamma_\ell(\eps\,\cdot,\eps t)\Psi_\ell(\cdot)\right),\Phi_j(\cdot)\rangle_{H^{-1}_\delta,H^1_\delta},\\
	\langle \partial_t\gamma_\ell(\eps\,\cdot,\eps t)\Psi_\ell(\cdot),\nabla\Phi_j(\cdot,\alpha)\rangle_{L^2_\delta} &= -\langle \sigma_\delta(\cdot) \nabla\left(\sigma_\delta^{-1}(\partial_t \gamma_\ell)(\eps\,\cdot,\eps t)\Psi_\ell(\cdot)\right), \Phi_j(\cdot)\rangle_{H^{-1}_\delta,H^1_\delta}.
	\end{aligned}
\end{equation*}
It follows that the contribution to $J_1$ from the terms involving $\gamma_\ell$, $\ell=5,6$, is bounded by
\begin{equation*}
C\eps^4\|\sigma_\delta(\cdot)\nabla \left(\sigma_\delta^{-1} \gamma_\ell(\eps\,\cdot,\eps T)\Psi_\ell(\cdot)\right)\|_{H^{-1}_\delta(\R^2)}^2 +C\eps^6 T\int_0^T\|\sigma_\delta(\cdot) \nabla\left(\sigma_\delta^{-1}(\partial_t \gamma_\ell)(\eps\,\cdot,\eps t)\Psi_\ell(\cdot)\right)\|_{H^{-1}_\delta(\R^2)}^2 \,dt.
\end{equation*}
Let $\mathcal{P}$ be the operator $\mathcal{P}(u) = -\sigma_\delta \nabla (\sigma_\delta^{-1} u)$. Then in view of the relation
\begin{equation*}
	\langle \mathcal{P}(u),\psi\rangle_{H^{-1}_\delta,H^1_\delta} = \langle u,\nabla \psi\rangle_{L^2_\delta},
\end{equation*}
we get $\|\mathcal{P}u\|_{H^{-1}_\delta} \le \|u\|_{L^2_\delta}$. Hence, the contribution to $J_1$ from the terms involving $\gamma_\ell$, $\ell=5,6$, is bounded by
\begin{equation}
\label{eq:J1l526}
C\eps^4 \|\gamma_\ell(\eps\,\cdot,\eps T)\Psi_\ell(\cdot)\|_{L^2_\delta(\R^2)}^2 + C\eps^6 T \int_0^T \|(\partial_t \gamma_\ell)(\eps\,\cdot,\eps t)\Psi_\ell(\cdot)\|_{L^2_\delta(\R^2)}^2 \,dt.
\end{equation} 
Note that this has the same form as \eqref{eq:J1l124}.

It remains to control the terms in \eqref{eq:J1l124} and \eqref{eq:J1l526}. We bound the $H^{-1}_\delta$ norms by $L^2_\delta$ norms and use Lemma \ref{lem:SpPL2bound} and Proposition \ref{prop:DiracFourierEstimate}. For $\ell=1,2$, $\gamma_\ell = 2\mi \omega^* \partial_t V_\ell$, $\Psi_\ell = S_{\ell,\delta}$ we get, for all $t\ge 0$, $k=1,2$,
\begin{equation*}
	\begin{aligned}
		\eps\|2\mi \omega^* \partial^k_t V_\ell(\eps x,\eps t) S_{\ell,\delta}\|_{L^2_\delta(\R^2)} \le C\delta^{\frac k2}\|(I-\Delta_\xi)\langle \xi\rangle^k \widehat{\bm{F}}\|_{L^1_\xi},
	\end{aligned}
\end{equation*}
where we have also used that $\|S_{\ell,\delta}\|_{L^2_\delta(Y)}\le C\delta^{-\frac12}$, $|\omega^*| \sim \sqrt{\delta}$ and the inequality \eqref{eq:decayDirac}. The contributions of those terms to \eqref{eq:J1l124} is then bounded by
\begin{equation}
\label{eq:J1contri-1}
	C(\eps^2\delta + \eps^4\delta^2 T^2) \left(\|(I-\Delta_\xi)\langle \xi\rangle \widehat{\bm{F}}\|^2_{L^1_\xi} + \|(I-\Delta_\xi)\langle \xi\rangle^2 \widehat{\bm{F}}\|^2_{L^1_\xi} \right).
\end{equation}

\smallskip

Similarly, for $\ell=3,4$, $\gamma_\ell = -\nabla_x V_{\ell-2}$, $\Psi_\ell = \nabla_x S_{\ell-2,\delta}$, we get that, for all $t\ge 0$, $k=0,1$,
\begin{equation*}
	\begin{aligned}
		\eps\|(\nabla_x \partial^k_t V_{\ell-2})(\eps x,\eps t) \cdot \nabla_x S_{\ell-2,\delta}\|_{L^2_\delta(\R^2)} \le C\delta^{\frac k2}\|(I-\Delta_\xi)\langle \xi\rangle^{k+1} \widehat{\bm{F}}\|_{L^1_\xi},
	\end{aligned}
\end{equation*}
where we have used that $\|\nabla S_{\ell,\delta}\|_{L^2_\delta(Y)}\le C$ and the inequality \eqref{eq:decayDirac}. The contributions of those terms to \eqref{eq:J1l124} is then bounded by \eqref{eq:J1contri-1} as well. 

\smallskip

For $\ell=5,6$, $\gamma_\ell = \nabla_x V_{\ell-4}$, $\Psi_\ell = S_{\ell-4}$, and we get: for all $t\ge 0$, and $k=0,1$,
\begin{equation*}
	\begin{aligned}
		\eps\|(\nabla_x \partial^k_t V_{\ell-4})(\eps x,\eps t) S_{\ell-4,\delta}\|_{L^2_\delta(\R^2)} \le C\delta^{\frac k2-\frac12}\|(I-\Delta_\xi)\langle \xi\rangle^{k+1} \widehat{\bm{F}}\|_{L^1_\xi}. 
	\end{aligned}
\end{equation*}
The contributions of those terms to \eqref{eq:J1l526} is then bounded by
\begin{equation}
\label{eq:J1contri-2}
	C(\eps^2\delta^{-1} + \eps^4 T^2) \left(\|(I-\Delta_\xi)\langle \xi\rangle \widehat{\bm{F}}\|^2_{L^1_\xi} + \|(I-\Delta_\xi)\langle \xi\rangle^2 \widehat{\bm{F}}\|^2_{L^1_\xi} \right).
\end{equation}

\medskip

We conclude that there is a constant $C>0$ such that
\begin{equation}
	\label{eq:eta3IIIest}
\|\eta^{(3)}_{\rm III}(x,T)\|_{L^2_\delta(\R^2)} \le C(\eps\delta^{-\frac12} + T\eps^2)\left(\|(I-\Delta_\xi)\langle \xi\rangle \widehat{\bm{F}}\|_{L^1_\xi} + \|(I-\Delta_\xi)\langle \xi\rangle^2 \widehat{\bm{F}}\|_{L^1_\xi} \right) .
\end{equation}


\begin{proposition}\label{prop:eta3bound}
	Under the assumption of Theorem \ref{thm:main}, for $\nu \in (0,1)$ and any integer $N\ge 2$, we can find a constant $C_{\bm{F}}>0$ depending on $\bm{F}$ and $N$ so that
	\begin{equation*}
 \begin{aligned}
		\|\eta^{(3)}(\cdot,T)\|_{L^2(\R^2)} &\le C_{\bm{F}} \eps\delta^{-\frac12} + C_{\bm{F}} T(\eps^{N+\nu}\delta^{-1} + \varepsilon^{\nu} \delta^{1/2}+ \varepsilon^{2\nu} \delta^{-1/2} )\\
    &\quad\, +C_{\bm{F}} T^2 (
    \varepsilon^{\nu} \delta^{3/2} + \varepsilon^{2\nu} \delta^{1/2} + \eps^{(1-\nu)N}). 
    \end{aligned}
	\end{equation*}
\end{proposition}

\section*{Acknowledgments}
This work was partially supported by Swiss National Science Foundation grant number 200021--200307 and by the New Cornerstone Science Foundation grant - NCI202310.  The support of Xin Fu by the Exchanging  Graduate Student Program at 
Tsinghua University is also acknowledged. 

\appendix

\section{Useful estimates via Floquet-Bloch representation formula}
\label{sec:proofL2wave}

In this appendix we prove Theorem \ref{thm:L2wave}
 and Proposition \ref{prop: exterior nonzero}. The basic idea is to consider the wave operators in and outside the inclusions separately. We will use a harmonic extension type operator. 
\subsection{Floquet-Bloch theory for the exterior Laplacian} 

To introduce the harmonic extension operator, we need to analyze the following Poisson problem in perforated domain $\Omega^\Lambda = \R^2\setminus \ol{D}^\Lambda$. Given $f\in L^2(\Omega^\Lambda)$, we seek for $v \in H^1_0(\Omega^\Lambda) = \{v\in H^1(\Omega^\Lambda) \,:\, v = 0\;\text{on}\, \partial \Omega^\Lambda\}$ such that
\begin{equation}
	\label{eq:exPoisson}
-\Delta v = f \qquad  \text{in }\  \Omega^\Lambda,
\end{equation}
The equation is understood in the weak sense: $v$ is a solution if $v\in H^1_0(\Omega^\Lambda)$ such that
\begin{equation*}
	\int_{\Omega^\Lambda} \nabla v\cdot \nabla \varphi = \int_{\Omega^\Lambda} f\varphi\,dx, \qquad \forall \varphi \in H^1_0(\Omega^\Lambda).
\end{equation*}
Using the following version of Poincar\'e's inequality (see \cite{MR1079190,MR4075336} for a proof):
\begin{equation*}
	\|v\|_{L^2(\mathbf{m}\bm{l}+\Omega)} \le C\|\nabla v\|_{L^2(\mathbf{m}\bm{l}+\Omega)}, \quad \forall\,\mathbf{m}\in \Z^2, \, \forall v\in H^1_0(\Omega^\Lambda),
\end{equation*}
we deduce that $\|v\|_{L^2(\Omega^\Lambda)} \le C\|\nabla v\|_{L^2(\Omega^\Lambda)}$. Hence, using the Lax-Milgram theorem, we see that there exists a unique solution $v\in H^1_0(\Omega^\Lambda)$ for \eqref{eq:exPoisson} which satisfies
\begin{equation}
\label{eq:exPestimate}
\|v\|_{L^2(\Omega^\Lambda)} + \|\nabla v\|_{L^2(\Omega^\Lambda)} \le C\|f\|_{L^2(\Omega^\Lambda)}.
\end{equation}
Invoking the standard elliptic regularity theory, we actually have $v\in H^2(\Omega^\Lambda)$ as well.

\medskip

Let $-\Delta_{\rm ex}$ be an unbounded operator on $L^2(\Omega^\Lambda)$ as follows:
\begin{equation*}
\begin{aligned}
	-\Delta_{\rm ex} \; :\;  \mathrm{Dom}\,(\Delta_{\rm ex}) \subset L^2(\Omega^{\Lambda}) \, &\rightarrow \, L^2(\Omega^{\Lambda})\\
	v \, &\mapsto \, -\Delta v
	\end{aligned}
\end{equation*}
where the domain of $-\Delta_{\rm ex}$ is 
\begin{equation*}
	\mathrm{Dom}\,(-\Delta_{\rm ex})  = \{ u \in H^1(\Omega^{\Lambda}) : u|_{\partial \Omega^{\Lambda}}  =0,\, \Delta u \in  L^2(\Omega^{\Lambda}) \}.
\end{equation*}
Similar to the framework outlined in Section \ref{sec:FB_L}, we can develop the Floquet-Bloch theory for $\Delta_{\rm ex}$. For each $\alpha \in Y^*$, we define the space
\begin{equation*}
	V(\alpha) := \{ u \in L^2_{\rm loc}(\Omega^\Lambda): u(x+ l)  = e^{\mi\alpha\cdot l} u(x) \ \mathrm{for}\ \mathrm{all} \ l \in \Lambda   \}
\end{equation*}
equipped with the inner product $\langle u,v \rangle_{V} = \int_{Y\setminus \ol D} u(x) \overline{v(x)}\,dx$.

For each $\alpha \in Y^*$, we define the unbounded operator
\begin{equation*}
\begin{aligned}
	-\Delta_{\Omega^\Lambda} (\alpha) \;:\;  \mathrm{Dom}\,( 	\Delta_{\Omega^\Lambda}(\alpha) )  \subset V(\alpha) \, &\rightarrow \, V(\alpha)\\
	u \, &\mapsto \, -\Delta u
	\end{aligned}
\end{equation*}
where the domain of definition is
\begin{equation*}
	\mathrm{Dom}\,( -\Delta_{\Omega^\Lambda}(\alpha) )  = \{ u \in H^1_{\rm loc}(\Omega^\Lambda) \cap V(\alpha): u|_{\partial D^\Lambda} = 0, \,\Delta u \in V(\alpha)\}.
\end{equation*}
In fact, for $\alpha \in Y^*$, we may view $-\Delta_{\Omega^\Lambda}(\alpha)$ as the Laplacian acting on functions defined on $\Omega=Y\setminus \ol{D}$ with an $\alpha$-quasiperiodic boundary condition at $\partial Y$ and with a homogeneous Dirichlet boundary condition on $\partial D$. 

It follows from standard spectral theory that $-\Delta_{\Omega^\Lambda}(\alpha)$ has compact resolvent and its spectrum consists of a sequence of positive eigenvalues:
\begin{equation*}
	0 < \mu_1^2(\alpha) \leq \mu_2^2(\alpha) \leq \cdots \leq \mu_j^2(\alpha) \leq \cdots,
\end{equation*}
counted with their multiplicities, and tending to $+\infty$. It is also known that $\alpha\mapsto \mu_1(\alpha)$ is Lipschitz, $\Lambda^*$-periodic and can be uniformly bounded from below.

As is standard, we can find a family of eigenfunctions $\Psi_j(\cdot,\alpha) \in H^1_{\rm loc}(\Omega^\Lambda) \cap V(\alpha)$ associated to the eigenvalue $\mu_j^2(\alpha) $ which satisfy
\begin{equation*}\left\{
	\begin{aligned}
		& -\Delta_{\Omega^\Lambda} (\alpha)\Psi_j(\cdot,\alpha) = \mu_j^2(\alpha) \Psi_j(\cdot,\alpha) , \\
  & \Psi_j(\cdot,\alpha) |_{\partial D} =0,\\
		& \langle \Psi_j(\cdot,\alpha) , \Psi_k(\cdot,\alpha) \rangle_{V} = \delta_{jk}, \quad \mathrm{for\ any\ }j,k  \geq 1.
	\end{aligned}\right. 
\end{equation*}
In analogy with Section \ref{sec:FB_L}, the family of eigenfunctions $\{\Psi_j(\cdot,\alpha)\}_{j\in \N^*,\alpha\in Y^*}$ can be used as the Floquet-Bloch basis associated to $-\Delta_{\rm ex}$ in $L^2(\Omega^\Lambda)$. In particular, we have the following decomposition:
\begin{equation*}\left\{
	\begin{aligned}
		& u(x) = \sum_{j=1}^{\infty} \int_{Y^*}\langle u(\cdot) , \Psi_j(\cdot,\alpha) \rangle_{L^2(\R^2)} \Psi_j(x,\alpha) \,d\alpha, \\
		& \|u \|_{L^2(\Omega^{\Lambda})}^2 = \sum_{j=1}^{\infty} \int_{Y^*} \left|\langle u(\cdot) , \Psi_j(\cdot,\alpha) \rangle_{L^2(\R^2)}\right|^2 \,d\alpha ,
	\end{aligned}\right.
	\quad \quad \mathrm{for\ any \ } u \in L^2(\Omega^{\Lambda}).
\end{equation*} 
Here, in order to compute the pairing $\langle \cdot,\cdot \rangle_{L^2(\R^2)}$, the basis functions $\Psi_j(\cdot,\alpha)$ are extended to zero in $D^\Lambda$. As an application of the above Floquet-Bloch theory for $\Delta_{\rm ex}$, we see that 
\begin{equation*}
	-\Delta_{\rm ex} u = \sum_{j=1}^{\infty} \int_{Y^*} \mu_j^2(\alpha)\langle u(\cdot) , \Psi_j(\cdot,\alpha) \rangle_{L^2(\R^2)} \Psi_j(x,\alpha) \,d\alpha , \quad \mathrm{for\ any\ }u \in 	\mathrm{Dom}\,(\Delta_{\rm ex}) .
\end{equation*}

Finally, we define a harmonic extension operator from functions defined on $\partial \Omega^\Lambda$ to functions defined on $\Omega^\Lambda$. Consider the trace operator
 \begin{equation*}
 	\bTtrace : H^1(\Omega^{\Lambda}) \rightarrow H^{\frac{1}{2}} (\partial \Omega^{\Lambda}). 
 \end{equation*}
Note that $\partial \Omega^\Lambda = \partial D^\Lambda$, and that the domain $D^\Lambda$ is a disjoint union of periodic copies of $D$. Hence, the above trace operator can be defined piece by piece from $\mathbf{m}\bm{l}+\Omega \to \mathbf{m}\bm{l} + \partial D$ for each $\mathbf{m}\in \Z^2$ using the standard trace operator, and hence $\bTtrace$ is bounded and surjective. Similarly, we may define a right inverse $\bTtrace$ that is bounded from $H^{\frac12}(\partial D^\Lambda)$ into $H^1(\Omega^\Lambda)$, and $\bTtrace \bTtrace^{-1}$ is the identity operator on $H^{\frac12}(\partial D^\Lambda)$. We define the bounded harmonic extension $\bEext$ by
\begin{equation*}
	\begin{aligned}
		\bEext \;:\; H^{\frac{1}{2}}(\partial \Omega^{\Lambda}) \, &\rightarrow\,  H^1(\Omega^{\Lambda}), \\
		f \, &\mapsto\,  \bTtrace^{-1} f + (-\Delta_{\rm ex})^{-1} \Delta \bTtrace^{-1} f   .
	\end{aligned}
\end{equation*}
We say $\bEext$ is a harmonic extension because
\begin{equation*}
\bTtrace (\bEext f) = f \; \text{on } \partial \Omega^\Lambda, \qquad \Delta (\bEext f) = 0\; \text{in }\; \Omega^\Lambda.
\end{equation*}

\subsection{Estimates for the wave equation in a perforated domain}

In this subsection, we analyze the initial boundary value problem for the wave equation in the perforated domain $\Omega^\Lambda$. Using linearity,  we treat the non-zero initial data and non-zero boundary data separately.

\begin{lemma}\label{lem:LapExterior}
	There is a universal constant $C>0$ such that if $w$ solves the exterior problem
	\begin{equation}
 \label{eq:wave_bd0}
		\left\{
		\begin{aligned}
			& (\partial_t^2 - \Delta ) v =  0 && \mathrm{in} \ \Omega^{\Lambda} \times [0,\infty) , \\
			& v(x,0) =0, \quad  \partial_t v (x,0) =g(x) && \mathrm{on} \ \Omega^{\Lambda} ,\\
			& v(x,t) = 0 && \mathrm{on} \ \partial \Omega^{\Lambda}   \times [0,\infty),
		\end{aligned}\right.  
	\end{equation}
then we have
\begin{equation*}
	\| v(\cdot,T) \|_{L^2(\Omega^{\Lambda})} \leq  C\| g\|_{L^2(\Omega^{\Lambda})}, \qquad \forall \, T\in \R.
\end{equation*}
\end{lemma}
\begin{proof}
	Using the Floquet-Bloch theory for $\Delta_{\rm ex}$, we represent $w$ as
	\begin{equation*}
			v(x,T) = \frac{\sin(\sqrt{-\Delta_{\rm ex}}T)}{\sqrt{-\Delta_{\rm ex}}} g(x) = \sum_{j\geq 1} \int_{Y^*} \frac{\sin(\mu_j(\alpha)T) }{\mu_j(\alpha)}\langle g(\cdot) , \Psi_j(\cdot,\alpha) \rangle_{L^2(\Omega^{\Lambda})} \Psi_j(x,\alpha)\,d\alpha.
	\end{equation*}
Because $\mu_j(\alpha)$ can be uniformly bounded from below with respect to $j \in \N^*$ and $\alpha \in Y^*$, we obtain that
\begin{equation*}
	\| v(\cdot,t) \|_{L^2(\Omega^{\Lambda})} \leq  C\| g\|_{L^2(\Omega^{\Lambda})}  .
\end{equation*}
This completes the proof.
\end{proof}

\subsection{Proof of Proposition \ref{prop: exterior nonzero}}
\label{sec:proof:propextwave}

This proposition yields estimates for the wave equation in the exterior domain $\Omega^\Lambda$ equipped with zero initial data but non-homogeneous boundary condition on $\partial \Omega^\Lambda$. To derive the estimates we need the harmonic extension operator introduced before.

Take $h(x,t)$ satisfying $h(\cdot,0)=\partial_t h(\cdot,0)=0$. For each $t>0$, let $\bEext h(\cdot, t)$ be the harmonic extension of $h(\cdot,t)$. Then, $\bEext h$ satisfies
	\begin{equation*}\left\{
		\begin{aligned}
			& 	\Delta (\bEext h)(\cdot, t) = 0  && \mathrm{in} \ \Omega^{\Lambda} , \\
			& \bEext h(\cdot, t)= h(\cdot,t) && \mathrm{on} \ \partial \Omega^{\Lambda} .
		\end{aligned}\right.
	\end{equation*}
In view of the initial condition, we have 
	\begin{equation*}
		\bEext h(x,0) = 0, \quad \partial_t \bEext h(x,0) = 0, \quad \mathrm{on} \ \Omega^{\Lambda} ,
	\end{equation*}
	
	Let $w (\cdot,t)=v (\cdot,t)- \bEext h(\cdot, t )$. Then, $w (x,t)$ satisfies 
	\begin{equation*}
		\left\{
		\begin{aligned}
			& (\partial_t^2 - \Delta ) w(x,t) = f - \partial_t^2 \bEext h(x,t)&& \mathrm{in} \ \Omega^{\Lambda} \times [0,\infty) , \\
			& w(x,0) =0, \quad  \partial_t w (x,0) =0 && \mathrm{on} \ \Omega^{\Lambda} , \\
			& w(x,t) = 0  && \mathrm{on} \ \partial \Omega^{\Lambda}  \times [0,\infty) .
		\end{aligned}\right.  
	\end{equation*}
By Duhamel's principle, we get
	\begin{equation}\label{formula A.2 w}
			w(x,T) =  \int_0^T \frac{\sin(\sqrt{-\Delta_{\rm ex}}(T-t))}{\sqrt{-\Delta_{\rm ex}}}  (f - \partial_{t}^2 \bEext h)(x,t)\,dt.
	\end{equation}
By Floquet-Bloch theory associated with the operator $-\Delta_{\rm ex}$, we obtain that
	\begin{equation*}
		v (x, T) = \bEext h(x,T) + \sum_{j\geq 1} \int_{Y^*} \int_0^T    \frac{\sin (\mu_j(\alpha) (T-t))}{\mu_j(\alpha)} \langle f- \partial^2_t \bEext h (\cdot, t), \Psi_j(\cdot,\alpha) \rangle_{L^2(\Omega^{\Lambda})} \Psi_j(x,\alpha)\,dt d\alpha.
	\end{equation*}
Because $\mu_j(\alpha)$ can be uniformly bounded from below with respect to $j \in \N^*$ and $\alpha \in Y^*$, we get
	\begin{equation*}
		\begin{aligned}
				\| v(\cdot, T) \|^2_{L^2(\Omega^{\Lambda})} &\leq \| \bEext h(\cdot, T) \|^2_{L^2(\Omega^{\Lambda})} + C T \int_0^T \| f- \partial^2_{t}\bEext h (\cdot, t) \|^2_{L^2(\Omega^{\Lambda})} \,dt \\
				& \leq C \| h(\cdot, T) \|^2_{H^{\frac{1}{2}}(\partial \Omega^{\Lambda})} + CT \int_0^T \|f \|_{L^2(\Omega^{\Lambda})}^2\,dt +CT \int_0^T \| \partial^2_{t}h(\cdot, t) \|^2_{H^{\frac{1}{2}}(\partial \Omega^{\Lambda})} \,dt.
		\end{aligned}
	\end{equation*}
Therefore, we conclude that, uniformly for $T>0$,
	\begin{equation*}
		\| v(\cdot, T) \|_{L^2(\Omega^{\Lambda})}  \leq C \| h(\cdot, T) \|_{H^{\frac{1}{2}}(\partial \Omega^{\Lambda})} +CT\max_{0\leq t\leq T} \| f \|_{L^2(\Omega^{\Lambda})} + CT \max_{0\leq t \leq T}\| \partial^2_{t}h(\cdot, t) \|_{H^{\frac{1}{2}}(\partial \Omega^{\Lambda})}  .
	\end{equation*}
 This proves \eqref{formula:prop 6.4}.

Moreover, we note that $\partial_t \mathcal{E} h (x,0)=0$. So, using integration by parts in \eqref{formula A.2 w}, we obtain that
	\begin{equation*}
			w(x,T) = \int_0^T \frac{\sin(\sqrt{-\Delta_{\rm ex}}(T-t))}{\sqrt{-\Delta_{\rm ex}}}  f (x,t)\,dt - \int_0^T \cos(\sqrt{-\Delta_{\rm ex}}(T-t))\partial_{t} \bEext h(x,t)\,dt.
	\end{equation*}
Then, we can prove \eqref{formula:prop 6.4 II} using the same argument above.

\subsection{Proof of Theorem \ref{thm:L2wave}}
\label{sec:L2wave}

Using the Floquet-Bloch theory for $\mathcal{L}_{\delta}$ (cf. Section \ref{sec:FB_L}), we have, for all $T>0$,
\begin{equation*}
	u(x,T) = \frac{\sin (\sqrt{\mathcal{L}_{\delta}} T)}{\sqrt{\mathcal{L}_{\delta}}} g = \sum_{j\geq 1} \int_{Y^*} \frac{\sin( \omega_{j,\delta}(\alpha)T) }{\omega_{j,\delta}(\alpha)}\langle g(\cdot) , \Phi_{j,\delta}(\cdot,\alpha) \rangle_{L^2_{\delta}(\mathbb{R}^2)} \Phi_{j,\delta}(x,\alpha)\,d\alpha.
\end{equation*}
In view of the characterization \eqref{eq:H1FB} of $H^1_{\delta}(\mathbb{R}^2)$ norm and taking derivative in $t$ above, we can find a universal constant $C>0$ such that
\begin{equation*}
 	\|\partial_t u(\cdot,T) \|_{H^1_{\delta}(\R^2)} \leq  C\|g\|_{H^1_{\delta}(\mathbb{R}^2)}, \qquad \forall\,T\in \R. 
\end{equation*} 
From the energy conservation we also have $\|\nabla u(\cdot,T) \|_{L^2_{\delta}(\R^2)} \le \|g\|_{L^2_\delta(\R^2)}$. 

For the $L^2_\delta$ norm of $u(\cdot, T)$, the contribution of $j\ge 3$ can be controlled also by $C\|g\|_{L^2_\delta(\R^2)}$, and the contribution of $j\in \{1,2\}$ from the region $\{|\alpha-\alpha^*|\le r\}$ can be bounded by $C\delta^{-1/2} \|g\|_{L^2_\delta(\R^2)}$ thanks to \eqref{eq:Y*cut}. For $|\alpha-\alpha^*| \ge r$, where $\omega_{j,\delta}(\alpha)$ can approach zero, we bound $\sin(\omega_{j,\delta}(\alpha)T)/\omega_{j,\delta}(\alpha)$ by $CT$, and hence get
\begin{equation*}
\|u(\cdot,T) \|_{L^2_{\delta}(\mathbb{R}^2)} \leq  CT\sum_{j=1,2} \max_{|\alpha-\alpha^*|>r} |\langle g(\cdot) , \Phi_{j,\delta}(\cdot,\alpha) \rangle_{L^2_{\delta}(\mathbb{R}^2)} | + C\delta^{-\frac{1}{2}} \|g \|_{L^2_{\delta}(\mathbb{R}^2)}.
\end{equation*}
We conclude the above analysis by
\begin{equation}\label{eq:uOexest}
	\begin{aligned}
		& \|u(\cdot,T) \|_{L^2(D^\Lambda)} \leq  CT\sqrt{\delta}\sum_{j=1,2} \max_{|\alpha-\alpha^*|>r} |\langle g(\cdot) , \Phi_{j,\delta}(\cdot,\alpha) \rangle_{L^2_{\delta}(\mathbb{R}^2)} | + C \|g \|_{L^2_{\delta}(\mathbb{R}^2)} , \\
        &\|\nabla u(\cdot,T)\|_{L^2(D^\Lambda)} \le \sqrt{\delta}\|g\|_{L^2_\delta(\R^2)} , \\
		& 	\| \partial_t u(\cdot,T) \|_{H^1(D^{\Lambda})} \leq  C\sqrt{\delta}\|g \|_{H^1_\delta(\R^2)} .
	\end{aligned}
\end{equation}

In summary, we see that $u$ restricted to $\Omega^\Lambda$ satisfies the initial boundary value problem 
	\begin{equation}\label{prob:exterior}
	\left\{
	\begin{aligned}
		& (\partial_t^2 -\Delta ) v =  0 && \mathrm{in} \ \Omega^{\Lambda} \times [0,\infty) , \\
		& v(x,0) =0, \quad  \partial_t v (x,0) =g(x) && \mathrm{on} \ \Omega^{\Lambda} ,\\
		& v(x,t) =h(x,t) && \mathrm{on} \ \partial \Omega^{\Lambda}   \times [0,\infty),
	\end{aligned}\right.  
\end{equation}
with $h(x,t) = u|_{\partial D^{\Lambda}} (x,t)$. In view of \eqref{eq:uOexest}, the boundary datum $h$ satisfies for all $T>0$,
\begin{equation}\label{trace estimate} \left\{
	\begin{aligned}
		& 	\| h   (\cdot,T) \|_{H^{\frac{1}{2}}(\partial D^{\Lambda})} \leq CT\sqrt{\delta}\sum_{j=1,2} \max_{|\alpha-\alpha^*|>r} |\langle g(\cdot) , \Phi_{j,\delta}(\cdot,\alpha) \rangle_{L^2_{\delta}(\mathbb{R}^2)} | + C \|g \|_{L^2_{\delta}(\mathbb{R}^2)}, \\
		& 	\| \partial_t h   (\cdot,T) \|_{H^{\frac{1}{2}}(\partial D^{\Lambda})} \leq C \sqrt{\delta} \|g \|_{H^1_\delta(\R^2)} .
	\end{aligned}\right.
\end{equation}
By linearity of the wave equation \eqref{prob:exterior}, the solution $u$ is the sum of the solution to the problem \eqref{eq:wave_bd0} and the solution to the problem \eqref{eq:wave_in0} with $f=0$. Applying Lemma \ref{lem:LapExterior} and  Proposition \ref{prop: exterior nonzero} accordingly, we obtain that uniformly for all $T>0$,
\begin{equation*}
	\| u(\cdot, T) \|_{L^2(\Omega^{\Lambda})}  \leq C \|g\|_{L^2_\delta(\R^2)} + CT\sqrt{\delta}\sum_{j=1,2} \max_{|\alpha-\alpha^*|>r} |\langle g(\cdot) , \Phi_{j,\delta}(\cdot,\alpha) \rangle_{L^2_{\delta}(\mathbb{R}^2)} | +CT\sqrt{\delta}\|g\|_{H^1_\delta(\R^2)}.
\end{equation*}
Combining the above with the $L^2$ estimate \eqref{eq:uOexest} of $u$ in $D^\Lambda$, we see that the above estimate holds for $\|u\|_{L^2(\R^2)}$. 
The proof of Theorem \ref{thm:L2wave} is now complete.


\section*{Data availability statement}

Data sharing is not applicable to this article as no datasets were generated or analyzed during the current study.

\section*{Conflict of interest}

The authors declare that they have no conflict of interest.

\bibliographystyle{siam}
\bibliography{mybib}

\end{document}